\documentclass[11pt,reqno]{amsart}

\usepackage[T1]{fontenc}
\usepackage{lmodern}
\usepackage{amsmath,amssymb,amsthm,mathtools}
\usepackage[margin=1.05in]{geometry}
\usepackage{microtype}
\usepackage{enumitem}
\usepackage[colorlinks=true,linkcolor=blue,citecolor=blue,urlcolor=blue]{hyperref}
\numberwithin{equation}{section}
\newtheorem{theorem}{Theorem}[section]
\newtheorem{proposition}[theorem]{Proposition}
\newtheorem{lemma}[theorem]{Lemma}
\newtheorem{corollary}[theorem]{Corollary}
\theoremstyle{remark}
\newtheorem{remark}[theorem]{Remark}
\newcommand{\R}{\mathbb R}
\newcommand{\C}{\mathbb C}
\newcommand{\N}{\mathbb N}
\newcommand{\F}{\mathcal F}
\newcommand{\cT}{\mathcal T}
\newcommand{\Sch}{\mathcal S}
\newcommand{\Hp}{\C_+}
\newcommand{\Hm}{\C_-}
\newcommand{\ii}{\mathrm i}
\newcommand{\dd}{\,\mathrm d}

\newcommand{\diag}{\operatorname{diag}}
\newcommand{\sgn}{\operatorname{sgn}}
\newcommand{\unit}{\operatorname{unit}}
\newcommand{\fint}{\mathop{\int\!\!\!\!-}\nolimits}
\allowdisplaybreaks
\setlist[enumerate,1]{label=\textup{(\arabic*)},leftmargin=*,itemsep=4pt}
\renewcommand{\leq}{\leqslant}
\renewcommand{\geq}{\geqslant}

\title[Modified scattering for the cubic NLS]
{Failure of fixed-profile modified scattering and generalized modified scattering
for the 1D defocusing cubic NLS}
\author{Xi Chen}
\address{Department of Mathematics and Computer Science, University of Basel,
Spiegelgasse 1, 4051 Basel, Switzerland}
\email{xi01.chen@unibas.ch}
\subjclass[2020]{35Q55, 35B40, 34L40}
\keywords{1D defocusing cubic NLS, modified scattering, long-time asymptotics}
\date{}

\begin{document}
\begin{abstract}
We study the long-time behavior of the one-dimensional defocusing cubic nonlinear Schrödinger equation with initial data in $L^2(\mathbb R)$. We first prove that the standard fixed-profile modified-scattering ansatz fails at
the unweighted $L^2$ endpoint for the one-dimensional defocusing cubic
nonlinear Schr\"odinger equation.  More precisely, there exists a real-valued
datum
\[
u_*\in L^1(\R)\cap\bigcap_{m\in\N}H^m(\R),
 \qquad xu_*\notin L^2(\R).,
\]
for which the corrected Fourier profile has no strong $L^2$ limit. Secondly, for every $L^2$ datum, we prove generalized modified scattering: in the physical ray coordinates $x=2tz$, the rescaled solution admits a strong $L^2$ asymptotic representation whose amplitude is the square root of the conserved transmission density and whose unit complex phase is determined by the initial datum.
\end{abstract}
\maketitle
\tableofcontents

\section{Introduction}
\label{sec:introduction}

In this paper we study the defocusing cubic nonlinear Schr\"odinger
equation on the real line:
\begin{equation}\label{eq:nls}
\begin{cases}
 \ii\partial_t u+\partial_x^2u=2|u|^2u,\\
 u(0,x)=u_0(x),\qquad u_0\in L^2(\R),
\end{cases}
\qquad (t,x)\in\R\times\R,
\end{equation}
where the solution $u$ is complex valued.  The question is to describe
this solution as $t\to+\infty$ without assuming decay of the initial
datum beyond square integrability.  The amplitude and the phase play
different roles in this question.  The amplitude can be read from a
conserved scattering quantity.  The phase can retain the effect of
parts of the datum situated arbitrarily far from the origin.

\subsection{The defocusing cubic equation}

The global Cauchy theory gives a unique flow
$S(t):L^2(\R)\to L^2(\R)$, so that $u(t)=S(t)u_0$, with
$u\in C(\R;L^2(\R))$ and continuous dependence on $u_0$ on every
bounded time interval; see \cite[Theorem~1.1 and Corollary~1.2]{Tsutsumi1987}.
We write $\|f\|_p$ for the norm of $f$ in the Lebesgue space
$L^p(\R)$.  In particular,
\begin{equation}\label{eq:mass-conservation}
 \|S(t)u_0\|_2^2=\|u_0\|_2^2,\qquad t\in\R.
\end{equation}
The transformation
$u_\lambda(t,x)=\lambda u(\lambda^2t,\lambda x)$, where $\lambda>0$
is a change of spatial scale, preserves \eqref{eq:nls} and satisfies
$\|u_\lambda(0)\|_2^2=\lambda\|u_0\|_2^2$.
Ifrim and Tataru proved a global $L^6(\R_t\times\R_x)$ estimate
for this equation at arbitrary mass
\cite[Theorem~2]{IfrimTataru2023}.  This spacetime estimate does not
specify an asymptotic phase, but will identify the modulus of any
possible strong Fourier profile.

\subsection{Long-range scattering}

We use the unitary Fourier transform
\begin{equation}\label{eq:fourier}
 \widehat f(\xi)=\frac{1}{\sqrt{2\pi}}
       \int_\R e^{-\ii x\xi}f(x)\dd x,
 \qquad F_{u_0}(t,\xi)=e^{\ii t\xi^2}\widehat{S(t)u_0}(\xi).
\end{equation}
Here $\xi$ is Fourier frequency, and $F_{u_0}$ removes the exact free
Schr\"odinger evolution from the solution.  The integral defining
$\widehat f$ extends by unitarity from integrable functions to $L^2$.
For $s\in\R$, let $H^s(\R)$ denote the Sobolev space with norm
$\|f\|_{H^s}^2=\int_\R(1+\xi^2)^s|\widehat f(\xi)|^2\dd\xi$.
The exponent $s$ measures the number of derivatives in this norm.
We write $\Sch(\R)$ for the Schwartz class of smooth functions
whose derivatives decay faster than every inverse power of $|x|$.
For a free solution and the ray variable $z=x/(2t)$, the exact
Fresnel formula proved in \eqref{eq:fresnel-identity} implies
\[
 \sqrt{2t}\,e^{-\ii tz^2}(e^{\ii t\partial_x^2}f)(2tz)
       =e^{-\ii\pi/4}\widehat f(z)+o_{L^2_z}(1),
       \qquad f\in L^2(\R).
\]
The remainder tends to zero in $L^2$ as $t\to+\infty$.
The variable $z$ therefore records the Fourier frequency
which travels along the ray $x=2tz$.  The cubic term at this scale
has size $t^{-1}|\widehat f(z)|^2$ relative to the solution.
Since $\int_1^t s^{-1}\dd s=\log t$, its phase contribution does
not converge.  This is the origin of the logarithmic correction.

For the normalization \eqref{eq:nls}, we say that classical modified
scattering holds if there exists a fixed $V\in L^2(\R)$ such that
\begin{equation}\label{eq:fixed}
 \lim_{t\to+\infty}
 \bigl\|F_{u_0}(t)-V e^{-\ii|V|^2\log t}\bigr\|_2=0.
\end{equation}
Equivalently, the exact free group reconstructs the asymptotic solution:
\begin{equation}\label{eq:fixed-physical}
 \bigl\|S(t)u_0-e^{\ii t\partial_x^2}\F^{-1}
       (V e^{-\ii|V|^2\log t})\bigr\|_2\longrightarrow0.
\end{equation}
The equivalence follows directly from Fourier unitarity.

Zakharov and Manakov obtained the logarithmic correction by inverse
scattering \cite{ZakharovManakov}; the Hamiltonian and spectral
framework is developed in \cite{FaddeevTakhtajan}.  Ozawa constructed
modified wave operators for small weighted final states
\cite{Ozawa1991}.  The Cauchy problem for small localized data was
treated by Hayashi and Naumkin \cite{HayashiNaumkin1998}, with further
wave-operator results in \cite{HayashiNaumkin2006,LindbladSoffer}.
Fourier and wave-packet proofs give related descriptions
\cite{KatoPusateri,IfrimTataru2015}.  For arbitrary data in
$H^{1,1}(\R)=\{f\in H^1(\R):xf\in L^2(\R)\}$, Deift and Zhou
proved a uniform physical-space asymptotic expansion
\cite[\S1, (1.2)]{DeiftZhou2003}.  Their method develops the nonlinear
steepest descent introduced in \cite{DeiftZhou1993}; a
$\bar\partial$ treatment appears in \cite{DiengMcLaughlin}.
Recent work considers large final states with Fourier regularity
\cite{KawamotoMizutani} and completeness under suitable localization
and decay assumptions \cite{GeorgievOzawa}.

These results explain the phase when the datum has the stated
localization or regularity.  At the unweighted $L^2$ level, a small
perturbation can be moved to a very large distance without changing
its norm.  Its interaction with the original datum changes the
scattering phase by an amount comparable to the perturbation size
times the logarithm of the distance.  This effect is compatible with
convergence of the amplitude, and prevents \eqref{eq:fixed} for the
examples constructed below.

\subsection{Statement of the results}

We first specify the conserved amplitude.  For a potential $u$, let
$r_u(k)$ be its reflection coefficient for the Zakharov--Shabat
system
\[
 \partial_xY=(-\ii kJ+Q(u))Y,\qquad
 J=\begin{pmatrix}1&0\\0&-1\end{pmatrix},\quad
 Q(u)=\begin{pmatrix}0&u\\\bar u&0\end{pmatrix}.
\]
The real number $k$ is the spectral parameter of this first-order
system, and $Y$ is a two-component solution.  We choose the sign of
$r_u$ by the inverse-scattering jump \eqref{eq:jump}; for
integrable potentials it is $-b_u/a_u$, where $a_u,b_u$ are the
transmission and reflection entries of the scattering matrix defined
in Appendix~\ref{app:conventions}.  Section~\ref{sec:spectral}
gives its half-line definition for every $u\in L^2$.
Set
\begin{equation}\label{eq:rho}
 \ell_u(k)=\log\frac{1}{1-|r_u(k)|^2},\qquad
 \rho_u(z)=\frac{1}{2\pi}\ell_u(-z/2).
\end{equation}
The function $\ell_u$ is the nonnegative logarithmic transmission
density in spectral coordinates.  The change $k=-z/2$ places it in
the ray coordinates of the solution.  The trace formula
\cite[Proposition~4.10]{BessonovGubkin} gives
\begin{equation}\label{eq:mass-rho}
 |r_u|<1\ \hbox{a.e.},\qquad
 \rho_u\geq0,\qquad \int_\R\rho_u(z)\dd z=\|u\|_2^2.
\end{equation}
For $t>0$ define the rescaled physical solution
\begin{equation}\label{eq:ray-profile}
 G_{u_0}(t,z)=\sqrt{2t}\,e^{-\ii tz^2}S(t)u_0(2tz).
\end{equation}
The factor $\sqrt{2t}$ preserves the $L^2$ norm under $x=2tz$;
the exponential removes the quadratic free phase.

\begin{theorem}[Failure of a fixed modified profile]\label{thm:failure}
For each $M>0$, where $M$ denotes the prescribed $L^2$ norm,
there is a real-valued initial datum $u_*$ with the following properties:
\begin{enumerate}
\item Its regularity and localization satisfy
\begin{equation}\label{eq:counterexample-class}
 u_*\in L^1(\R)\cap\bigcap_{m\in\N}H^m(\R),
 \qquad \|u_*\|_2=M,\qquad xu_*\notin L^2(\R).
\end{equation}
\item The family $e^{\ii\rho_{u_*}\log t}F_{u_*}(t)$ has no strong
limit in $L^2(\R)$ as $t\to+\infty$.  Consequently no
$V\in L^2(\R)$ satisfies \eqref{eq:fixed} for $u_*$.
\end{enumerate}
\end{theorem}

Our positive result uses a phase which may depend on time beyond a
single logarithmic rotation.  For $w\in\C$, write
$\unit(w)=w/|w|$ if $w\neq0$ and $\unit(0)=1$.  This convention
allows a unit phase to be defined even where the amplitude vanishes.

\begin{theorem}[Generalized modified scattering]\label{thm:generalized}
Given $u_0\in L^2(\R)$, the construction in
Subsection~\ref{sec:phase-construction} determines a jointly measurable
function $\Phi_{u_0}:(0,+\infty)\times\R\to\C$ with
$|\Phi_{u_0}|=1$ and the following properties:
\begin{enumerate}
\item The physical solution has the strong asymptotics
\begin{equation}\label{eq:generalized}
 \bigl\|G_{u_0}(t)-\sqrt{\rho_{u_0}}\,\Phi_{u_0}(t)\bigr\|_2
       \longrightarrow0.
\end{equation}
Its amplitude satisfies
\begin{equation}\label{eq:amplitude}
 \bigl\||G_{u_0}(t)|^2-\rho_{u_0}\bigr\|_1\longrightarrow0,
 \qquad
 \bigl\||G_{u_0}(t)|-\sqrt{\rho_{u_0}}\bigr\|_2\longrightarrow0.
\end{equation}
\item If $\widetilde\Phi_{u_0}$ is obtained from a different
admissible regularization of the same initial datum, then
\[
 \int_\R|\Phi_{u_0}(t,z)-\widetilde\Phi_{u_0}(t,z)|^2
             \rho_{u_0}(z)\dd z\longrightarrow0.
\]
\end{enumerate}
\end{theorem}

The theorem places no upper bound on the mass and requires neither
a spatial weight nor a constant $\kappa>0$ with
$1-|r_{u_0}|^2\geq\kappa$.  The last inequality, when available,
will be called a uniform reflection gap.  It gives a coercive
inverse problem, but is unnecessary for the original datum.

\begin{theorem}[The logarithmic phase for localized data]\label{thm:localized}
Let $u_0\in H^{1,1}(\R)$, and let $W_{u_0}$ denote the classical
ray coefficient given by \eqref{eq:classical-ray-coefficient}.
Then the following hold:
\begin{enumerate}
\item The coefficient satisfies $|W_{u_0}|^2=\rho_{u_0}$, and the
phase in Theorem~\ref{thm:generalized} satisfies
\begin{equation}\label{eq:phase-agreement}
 \bigl\|\sqrt{\rho_{u_0}}\,
  [\Phi_{u_0}(t)-e^{-\ii\rho_{u_0}\log t}\unit(W_{u_0})]\bigr\|_2
       \longrightarrow0.
\end{equation}
\item If $u_0$ belongs to the Schwartz class $\Sch(\R)$, then
\eqref{eq:fixed} holds with $V_{u_0}=e^{\ii\pi/4}W_{u_0}$.
\end{enumerate}
\end{theorem}

\subsection{Comments on the results}

\noindent\textit{The asymptotic phase.}
The expression ``generalized modified scattering'' refers here to
\eqref{eq:generalized} with the specified construction from the initial
datum.  Each approximation is obtained by restricting the initial
reflection to a finite spectral interval and solving its inverse
problem.  The interval and the approximation index are chosen by a
finite procedure which does not inspect $S(t)u_0$.  This is essential:
choosing $\Phi=\unit(G)$ would merely rephrase amplitude convergence.
The theorem does not provide an autonomous equation for $\Phi$.

The $L^2$ spectral framework was developed by Sylvester and
Winebrenner \cite{SylvesterWinebrenner} and Denisov
\cite{Denisov2006}; the continuity theorem used here is due to
Bessonov and Gubkin \cite{BessonovGubkin}.  Related spectral
conservation laws and scattering results appear in
\cite{Bessonov2020,KillipVisanZhang}.  Pointwise convergence of
truncated nonlinear Fourier transforms is a separate question
\cite{Poltoratski2024,DenisovNCC,DenisovPointwise}; our proof uses
convergence in measure and strong norm compactness.

The half-line scattering functions preserve information about the
potential on either side of a spatial point.  Full-line reflection
alone need not determine an arbitrary $L^2$ potential
\cite[Proposition~4.12]{BessonovGubkin}.  Accordingly, the phase is
attached to the initial potential and its half-line factors.  The
compactness argument below retains these factors when removing the
uniform reflection gap.

\smallskip
\noindent\textit{Fourier and physical coordinates.}
The two profiles \eqref{eq:fourier} and \eqref{eq:ray-profile} are
related by the exact identity
\begin{equation}\label{eq:fresnel-intro}
 e^{\ii\pi/4}G_{u_0}(t)
    =e^{-\ii(4t)^{-1}\partial_z^2}F_{u_0}(t).
\end{equation}
We prove it in \eqref{eq:fresnel-identity}.  The operator on the right
converges strongly to the identity on each fixed $L^2$ function.
The function $F_{u_0}(t)$ also varies with time, so this fact alone
does not allow the operator to be omitted.  For a smooth logarithmic
profile we estimate its derivatives and obtain the required error.
For general $L^2$ data we retain the exact distinction between the
two coordinates.

\smallskip
\noindent\textit{The role of localization.}
The datum in Theorem~\ref{thm:failure} has arbitrarily high Sobolev
regularity and is integrable.  The missing condition is a spatial
moment.  A sequence of small disjoint translates can have finite
$L^1$ and Sobolev norms while $\|xu_*\|_2=+\infty$.  Each new
translate changes the eventual phase by a definite amount, although
its effect at previously chosen finite times is small.  Thus the
failure concerns a single datum and arbitrarily large times.

\subsection{Strategy of the proof}

We first describe the argument for the strong amplitude.  Along
$x=2tz$ the oscillatory jump has phase $2kx+4k^2t$, whose stationary
point is $k=-z/2$.  The relevant spectral distance from this point is
$t^{-1/2}$.  On this scale the reflection and its logarithmic density
are approximately constant at almost every spectral center.  The
constant-reflection inverse problem is explicit.  After returning
to the ray normalization, its amplitude has squared modulus $\rho(z)$.
To pass from this model to arbitrary
$L^2$ data, we first prove compactness of inverse potentials in the
topology given by reflection convergence in measure and density
convergence in $L^1$.  Compactness identifies the local limiting
amplitude, and conservation of the full mass then turns this local
convergence into the strong norm statements.

The model also has a unit phase.  The remaining factors vary on
spectral scales larger than $t^{-1/2}$ and supply an additional phase.
We compare these original factors with the inverse problem for a
finite interval centered at $-z/2$.  Weighted Hardy estimates control
the factors at height comparable to $t^{-1/2}$, while an averaged
rectangle estimate transfers information from a fixed positive
height.  The comparison is relative to the original factors; this
allows reflections with no uniform gap.

For each fixed interval width, the comparison is uniform over all
smooth approximations and all interval centers.  A quantitative
finite-time stability estimate then determines how accurately the
initial datum must be approximated at a given time.  The interval
width is chosen by comparing the finitely many available candidate
phases.  This order of choices yields a phase determined by the
initial datum and avoids interchanging a large-time limit with an
uncontrolled approximation limit.  The localized result follows by
comparison with the Deift--Zhou expansion.

For the counterexample, we compute the scattering matrix of two
separated potentials.  Translating the small added potential by $X$
introduces a factor $e^{-2\ii kX}$ in its reflection entry.  The
linear contribution to the phase is a one-sided singular integral;
its leading size is $\varepsilon\log X$ for an insertion of size
$\varepsilon$.  Its higher harmonics contribute only
$O(\varepsilon^2\log X)$.  We choose $\varepsilon$ small and then
$X$ exponentially large, giving a definite phase change at a fixed
small cost in the initial norms.  After selecting the time at which
this change is observed, every later insertion is chosen small
enough to preserve all earlier observations.  The resulting sequence
converges to one smooth datum whose corrected Fourier profile is
not Cauchy.

Section~\ref{sec:spectral} develops the scattering transform and
proves compactness of inverse potentials.  Section~\ref{sec:classical}
identifies the classical coefficient and the necessary modulus of a
Fourier limit.  Section~\ref{sec:ray-asymptotics} treats the
constant-reflection model and establishes strong ray asymptotics.
Section~\ref{sec:asymptotic-phase} constructs the phase from the initial
datum and proves its agreement with the localized logarithmic phase.
Section~\ref{sec:failure-proof} computes the effect of a separated
potential and constructs one fixed datum for which the classical
modified profile fails to converge.
The appendices fix the scattering conventions and collect the
analytic tools, with references for the standard estimates.

\subsection{Notation}

Unless a domain is displayed, Lebesgue and Sobolev spaces are over
$\R$.  We use the Euclidean norm of the entries for matrices;
$\|\cdot\|_{\mathrm{op}}$ denotes the induced operator norm on $\C^2$.
Compactness in $L^2$ always refers to its norm topology.
The notation $A\lesssim B$ means $A\leq CB$ for a constant whose
allowed dependence is specified locally, and $A\sim B$ means that
both inequalities hold.  The upper and lower spectral half-planes
are $\Hp=\{k\in\C:\Im k>0\}$ and
$\Hm=\{k\in\C:\Im k<0\}$.  Boundary values of analytic functions
are understood almost everywhere.  A subscript on a scattering
quantity records its initial potential and is omitted when that
potential is fixed.  All asymptotic statements concern $t\to+\infty$.

\section{The scattering transform and compactness of inverse potentials}
\label{sec:spectral}

The inverse-scattering argument needs more than continuity of the
direct transform.  We will restrict reflections to finite intervals
and vary the interval center.  To control the associated potentials
uniformly, we must show that a compact family of spectral data has
compact inverse images in $L^2$.  The half-line inverse theorem is
the starting point.  The proof below keeps both half-line factors
and the full matrix jump, so that it does not assume injectivity of
the full-line reflection map.

\subsection{Half-line functions and the full-line jump}

Recall the matrices $J$ and $Q$ introduced in the introduction.  The
spatial scattering equation is
\begin{equation}\label{eq:zs}
 \partial_xY=(-\ii kJ+Q(u))Y.
\end{equation}
The matrices $M_+(t,x,k)$ and $M_-(t,x,k)$ denote the inverse factors
analytic in $\Hp$ and $\Hm$, respectively.  They are normalized to
the identity matrix $I$ at imaginary infinity.  At a spatial point
$x$ these factors are constructed from the restrictions of $S(t)u$
to the right and left of $x$.  Their boundary values satisfy
\begin{equation}\label{eq:jump}
 M_+(t,x,k)=M_-(t,x,k)V_r(t,x,k),\qquad
 V_r(t,x,k)=
 \begin{pmatrix}
  1-|r(k)|^2&\overline{r(k)}e^{-\ii(2kx+4k^2t)}\\
  -r(k)e^{\ii(2kx+4k^2t)}&1
 \end{pmatrix}.
\end{equation}
Here $r=r_u$ refers to the initial potential.  The expression
$2kx+4k^2t$ is the combined spatial and temporal scattering phase.
Direct calculation gives $\det V_r=1$ and
$\|V_r\|_{L^\infty_k}\leq\sqrt3$, with the Euclidean matrix norm.
The time dependence and the sign of $r$ are checked in
Appendix~\ref{app:conventions}.

We describe the half-line functions precisely.  Given
$v\in L^2([0,+\infty))$, consider the Dirac operator
\[
 D_v=J_0\partial_y+Q_G(v),\qquad
 J_0=\begin{pmatrix}0&-1\\1&0\end{pmatrix},\qquad
 Q_G(v)=\begin{pmatrix}\Im v&\Re v\\\Re v&-\Im v\end{pmatrix}.
\]
It acts on locally absolutely continuous two-component functions
$Y\in L^2([0,+\infty);\C^2)$ for which $D_vY\in L^2$, with
$Y_2(0)=0$.  Let $N_v(y,k)$ solve
$D_vN_v=kN_v$, $N_v(0,k)=I$.  Its Weyl function and Schur function
are, respectively,
\[
 m_v(k)=\lim_{y\to+\infty}\frac{(N_v)_{22}(y,k)}{(N_v)_{21}(y,k)},
 \qquad f_v(k)=\frac{m_v(k)-\ii}{m_v(k)+\ii},\qquad k\in\Hp.
\]
The Weyl function encodes the square-integrable solution at
infinity; the Cayley transform $f_v$ takes values in the closed unit
disc.  Such bounded analytic functions are called Schur functions.
Their almost-everywhere boundary values will use the same notation.

The following facts are the half-line input
\cite[Theorem~1.1, Lemmas~4.4 and~4.6]{BessonovGubkin}.  The map
$v\mapsto f_v$ is a bijection onto the Schur functions for which
$-\log(1-|f_v|^2)$ is integrable, and
\begin{equation}\label{eq:halfline-trace}
 \int_0^{+\infty}|v(y)|^2\dd y
     =\frac1\pi\int_\R-\log(1-|f_v(k)|^2)\dd k.
\end{equation}
Strong convergence of potentials is equivalent to boundary
convergence in measure together with convergence of the entropy
integrals.  For a bounded sequence of potentials, weak convergence
is equivalent to locally uniform convergence of their Schur
functions in $\Hp$.  These are established spectral theorems; we
will use them with the signs below.

At a fixed split $x$, define the right and reflected-left potentials
on $y\geq0$ by
\[
 v_+(y)=\overline{u(x+y)},\qquad v_-(y)=-u(x-y),
 \qquad f=-f_{v_+},\qquad g=-f_{v_-}.
\]
The functions $f,g$ thus encode the two spatial halves.  Their
normalized outer transmission functions are denoted by $A,B$.
An outer function here means the zero-free analytic function whose
logarithmic modulus is extended by the Poisson integral, with its
phase fixed at imaginary infinity.  More explicitly,
\[
 A(w)=\exp\left\{\frac1{2\pi\ii}
       \int_\R\frac{-\log(1-|f(s)|^2)}{s-w}\dd s\right\},
 \qquad w\in\Hp,
\]
and $B$ is given by replacing $f$ with $g$.  In particular,
$|A|^{-2}=1-|f|^2$ and $|B|^{-2}=1-|g|^2$ on the boundary.
The full transmission $a$ and the inverse matrix are
\begin{equation}\label{eq:half-splice}
 a=AB(1-fg),\qquad
 M_+=\begin{pmatrix}B/a&Af\\Bg/a&A\end{pmatrix},\qquad
 M_- =\sigma_1\overline{M_+(\bar k)}\sigma_1,
 \quad \sigma_1=\begin{pmatrix}0&1\\1&0\end{pmatrix}.
\end{equation}
The last expression analytically reflects the upper factor to the
lower half-plane.  Appendix~\ref{app:conventions} derives these
identities from the two decaying solutions.

With $\ell=\ell_u$ from \eqref{eq:rho}, the full outer formula and
trace identity are
\begin{equation}\label{eq:outer-a}
 a(w)=\exp\left\{\frac1{2\pi\ii}
                \int_\R\frac{\ell(s)}{s-w}\dd s\right\},
 \quad w\in\Hp,
 \qquad |a(k)|^2=e^{\ell(k)},
\end{equation}
\begin{equation}\label{eq:trace}
 \|u\|_2^2=\frac1\pi\int_\R\ell(k)\dd k.
\end{equation}
These follow from the full-line half-line identities in
\cite[\S4.4, Proposition~4.10]{BessonovGubkin} in our convention.
The logarithms of $a,A,B$ all tend to zero at positive imaginary
infinity.  In particular their reciprocals are bounded by one.

To formulate compactness, let $X$ consist of the measurable
reflections $r$ with $|r|<1$ almost everywhere and
$\ell_r=-\log(1-|r|^2)\in L^1$.  Convergence in $X$ means local
convergence in measure of $r$ and $L^1$ convergence of $\ell_r$.
One metric with this convergence criterion is
\[
 d_X(r,s)=\|\ell_r-\ell_s\|_1+
   \sum_{m=1}^{+\infty}\frac{2^{-m}}{2m}
          \int_{-m}^m\min\{1,|r(k)-s(k)|\}\dd k.
\]
The integer $m$ specifies a bounded spectral interval; the summable
coefficients give a single metric for all such intervals.  The
$L^1$ term retains the entire mass, including its spectral tails.

\begin{lemma}[Continuity and evolution of the scattering data]\label{lem:direct}
The scattering quantities have the following properties.
\begin{enumerate}
\item If $v_j\to v$ in $L^2(\R)$, then $r_{v_j}\to r_v$ in $X$,
$\log a_{v_j}\to\log a_v$ locally uniformly in $\Hp$, and
$a_{v_j}^{-1}\to a_v^{-1}$ locally in measure on $\R$.
\item At every fixed $t\in\R$,
\[
 a_{S(t)v}=a_v,
 \qquad r_{S(t)v}(k)=e^{4\ii k^2t}r_v(k)\quad\hbox{a.e. }k.
\]
In particular $\rho_{S(t)v}=\rho_v$.
\end{enumerate}
\end{lemma}

\begin{proof}
The direct continuity of the full-line reflection follows from
\cite[Proposition~4.11]{BessonovGubkin}.  We explain why it also
gives the asserted logarithmic convergence.  From any subsequence
we can extract one for which $r_{v_j}\to r_v$ almost everywhere.
Since $|r_v|<1$ almost everywhere, the nonnegative functions
$\ell_{v_j}$ then converge pointwise to $\ell_v$.  Their integrals
converge by \eqref{eq:trace}.  Scheff\'e's lemma implies
$\|\ell_{v_j}-\ell_v\|_1\to0$.  Every subsequence has such a further
subsequence, so the full sequence has this limit.
For $w=\alpha+\ii\beta$, $\beta>0$, the outer formula gives
\[
 |\log a_{v_j}(w)-\log a_v(w)|
 \leq\frac{1}{2\pi}\int_\R
      \frac{|\ell_{v_j}(s)-\ell_v(s)|}{|s-w|}\dd s
 \leq\frac{\|\ell_{v_j}-\ell_v\|_1}{2\pi\beta}.
\]
The boundary logarithms have real part $\ell/2$ and imaginary
part its Hilbert transform, up to the sign fixed in the outer
formula.  The weak $(1,1)$ estimate for the Hilbert transform
therefore gives convergence of the imaginary parts in measure.
The corresponding boundary reciprocals converge in measure and
remain bounded by one.  The analytic facts used here are collected
in Appendix~\ref{app:analytic}.

For Schwartz initial data the Lax evolution gives the two identities
in the second item; see \cite[Theorem~2.3 and (2-9)]{BessonovDenisov}
and Appendix~\ref{app:conventions} for the change of spectral
parameter.  Approximate $v$ in $L^2$ by Schwartz potentials and fix
$t$.  The Cauchy theory gives convergence of their solutions at
this time.  The first item then passes the reflection identity in
measure and the transmission identity locally uniformly in $\Hp$.
This proves the statement at each fixed time; taking absolute values
gives conservation of $\rho$.
\end{proof}

\subsection{Compactness without a uniform reflection gap}

\begin{lemma}[Compactness of inverse potentials]\label{lem:proper}
The space $X$ has the following inverse-scattering properties.
\begin{enumerate}
\item If $r_{u_n}\to r$ in $X$, every subsequence of $(u_n)$
contains a further subsequence converging strongly in $L^2$ to a
potential $u$ with $r_u=r$.
\item Every $r\in X$ is the reflection of an $L^2$ potential.
Consequently the inverse image of a compact subset of $X$ is compact
in $L^2$.
\end{enumerate}
\end{lemma}

The conclusion does not require uniqueness of the inverse.  The
essential issue is to exclude loss of mass in a weak limit.  The
full jump identifies the boundary moduli of the limiting half-line
transmissions; the outer normalization then identifies their
interior values.  The trace identity will finish the argument.

\begin{proof}
The trace identity bounds the potentials and determines the limit
of their squared norms.  Pass to a subsequence with
$u_n\rightharpoonup u$ in $L^2$.  At the origin, let $f_n,g_n$ be
the two half-line Schur functions, and $A_n,B_n$ their outer
transmissions.  Weak half-line continuity yields locally uniform
limits $f_n\to f=f_u$, $g_n\to g=g_u$ in $\Hp$.
Let $a$ be the outer function determined by $\ell_r$.
Then $a_n\to a$ locally uniformly by the estimate in the preceding
proof.  If $M=\sup_n\|u_n\|_2$, the half-line outer integrals give
\[
 |\log A_n(w)|+|\log B_n(w)|\leq\frac{M^2}{2\Im w}.
\]
Thus, after another subsequence, their logarithms converge locally
uniformly to those of $A_*,B_*$.  These limits have the same
normalization at imaginary infinity and obey
\begin{equation}\label{eq:limit-splice}
 a=A_*B_*(1-fg).
\end{equation}
Moreover, $\bar a_n/a_n\to\bar a/a$ in measure on each bounded
boundary interval by the Hilbert-transform estimate just used.

We now pass the matrix jump to the limit.  Introduce the probability
measure $\dd\mu(k)=\dd k/[\pi(1+k^2)]$, which is harmonic measure
at $\ii$.  For an analytic function in a half-plane,
$H^p(\mu)$ denotes the Hardy space obtained by the Cayley transform
to the unit disc, with its boundary norm in $L^p(\mu)$.
Set
\[
 K_{n,+}=M_{n,+}/a_n,
 \qquad K_{n,-}=M_{n,-}/\overline{a_n(\bar k)},
 \qquad H_n=\Re\frac{1+f_ng_n}{1-f_ng_n}.
\]
Dividing by transmission places the matrix entries in a common
Hardy space even when $a_n$ is unbounded.  The function $H_n$ is
positive and harmonic.  The bounded analytic vector
$(B_n^{-1},g_n)$ has boundary squared norm one, hence
$|B_n^{-1}|^2+|g_n|^2\leq1$ inside $\Hp$.  Therefore
\[
 \left|\frac{A_n}{a_n}\right|^2
 =\frac{|B_n^{-1}|^2}{|1-f_ng_n|^2}
 \leq\frac{1-|g_n|^2}{|1-f_ng_n|^2}
 \leq\frac{1-|f_ng_n|^2}{|1-f_ng_n|^2}=H_n.
\]
Interchanging $f_n,g_n$ gives $|B_n/a_n|^2\leq H_n$.
Since $|f_n|,|g_n|,|a_n^{-1}|\leq1$, the four entries of
$K_{n,+}$, and by reflection those of $K_{n,-}$, have the same
majorant.  The harmonic-majorant characterization of Hardy spaces
and the half-line normalization give
\[
 \|K_{n,\pm}\|_{H^2(\mu)}^2\leq4H_n(\ii)
 \leq8e^{M^2}.
\]
For the last bound, \eqref{eq:half-splice} implies
$|1-f_ng_n|^{-1}=|A_nB_n/a_n|\leq e^{M^2/2}$ at $\ii$,
while $|(1+f_ng_n)/(1-f_ng_n)|\leq2|1-f_ng_n|^{-1}$;
the displayed constant allows a weaker exponent.

The normalized jump is
\[
 K_{n,+}=K_{n,-}W_n,
 \qquad W_n=(\bar a_n/a_n)V_{r_{u_n}}.
\]
Its multiplier is uniformly bounded and converges to
$W=(\bar a/a)V_r$ in $L^p(\mu)$ for every finite $p$.
To see the global convergence, first use convergence in measure
on bounded intervals and the uniform bound, and then make the
$\mu$-measure of the complement small.  Weak Hardy compactness
gives limits determined by $A_*,B_*,f,g$.
For a boundary test $h\in L^2(\mu)$,
$h(W_n-W)\to0$ in $L^2(\mu)$: it holds for bounded $h$ by the
$L^2$ convergence of $W_n$, and extends to all $h$ by approximation
and the uniform bound on $W_n$.  The weak jump therefore passes to
the limit.  Restoring the outer factors gives
$M_{*,+}=M_{*,-}V_r$.

We spell out how this jump fixes the moduli.  Its second column
gives
\[
 A_*f=\bar A_*\bar r+\frac{\bar B_*\bar g}{\bar a},
 \qquad
 A_*=\bar A_*\bar f\bar r+\frac{\bar B_*}{\bar a}.
\]
Multiply the first relation by $\bar f$ and subtract it from the
second.  Using $\bar a=\bar A_*\bar B_*(1-\bar f\bar g)$ yields
\[
 A_*(1-|f|^2)=\frac{\bar B_*}{\bar a}(1-\bar f\bar g)
             =\frac1{\bar A_*}.
\]
Thus $|A_*|^2(1-|f|^2)=1$.  Substitution into the first of the
two column relations and complex conjugation give
\[
 r=\frac{\bar A_*}{A_*}\frac{\bar f-g}{1-fg},
 \qquad
 1-|r|^2=\frac{(1-|f|^2)(1-|g|^2)}{|1-fg|^2}.
\]
Combining this with \eqref{eq:limit-splice} gives
\[
 1-|r|^2=\frac{|B_*|^2(1-|g|^2)}{|a|^2}.
\]
Since $1-|r|^2=|a|^{-2}$, these are exactly the boundary moduli of
the half-line outer functions $A_u,B_u$ of the weak limit $u$.

Boundary moduli alone would allow a reciprocal singular inner
factor.  We exclude this possibility before using the trace
identity.  Fix $w\in\Hp$.  For each positive integer $m$,
the bounded analytic functions $f_n^m$ converge weakly in the
Hardy space with harmonic measure at $w$.  Lower semicontinuity
of the norm, followed by the outer formula, implies for every
positive integer $N$ that
\[
 2\Re\log A_*(w)\geq
 \sum_{m=1}^N\frac1m\int_\R|f(k)|^{2m}
       \frac{\Im w}{\pi|k-w|^2}\dd k.
\]
Here the finite sum is bounded by
$\int-\log(1-|f_n|^2)\dd\mu_w=2\Re\log A_n(w)$
before passage to the limit.  Increasing $N$ and using
$\sum_{m\geq1}s^m/m=-\log(1-s)$ gives
\[
 \Re\log(A_*/A_u)\geq0,
 \qquad \Re\log(B_*/B_u)\geq0.
\]
Both quotients belong to the Smirnov class.  Indeed
$A_*/A_u=(A_*/a)(a/A_u)$, with the first factor a Hardy limit
entry and the second a quotient of outer functions.  For $B_*$,
use $B_*/B_u=(B_*/a^2)(a^2/B_u)$ and the $(1,1)$ entry of the
same normalized matrix.  The Smirnov class and its maximum principle
are recalled in Appendix~\ref{app:analytic}.
Each quotient has boundary modulus one and interior modulus at
least one.  The maximum principle gives interior modulus at most
one; the open mapping theorem makes the quotient constant.
Normalization at imaginary infinity makes that constant one.
Hence $A_*=A_u$, $B_*=B_u$, and $a=a_u$.
The full jump now identifies $r_u=r$.  Consequently
\[
 \lim_n\|u_n\|_2^2
 =\frac1\pi\int_\R\ell_r\dd k
 =\|u\|_2^2.
\]
Together with weak convergence this proves strong convergence.

For an arbitrary $r\in X$, first restrict $r$ to a bounded interval,
then replace it by $\min\{|r|,1-\varepsilon\}\unit(r)$, and
finally mollify.  Choosing the interval radius to increase and
$\varepsilon$ to decrease gives smooth compactly supported
reflections $r_j$ with a strict gap and $r_j\to r$ in $X$.
For clarity, truncation and clipping converge by dominated
convergence of $\ell_r$.  For a fixed clipped reflection,
mollification converges in $L^2$ and preserves a slightly weaker
gap; on that disc $-\log(1-|r|^2)$ is Lipschitz in $|r|^2$.
This proves the claimed entropy convergence as well.
The weighted inverse-scattering theorem provides potentials for
these smooth data
\cite[\S3, pp.~1058--1059]{DeiftZhou2003},
\cite[Theorem~2.4]{Zhou1998}.
The first part of the lemma supplies a strong subsequential limit
with reflection $r$.  Finally, any sequence in the inverse image
of a compact subset of $X$ has convergent reflections along a
subsequence, and then strongly convergent potentials by the first
part.  The inverse image is closed by Lemma~\ref{lem:direct}, which
proves compactness.
\end{proof}

\subsection{Identification under a uniform reflection gap}

When a reflection satisfies $1-|r|^2\geq\kappa>0$, its inverse
problem is coercive on the line Hardy space.  The next proposition
identifies this inverse with the original factors of the potential.
This is needed for the restricted intervals used later, rather
than as a hypothesis on the initial datum.

\begin{proposition}[The inverse factor with a uniform gap]\label{prop:strict-gap}
Let $u\in L^2$ and assume $1-|r_u|^2\geq\kappa>0$ almost
everywhere.  Then:
\begin{enumerate}
\item At each fixed $(t,x)$, its original factors are the
unique solution of \eqref{eq:jump} with
\[
 M_+-I\in H^2(\Hp),\qquad M_--I\in H^2(\Hm).
\]
Here $H^2(\C_\pm)$ are the unweighted line Hardy spaces.  Thus
their values at every nonreal spectral point are determined by $r_u$.
\end{enumerate}
\end{proposition}

\begin{proof}
The outer formula gives $|a|\leq\kappa^{-1/2}$ in $\Hp$.
The harmonic majorant from the compactness proof consequently puts
the original entries of $M_\pm$ in $H^2(\mu)$.
We construct the line Hardy solution and then compare it with these
entries.  Let $W=V_r-I$ and let $P_-$ be the orthogonal projection
from $L^2(\R)$ to boundary values of $H^2(\Hm)$.
Since $|r|^2\leq\ell_r$, we have $W\in L^2$.  For the Hermitian
part of the jump,
\[
 \frac{V_r+V_r^*}{2}=\diag(1-|r|^2,1)\geq\kappa I.
\]
On row vectors $F\in H^2(\Hm)$, define
$T_{V_r}F=P_-(FV_r)$.  Orthogonality of $P_-$ gives
\[
 \Re\langle T_{V_r}F,F\rangle
 =\Re\langle FV_r,F\rangle\geq\kappa\|F\|_2^2.
\]
The same inequality holds for the adjoint.  Thus the operator and
its adjoint have trivial kernel and closed range.  The orthogonal
complement of the range is the kernel of the adjoint, so the range
is all of $H^2(\Hm)$.  Solve
\[
 T_{V_r}F=-P_-W,
 \qquad M_{\mathrm{reg},-}=I+F,
 \qquad M_{\mathrm{reg},+}=(I+F)V_r.
\]
The lower difference is in $H^2(\Hm)$ by construction.  The
equation says that the lower projection of the upper difference
vanishes, placing it in $H^2(\Hp)$.

The determinants have matching boundary values because
$\det V_r=1$.  They belong to $H^1(\mu)$, since the entries are
constants plus line $H^2$ functions.  The Hardy gluing fact in
Appendix~\ref{app:analytic} makes these determinants one
constant.  The line Hardy normalization at imaginary infinity
fixes this constant as one.  The adjugate matrices therefore give
inverse matrices with the same Hardy regularity.

We also check the normalization of the original factors, uniformly
in the spatial split.  The half-line trace formula and the line
Hardy Cauchy estimate give, for $w=\alpha+\ii\beta$,
\[
 \|f\|_2^2\leq\pi\|u\|_{L^2(x,+\infty)}^2,
 \qquad |f(w)|\leq\frac{\|f\|_2}{2\sqrt{\pi\beta}},
 \qquad |\log A(w)|\leq
             \frac{\|u\|_{L^2(x,+\infty)}^2}{2\beta}.
\]
The corresponding left-half estimates apply to $g,B$.  It follows
that $f,g\to0$, $A,B,a\to1$, and $M_+\to I$ as
$\beta\to+\infty$; reflection gives the lower normalization.
The same Cauchy estimate gives the normalization of $M_{\mathrm{reg}}$.

The products $M_{\mathrm{phys},\pm}M_{\mathrm{reg},\pm}^{-1}$
have $H^1(\mu)$ entries and equal boundary values, as follows
directly by cancelling their common jump.  Hardy gluing makes
them a constant matrix and the normalization makes it $I$.
This identifies the two solutions.  Comparing any two line Hardy
solutions by the same determinant and gluing argument proves
uniqueness.
\end{proof}

\section{Classical asymptotics and the spectral density}
\label{sec:classical}

We first identify the amplitude that any classical modified scattering
profile must have.  The point is that the transmission coefficient is
conserved, whereas at times of small $L^6$ norm it can be approximated by
a quadratic expression in the ordinary Fourier transform.  Comparing
these two descriptions determines the amplitude without assuming a
spatial weight.  We then recall the asymptotic formula for localized
data and explain its conversion from physical rays to the exact Fourier
profile.  This conversion also fixes the constant phase used below.

\subsection{The amplitude of a fixed Fourier profile}

\begin{proposition}\label{prop:necessity}
Let $u_0\in L^2(\R)$, and suppose that $V\in L^2(\R)$ satisfies
\eqref{eq:fixed}.  Then the following assertions hold.
\begin{enumerate}
\item The amplitude of $V$ is determined by the initial transmission
coefficient:
\begin{equation}\label{eq:fourier-density}
                 |V(\xi)|^2=\rho_{u_0}(\xi)
                 \quad\hbox{for almost every }\xi\in\R.
\end{equation}
\item The Fourier profile corrected by this density converges strongly:
\begin{equation}\label{eq:corrected-fourier-profile}
 Z_{u_0}(t):=e^{\ii\rho_{u_0}\log t}F_{u_0}(t)
                         \longrightarrow V\quad\hbox{in }L^2(\R).
\end{equation}
\end{enumerate}
\end{proposition}

\begin{proof}
Multiplication of $u$ by $\sqrt2$ gives the defocusing equation with
cubic coefficient one.  Therefore \cite[Theorem~2, (1.12)]{IfrimTataru2023}
gives
\[
                    \int_0^\infty\|u(t)\|_6^6\dd t<\infty.
\]
We can consequently choose $t_j\to\infty$ such that
$u(t_j)\in L^2\cap L^6$ and $\|u(t_j)\|_6\to0$.  We will compute the
transmission coefficient at these times.

Fix $w=\alpha+\ii y$ with $\alpha\in\R$ and $y>0$.  The height $y$
provides exponential decay in the integral operator
\[
 (K_wf)(x)=\int_{-\infty}^x e^{2\ii w(x-s)}f(s)\dd s.
\]
The $L^{6/5}$ and $L^1$ norms of its convolution kernel give, respectively,
\[
 \|K_wf\|_\infty\le c_y\|f\|_6,
 \qquad \|K_wf\|_2\le\frac1{2y}\|f\|_2,
 \qquad c_y=\left(\frac5{12y}\right)^{5/6}.
\]
Here $c_y$ is a fixed kernel norm; it does not depend on the potential.
For $u\in L^2\cap L^6$ with $c_y\|u\|_6\le1/4$, consider the Riccati
equation
\[
                  m_w=K_w(\overline u-um_w^2).
\]
The left Jost solution is the column $Y$ of \eqref{eq:zs}
which agrees with the first free column
$(e^{-\ii wx},0)^{\mathrm T}$ as $x\to-\infty$ after removal
of that free exponential.  Its component quotient is
$m_w=Y_2/Y_1$.  The displayed integral equation selects this
quotient by the condition that it vanishes at the left endpoint.
Its right-hand side maps the ball
$\|m\|_\infty\le2c_y\|u\|_6$ into itself: indeed,
\[
 \|K_w(\overline u-um^2)\|_\infty
 \le c_y\|u\|_6(1+4c_y^2\|u\|_6^2)
 \le2c_y\|u\|_6.
\]
For two elements of this ball, the difference of their images is at
most $4c_y^2\|u\|_6^2$ times their $L^\infty$ distance.  This factor
is at most $1/4$, so contraction gives a unique solution.  The $L^2$
kernel bound and the equation also show that $m_w\in L^2$.  More
precisely,
\[
 \|m_w-K_w\overline u\|_2
 \le\frac1{2y}\|um_w^2\|_2
 \le\frac{2c_y^2}{y}\|u\|_2\|u\|_6^2.
\]

For a smooth compactly supported potential, integration of the first
Jost component yields $\log a_u(w)=\int_\R um_w\dd x$.  Taking its
real part and estimating the preceding difference therefore gives
\begin{equation}\label{eq:quadratic-transmission}
 \left|\log|a_u(w)|-
       \int_\R\frac{2y\,|\widehat u(\xi)|^2}
                         {(2\alpha+\xi)^2+4y^2}\dd\xi\right|
 \le\frac{2c_y^2}{y}\|u\|_2^2\|u\|_6^2.
\end{equation}
For clarity, the quadratic term in this calculation follows directly
from Plancherel.  If
$C(s)=\int_\R u(x)\overline{u(x-s)}\dd x$, then
\[
 C(s)=\int_\R e^{\ii s\xi}|\widehat u(\xi)|^2\dd\xi,
 \qquad
 \int_\R uK_w\overline u\dd x
 =\int_\R\frac{|\widehat u(\xi)|^2}
                       {2y-\ii(2\alpha+\xi)}\dd\xi.
\]
The real part of the reciprocal of the last denominator is exactly the positive kernel
in \eqref{eq:quadratic-transmission}.

The same identity holds for $u\in L^2\cap L^6$ under the stated
smallness condition.  To see this, approximate $u$ in both spaces by
$u_n\in C_c^\infty$, preserving that condition.  The contraction
estimate gives $m_{w,n}\to m_w$ in $L^\infty$.  Subtracting their
integral equations, using the $L^2$ kernel bound, and absorbing the
term involving $m_{w,n}-m_w$ gives convergence in $L^2$ as well.
Thus
\[
 \left|\int u_nm_{w,n}-\int um_w\right|
 \le\|u_n-u\|_2\|m_{w,n}\|_2
       +\|u\|_2\|m_{w,n}-m_w\|_2\longrightarrow0.
\]
Lemma~\ref{lem:direct} identifies this limit with the normalized
$\log a_u(w)$.  The logarithm is fixed by the outer normalization,
so no ambiguity by an integral multiple of $2\pi\ii$ occurs.

Define the positive kernel
\[
                  P_w(\xi)=\frac{2y}{(2\alpha+\xi)^2+4y^2}.
\]
Conservation of $a_u$, its outer representation, and the change of
variables $\xi=-2k$ give
\begin{equation}\label{eq:transmission-poisson}
               \log|a_{u_0}(w)|
                  =\int_\R P_w(\xi)\rho_{u_0}(\xi)\dd\xi.
\end{equation}
On the other hand, the assumed convergence \eqref{eq:fixed} and mass
conservation imply $\|V\|_2=\|u_0\|_2$ and
\begin{equation}\label{eq:density-poisson}
       \big\||F_{u_0}(t_j)|^2-|V|^2\big\|_1\longrightarrow0.
\end{equation}
In fact, its left-hand side is bounded by
\[
 (\|F_{u_0}(t_j)\|_2+\|V\|_2)
 \|F_{u_0}(t_j)-e^{-\ii|V|^2\log t_j}V\|_2.
\]
Since $|F_{u_0}(t_j)|=|\widehat{u(t_j)}|$, we may let $j\to\infty$
in \eqref{eq:quadratic-transmission}.  The error tends to zero, while
\eqref{eq:transmission-poisson}--\eqref{eq:density-poisson} identify the two remaining
terms.  Consequently
\[
             \int_\R P_w(\xi)
                    (|V(\xi)|^2-\rho_{u_0}(\xi))\dd\xi=0
                    \qquad(w\in\C_+).
\]
These are all the Poisson averages of the integrable difference.
Indeed, for $s>0$ the usual Poisson kernel is
$P_s^0(x)=\pi^{-1}s/(x^2+s^2)$, and
$P_{\alpha+\ii y}(\xi)=\pi P_{2y}^0(\xi+2\alpha)$.
Thus $P_{2y}^0*(|V|^2-\rho_{u_0})$ vanishes everywhere.  Letting
$y\downarrow0$ and using the $L^1$ approximate-identity property
proves \eqref{eq:fourier-density}.  Substitution of that identity into
\eqref{eq:fixed} gives \eqref{eq:corrected-fourier-profile}.
\end{proof}

\subsection{The coefficient in physical ray coordinates}

The classical asymptotic formula is expressed with the scattering
normalization of Deift and Zhou.  We record the change of variables
explicitly because its phase, and not only the modulus of the
reflection coefficient, enters the final coefficient.  In this
subsection the spectral variable is $z=-2k$.  It has the same numerical
value as the ray coordinate $x/(2t)$, but the Fourier and physical ray
profiles remain different functions of time.

For a localized potential $u$, let $\psi^{(\pm)}$ solve
\begin{equation}\label{eq:classical-spatial-system}
                   \psi_x=(\ii zJ/2+Q(u))\psi,
\end{equation}
with $\psi^{(\pm)}e^{-\ii xzJ/2}\to I$ as $x\to\pm\infty$.
The connection matrix and its reflection coefficient are defined by
\begin{equation}\label{eq:classical-scattering-matrix}
 \psi^{(+)}=\psi^{(-)}\widetilde T_u,
 \qquad
 \widetilde T_u=
 \begin{pmatrix}\widetilde a_u&\overline{\widetilde b_u}\\
                 \widetilde b_u&\overline{\widetilde a_u}\end{pmatrix},
 \qquad R_u=-\frac{\overline{\widetilde b_u}}{\overline{\widetilde a_u}}.
\end{equation}
These coefficients are used only when invoking this classical
normalization.  Appendix~\ref{app:conventions} gives the basis
comparison with the coefficients fixed in
Section~\ref{sec:spectral}:
\begin{equation}\label{eq:dz-conversion}
 \widetilde a_u(z)=\overline{a_u(-z/2)},\qquad
 R_u(z)=-\overline{r_u(-z/2)}
                     \frac{\overline{a_u(-z/2)}}{a_u(-z/2)}.
\end{equation}
In particular, if $L_u$ denotes the logarithm of the squared
transmission reciprocal in this variable, then
\begin{equation}\label{eq:classical-log-density}
 L_u(z):=\log(1-|R_u(z)|^2)
       =-2\log|\widetilde a_u(z)|=-2\pi\rho_u(z).
\end{equation}

For $u\in H^{1,1}(\R)$, the direct scattering theorem gives
$R_u\in H^{1,1}(\R)$ and $\|R_u\|_\infty<1$;
see \cite[\S3]{DeiftZhou2003}.  The phase contribution from spectral
points to the left of $z$ is described by
\begin{equation}\label{eq:one-sided-log-transform}
           (\cT h)(z)=\frac1\pi
                \int_{-\infty}^z\log(z-s)h'(s)\dd s.
\end{equation}
We apply this transform only when the integral is defined.  In
particular it is absolutely convergent for $h=L_u$.  The strict
reflection bound implies
$|L_u'|\lesssim_u |R_u||R_u'|$.  Near $s=z$, use boundedness of $R_u$,
the $L^2$ norm of $R_u'$, and the local square integrability of the
logarithm.  In the left tail, use
\[
 \int_{-\infty}^{z-1}\log(z-s)|R_u(s)R_u'(s)|\dd s
 \le
 \|\log(z-\cdot)R_u\|_{L^2(-\infty,z-1)}\|R_u'\|_2<\infty;
\]
the weighted assumption $sR_u\in L^2$ bounds the first factor.

To write the remaining phase without choosing an argument at a zero
of the reflection coefficient, let $\Gamma$ be Euler's gamma function
and set
\begin{equation}\label{eq:gamma-phase}
 \mathcal G(s)=
 \begin{cases}
       \Gamma(\ii s)/|\Gamma(\ii s)|,&s>0,\\
       -\ii,&s=0.
 \end{cases}
\end{equation}
The identity $\Gamma(1+\ii s)=\ii s\Gamma(\ii s)$ shows that
$\mathcal G$ extends smoothly to $s=0$.  The ratio
$\rho_u/|R_u|^2$ has the positive limiting value $1/(2\pi)$ at
$R_u=0$.  We can therefore define the classical coefficient on all
of $\R$ by
\begin{equation}\label{eq:classical-ray-coefficient}
 W_u(z)=e^{\ii\cT L_u(z)+\ii\pi/4-\ii\rho_u(z)\log2}
       R_u(z)\left(\frac{\rho_u(z)}{|R_u(z)|^2}\right)^{1/2}
       \mathcal G(\rho_u(z)),
\end{equation}
where $W_u=0$ wherever $R_u=0$.  This formula retains the phase of
$R_u$ through its complex value and has the exact modulus
$|W_u|^2=\rho_u$.

\begin{lemma}\label{lem:classical-ray}
For $u_0\in H^{1,1}(\R)$, the coefficient
\eqref{eq:classical-ray-coefficient} has the following properties.
\begin{enumerate}
\item Its amplitude and the strong ray asymptotic formula are
\begin{equation}\label{eq:classical-ray-strong}
 |W_{u_0}|^2=\rho_{u_0},\qquad
 \|G_{u_0}(t)-e^{-\ii\rho_{u_0}\log t}W_{u_0}\|_2
                                      \longrightarrow0.
\end{equation}
\end{enumerate}
\end{lemma}

\begin{proof}
By \cite[\S1, asymptotic formula (1.2)]{DeiftZhou2003}, for every
$0<\kappa<1/4$ there is a coefficient $\alpha$ such that
\[
 u(t,x)=t^{-1/2}\alpha(x/(2t))
        e^{\ii x^2/(4t)-\ii\rho(x/(2t))\log(2t)}
                +O_u(t^{-1/2-\kappa}),
\]
uniformly for $x\in\R$.  Its modulus is $|\alpha|^2=\rho/2$, and
its phase is specified by the factors appearing in
\eqref{eq:classical-ray-coefficient}.  In particular
$W=\sqrt2e^{-\ii\rho\log2}\alpha$, with both sides zero on
$\rho=0$.

Substitution of $x=2tz$ into the ray normalization gives
\[
 \sup_{z\in\R}|e^{\ii\rho(z)\log t}G(t,z)-W(z)|
                                      \lesssim_u t^{-\kappa}.
\]
This estimate implies convergence against every compactly supported
smooth test function.  Moreover, mass conservation bounds
$e^{\ii\rho\log t}G(t)$ in $L^2$, so density of these test functions
gives weak $L^2$ convergence to $W$.  The trace identity now accounts
for all the mass:
\[
 \|e^{\ii\rho\log t}G(t)\|_2^2
       =\|u_0\|_2^2=\int_\R\rho(z)\dd z=\|W\|_2^2.
\]
The Hilbert-space identity for the squared distance then upgrades
weak convergence to strong convergence, proving the assertion.
\end{proof}

\subsection{Passage to the exact Fourier profile}

Physical ray asymptotics do not by themselves imply an asymptotic
formula for the exact Fourier profile: the conversion contains a
Schr\"odinger operator acting on a function which depends on time.
For smooth data we can control the derivatives of that function
explicitly, and this makes the conversion error tend to zero.

\begin{lemma}\label{lem:smoothstrong}
For $u\in\Sch(\R)$, define $V_u=e^{\ii\pi/4}W_u$.  Then:
\begin{enumerate}
\item The fixed Fourier coefficient has modulus $|V_u|^2=\rho_u$ and
satisfies
\begin{equation}\label{eq:classical-asymptotic-coefficient}
             \|e^{\ii\rho_u\log t}F_u(t)-V_u\|_2
                                               \longrightarrow0.
\end{equation}
\item On $\rho_u>0$, its argument is given, modulo $2\pi$, by
\begin{equation}\label{eq:classical-profile}
 \arg V_u=\cT L_u+\arg\Gamma(\ii\rho_u)+\arg R_u
                            -\rho_u\log2+\frac\pi2.
\end{equation}
\end{enumerate}
\end{lemma}

\begin{proof}
For $t>0$, let $K_t=e^{-\ii(4t)^{-1}\partial_z^2}$.  It is a unitary
operator on $L^2(\R)$ with Fourier multiplier
$e^{\ii\eta^2/(4t)}$, where $\eta$ is the Fourier variable dual to
$z$.  Completing the square in the inverse Fourier formula gives the
exact identity
\begin{equation}\label{eq:fresnel-identity}
                         K_tF_u(t)=e^{\ii\pi/4}G_u(t).
\end{equation}
Indeed,
\[
 u(t,2tz)=\frac{e^{\ii tz^2}}{\sqrt{2\pi}}
                  \int_\R e^{-\ii t(\xi-z)^2}F_u(t,\xi)\dd\xi,
\]
while the kernel of $K_t$ is
$e^{\ii\pi/4}\sqrt{t/\pi}\,e^{-\ii t(z-\xi)^2}$.  The identity
first holds for Schwartz functions, with the oscillatory integrals
understood in the usual limiting sense, and then holds in $L^2$ by
unitarity.  Lemma~\ref{lem:classical-ray} therefore gives
\begin{equation}\label{eq:fourier-asymptotic-error}
            \|K_tF_u(t)-e^{-\ii\rho_u\log t}V_u\|_2
                                                   \longrightarrow0.
\end{equation}

For Schwartz potentials the direct scattering coefficients
$R_u,L_u,\rho_u$ are Schwartz.  The multiplier formula in
Lemma~\ref{lem:multiplier} implies
$\cT L_u\in H^m(\R)$ for every integer $m\ge0$: the logarithmic
singularity of the multiplier is square integrable near zero, and
its growth at infinity is dominated by any positive power.
Sobolev embedding consequently bounds each derivative of
$\cT L_u$.  The smooth extensions described before
\eqref{eq:classical-ray-coefficient} show that $V_u\in H^m$ for every
$m$ as well.

Write $\lambda=\log t$ only in the following differentiation.  The
product rule gives all the terms of the second derivative:
\[
 \partial_z^2(e^{-\ii\lambda\rho_u}V_u)
 =e^{-\ii\lambda\rho_u}
   \big[V_u''-2\ii\lambda\rho_u'V_u'
              -\ii\lambda\rho_u''V_u
              -\lambda^2(\rho_u')^2V_u\big].
\]
Since $\rho_u'$ and $\rho_u''$ are bounded and $V_u,V_u',V_u''$
belong to $L^2$, this yields
\begin{equation}\label{eq:logarithmic-profile-derivatives}
        \|\partial_z^2(e^{-\ii\rho_u\log t}V_u)\|_2
                                      \lesssim_u1+(\log t)^2.
\end{equation}
The elementary bound $|e^{\ii s}-1|\le|s|$ applied to the multiplier
of $K_t-I$ now gives
\begin{equation}\label{eq:fresnel-profile-error}
       \|(K_t-I)e^{-\ii\rho_u\log t}V_u\|_2
               \lesssim_u\frac{1+(\log t)^2}{t}\longrightarrow0.
\end{equation}
Finally, unitarity of $K_t$ and the triangle inequality show that
\begin{align*}
 \|F_u(t)-e^{-\ii\rho_u\log t}V_u\|_2
 &\le\|K_tF_u(t)-e^{-\ii\rho_u\log t}V_u\|_2\\
 &\quad+\|(K_t-I)e^{-\ii\rho_u\log t}V_u\|_2.
\end{align*}
Both terms tend to zero by \eqref{eq:fourier-asymptotic-error} and
\eqref{eq:fresnel-profile-error}.  The modulus and phase follow from the defining
formula for $W_u$.
\end{proof}

\section{Strong asymptotics in physical ray coordinates}
\label{sec:ray-asymptotics}

\subsection{The constant-reflection problem}
\label{sec:model}

The quadratic oscillation in the inverse problem changes on a spectral
interval of length $t^{-1/2}$. On this interval we expect the reflection
coefficient to be almost constant. We first solve the problem with a
constant reflection coefficient, keeping its complex phase. This will
identify both the size of the limiting ray profile and the matrix from
which its phase can be recovered.

Let $b\in\C$, $|b|<1$, be the constant reflection coefficient. Define
the nonnegative number $\tau_b$ and the complex number $p_b$ by
\begin{equation}\label{eq:weber-coeff}
 \tau_b=-\frac{\log(1-|b|^2)}{2\pi},\qquad
 p_b=-\frac{2\sqrt\pi\,e^{\ii\pi/4}e^{-\pi\tau_b/2}
                    2^{-\ii\tau_b}}
                 {b\,\Gamma(-\ii\tau_b)}\quad(b\ne0),
 \qquad p_0=0.
\end{equation}
Here $\Gamma$ is the gamma function, and $2^{-\ii\tau_b}$ means
$\exp(-\ii\tau_b\log2)$ with the real logarithm. The number $\tau_b$
measures the transmission density of the constant data; $p_b$ will be
the coefficient of the corresponding potential. The gamma identity
$|\Gamma(\ii\tau)|^2=\pi/(\tau\sinh\pi\tau)$, for $\tau>0$,
gives
\begin{equation}\label{eq:weber-modulus}
 |p_b|^2
 =\frac{4\tau_b e^{-\pi\tau_b}\sinh(\pi\tau_b)}{|b|^2}
 =2\tau_b=\frac{-\log(1-|b|^2)}{\pi}.
\end{equation}
The identity extends to $b=0$ by continuity.

\begin{lemma}[The constant-reflection factors]\label{lem:weber}
For $|b|<1$, there exist matrices $M^W_{b,+}(y,\zeta)$ and
$M^W_{b,-}(y,\zeta)$, where $y\in\R$ is the spatial variable and
$\zeta\in\Hp$ or $\zeta\in\Hm$ is the spectral variable, with the following
properties.
\begin{enumerate}
\item The factors are bounded and analytic in their respective
half-planes, have determinant one, and satisfy
\begin{align}
 M^W_{b,-}(y,\zeta)&=\sigma_1\overline{M^W_{b,+}(y,\bar \zeta)}\sigma_1,
 \notag\\
 M^W_{b,+}(y,\zeta)&=M^W_{b,-}(y,\zeta)
 \begin{pmatrix}
 1-|b|^2&\bar b e^{-\ii(\zeta^2+2y\zeta)}\\
 -b e^{\ii(\zeta^2+2y\zeta)}&1
 \end{pmatrix},\qquad \zeta\in\R.
 \label{eq:weber-jump}
\end{align}
\item In either half-plane the spatial equation is
\begin{equation}\label{eq:weber-lax}
 \partial_yM^W_b=-\ii \zeta[J,M^W_b]
                +Q(p_b e^{\ii y^2})M^W_b.
\end{equation}
At a fixed nonreal $\zeta$, both the factors and their inverses are
continuous in $b$ and $y$, uniformly when $b$ ranges over a compact
subset of the unit disk and $y$ ranges over a bounded interval.
\end{enumerate}
\end{lemma}

\begin{proof}
We calculate the coefficient because changing the reflection
convention changes its phase. The special-function identities used
below are the standard differential equation, recurrence, values at
zero, and sector asymptotics for the parabolic cylinder function;
see \cite[\S\S12.2, 12.8, 12.9]{DLMF}.

For a nonzero complex number $p$, set $\tau=|p|^2/2$. Let
$c=\sqrt2e^{\ii\pi/4}$ and $c_1=e^{\pi\tau/4}2^{-\ii\tau/2}$.
Using the entire parabolic cylinder function $D_\nu$, define the
column vector
\[
 f(v)=c_1\begin{pmatrix}
 D_{\ii\tau}(cv)\\
 (c\ii\tau/p)D_{\ii\tau-1}(cv)
 \end{pmatrix}.
\]
The recurrence relations
$D_\nu'(z)=\nu D_{\nu-1}(z)-zD_\nu(z)/2$ and
$D_{\nu-1}'(z)=zD_{\nu-1}(z)/2-D_\nu(z)$ show, since $c^2=2\ii$,
that
\[
 f_1'=-\ii vf_1+pf_2,\qquad
 f_2'=\ii vf_2+\bar pf_1.
\]
Thus $f'=(-\ii vJ+Q(p))f$. For positive real $v$, the
normalization of $c_1$ and the sector expansion give
\[
 f_1(v)=e^{-\ii v^2/2}v^{\ii\tau}(1+O(v^{-2})),
 \qquad f_2(v)=O(v^{-1}).
\]
On the real axis the differential equation preserves
$|f_1|^2-|f_2|^2$. The preceding limits therefore imply that this
difference equals one.

For a column $h$ write
$h^\sharp(v)=\sigma_1\overline{h(\bar v)}$. Introduce the columns
$g=f^\sharp$, $l(v)=Jf(-v)$, and $k=l^\sharp$, and the matrices
$F=[f,g]$, $L=[l,k]$. They solve the same first-order system.
Moreover, their determinants equal one on the real axis and hence
everywhere by analyticity. Their transition matrix is constant:
\[
 L=FS,\qquad S=\begin{pmatrix}A&\bar B\\B&A\end{pmatrix},
 \quad
 A=|f_1(0)|^2+|f_2(0)|^2,\quad B=-2f_1(0)f_2(0).
\]
Here $A$ and $B$ are local constants in this calculation, rather
than the half-line transmissions used elsewhere. The formula for
$D_\nu(0)$ gives
\[
 |f_1(0)|^2=\frac{e^{\pi\tau}+1}{2},\qquad
 |f_2(0)|^2=\frac{e^{\pi\tau}-1}{2},\qquad A=e^{\pi\tau}.
\]
With $P_+=[l/A,g]$ and $P_-=[f,k/A]$, the identities
$l=Af+Bg$ and $k=\bar Bf+Ag$ yield
\[
 P_-^{-1}P_+=
 \begin{pmatrix}A^{-2}&-\bar B/A\\B/A&1\end{pmatrix}.
\]
Set $b=-B/A$. Since $A^2-|B|^2=1$, its modulus satisfies
$|b|^2=1-e^{-2\pi\tau}$. The complex value follows from
\[
 \Gamma\!\left(\frac{1-\ii\tau}{2}\right)
 \Gamma\!\left(1-\frac{\ii\tau}{2}\right)
 =2^{\ii\tau}\sqrt\pi\,\Gamma(1-\ii\tau),
 \qquad \Gamma(1-\ii\tau)=-\ii\tau\Gamma(-\ii\tau).
\]
Substituting the values of $f_1(0)$ and $f_2(0)$ into $b=-B/A$
now gives exactly $p=p_b$ in \eqref{eq:weber-coeff}.

To insert the spatial variable, put
$D_y=e^{\ii y^2J/2}$ and $\theta(y,\zeta)=\zeta^2/2+y\zeta$, and define
\[
 M^W_{b,\pm}(y,\zeta)=D_yP_\pm(\zeta+y)e^{\ii\theta(y,\zeta)J}.
\]
Conjugating the last constant jump by $e^{\ii\theta J}$ gives
\eqref{eq:weber-jump}. Differentiation gives
\[
 \partial_yM^W_b
 =-\ii \zeta JM^W_b+\ii \zeta M^W_bJ
   +D_yQ(p_b)D_y^{-1}M^W_b.
\]
Since $D_yQ(p_b)D_y^{-1}=Q(p_be^{\ii y^2})$, this is
\eqref{eq:weber-lax}. The determinant and the reflection
identity follow directly from the paired columns.

We finally verify boundedness in the open half-planes, which is
needed for uniqueness in Subsection~\ref{sec:critical}. The expansion
\[
 D_\nu(z)=e^{-z^2/4}z^\nu(1+O(|z|^{-2}))
\]
is uniform on closed subsectors of $|\arg z|<3\pi/4$.
Applied to $-cv$ and to $\bar cv$, it bounds the columns of
$P_+(v)e^{\ii v^2J/2}$ in the upper half-plane except for small
wedges next to its two real boundary rays. The relation $L=FS$
and the real-axis expansions bound the columns on those boundary
rays as well. The gauged entries have growth at most
$\exp(C|v|^2)$, by the differential equation. Choose each missing
wedge with opening less than $\pi/2$. The Phragm\'en--Lindel\"of
principle \cite[Chapter~VI, \S4]{Conway} applies because the growth order is at most two and the
wedge opening is strictly smaller than $\pi/2$; it supplies the
bound between the real ray and the controlled sector ray.
Reflection proves the lower half-plane bound. Translation by
the real number $y$ and multiplication by the constant diagonal
matrices in the definition of $M^W$ preserve these bounds.

All these constructions and sector bounds are locally uniform
in $b$ and $y$. At $b=0$ the second components tend to zero and
the gauged first components tend to their vacuum values, so we
set $M^W_{0,\pm}=I$. The formulas and the determinant-one
identity also give the stated continuity of the inverses.
\end{proof}

\subsection{The limiting ray amplitude}
\label{sec:critical}

We now compare the factors of the actual solution with the model
at the spectral scale $t^{-1/2}$. A pointwise bound on the full
transmission is unavailable for general $L^2$ data. Dividing by
that transmission gives Hardy-space bounds which remain valid
without a common strict reflection gap. Compactness will then
identify every possible limit with the model, up to a constant
phase.

For the spectral variable $\zeta\in\R$, let
$\dd\mu(\zeta)=\dd \zeta/[\pi(1+\zeta^2)]$. This is harmonic measure on the
real line as seen from $\ii$ in the upper half-plane, or from
$-\ii$ in the lower half-plane. The notation $H^2(\mu)$ denotes
the Hardy space obtained from the usual disk $H^2$ space by the
Cayley transformation, with this boundary measure. For matrices
we sum the squared Hardy norms of their entries. We use
$\|C\|_F^2=\sum_{m,n=1}^2|C_{mn}|^2$ for the Frobenius norm.

\begin{lemma}[Transmission-weighted factors]\label{lem:weighted}
At a fixed spatial split, write
$a^\sharp(\zeta)=\overline{a(\bar \zeta)}$ in the lower half-plane and
define
\[
 K_+=M_+/a,\qquad K_-=M_-/a^\sharp,\qquad C=M_-(-\ii).
\]
The same definitions may be made after a constant diagonal unitary
conjugation. Then
\begin{enumerate}
\item The factors obey
\begin{equation}\label{eq:weighted}
 \|K_+\|_{H^2(\mu)}^2+\|K_-\|_{H^2(\mu)}^2
 \le4\|C\|_F^2.
\end{equation}
\item Their boundary relation is
\begin{equation}\label{eq:weighted-jump}
 K_+=K_-W,\qquad W=(\bar a/a)V_r,
 \qquad\|W\|_{L^\infty}\le\sqrt3.
\end{equation}
\end{enumerate}
\end{lemma}

\begin{proof}
The purpose of the transmission weight is to give each entry a
positive harmonic majorant. With the physical half-line quotients
$f$ and $g$, set
\[
 H(\zeta)=\Re\frac{1+f(\zeta)g(\zeta)}{1-f(\zeta)g(\zeta)}.
\]
The maximum-principle argument in Lemma~\ref{lem:proper}
gives $|A/a|^2\le H$ and $|B/a|^2\le H$ in the upper
half-plane. Each remaining entry of $K_+$ is one of these
functions times analytic contractions. Consequently
\[
 \|K_+\|_{H^2(\mu)}^2\le4H(\ii),\qquad
 \|K_-\|_{H^2(\mu)}^2\le4H(\ii).
\]
The second estimate follows by reflection. Here the passage
from a harmonic majorant to a Hardy norm is the standard
Hardy-space estimate recalled in
Appendix~\ref{app:analytic}.

The splice formula \eqref{eq:half-splice} also gives
\[
 \frac{1+fg}{1-fg}=M_{+,11}M_{+,22}+M_{+,12}M_{+,21}.
\]
At $\zeta=\ii$, the inequality $2|vw|\le|v|^2+|w|^2$ bounds
the real part by $\|M_+(\ii)\|_F^2/2$. Reflection identifies
this norm with $\|C\|_F$, proving \eqref{eq:weighted}.
Dividing the original jump by $a$ and $a^\sharp$ gives
\eqref{eq:weighted-jump}; on the real line the extra
scalar has modulus one. Diagonal unitary conjugation does
not change any of these norms.
\end{proof}

Write $s=1/a$ for the reciprocal transmission. It is analytic
in the upper half-plane and has modulus at most one there.
The following lemma is stated for sequences because it must also
apply when the initial datum varies through smooth approximations.

\begin{lemma}[Compactness on the rescaled spatial interval]
\label{lem:block}
For each $j$, let $p_j$ be the potential on $[-1,1]$ associated
with original Dirac factors $M_{j,\pm}(y,\zeta)$ whose jump has the
quadratic phase in \eqref{eq:weber-jump}. Let $r_j$ and $s_j$
be their rescaled reflection and reciprocal transmission. Assume
\[
 \sup_j\|p_j\|_{L^2[-1,1]}<\infty,\qquad
 r_j\longrightarrow b,\quad s_j\longrightarrow\sigma
 \quad\hbox{in }L^2(\mu),
 \qquad |b|<1,\quad\sigma\ne0.
\]
Set $C_j(y)=M_{j,-}(y,-\ii)$. Then
\begin{enumerate}
\item The matrices $C_j$ are uniformly bounded on $[-1,1]$.
\item After passing to a subsequence, there is $d\in\C$, $|d|=1$,
such that
\begin{align*}
 p_j&\rightharpoonup d^2p_be^{\ii y^2}
           &&\hbox{in }L^2[-1,1],\\
 C_j(y)&\longrightarrow\diag(d,\bar d)M^W_{b,-}(y,-\ii)
           &&\hbox{uniformly for }y\in[-1,1].
\end{align*}
\end{enumerate}
\end{lemma}

\begin{proof}
The only possible loss of compactness is an unbounded value of
an inverse factor at a nonreal spectral point. We rule this out
by normalizing such a value and showing that its limit would
have both zero determinant and nonzero determinant.

By weak compactness, pass to a subsequence such that
$p_j\rightharpoonup p$ in $L^2[-1,1]$. Bounded analytic
evaluation and the assumed Hardy boundary convergence give
$s_j\to\sigma$ locally uniformly in $\Hp$. Hence
$a_j\to\sigma^{-1}$ there; reflection gives
$a_j^\sharp\to\bar\sigma^{-1}$ in $\Hm$.

Put $\lambda_j=\|C_j(0)\|_F$. The determinant-one identity
implies $\lambda_j\ge\sqrt2$. Suppose a subsequence satisfies
$\lambda_j\to\infty$. At $\zeta=-\ii$, the spatial equation and
$\|p_j\|_1\le\sqrt2\|p_j\|_2$ imply, by forward and backward
Gronwall estimates,
\[
 c\lambda_j\le\|C_j(y)\|_F\le C\lambda_j,
 \qquad -1\le y\le1,
\]
with constants independent of $j$. Lemma~\ref{lem:weighted}
therefore bounds
\[
 \frac{M_{j,+}}{a_j\lambda_j},\qquad
 \frac{M_{j,-}}{a_j^\sharp\lambda_j}
\]
in the two harmonic-measure Hardy spaces. At $y=0$, take their
weak Hardy limits, denoted by $K_{\infty,+}$ and $K_{\infty,-}$.
Analytic evaluation makes the convergence locally uniform.
Multiplication by the limiting transmissions then produces
locally uniform limits $M_{\infty,\pm}$ of
$M_{j,\pm}/\lambda_j$. The lower limit is nonzero at $-\ii$
because its Frobenius norm there equals one.

These limits propagate throughout $[-1,1]$. Indeed, on every
compact spectral set, the spatial fundamental matrices are
uniformly bounded and equicontinuous; the $L^2$ bound on $p_j$
gives a common modulus of continuity through
$\int_I|p_j|\le |I|^{1/2}\|p_j\|_2$. In the integral
equation, uniform convergence of a fundamental matrix and weak
convergence of $p_j$ pass its potential term to the limit.
Thus the limiting factors solve the spatial equation with
potential $p$ for every $y$. Their determinants vanish because
the determinants before passage to the limit are
$\lambda_j^{-2}$.

We must also keep the boundary jump in this limit. Since
$s_j\to\sigma\ne0$ in measure and $|s_j|\le1$, the scalar
boundary ratios $s_j/\bar s_j$ converge in measure to
$\sigma/\bar\sigma$. The weighted jump matrices are uniformly
bounded, so they converge strongly in $L^2(\mu)$ to
$(\sigma/\bar\sigma)V_b$. Here $V_b$ denotes the matrix in
\eqref{eq:weber-jump}, at the given $y$. Weak Hardy
convergence then passes the jump to the limit: one first tests
against bounded functions, for which strong $L^2$ convergence
controls the product, and uses density for general $L^2$ tests.
Consequently
\[
 K_{\infty,+}=K_{\infty,-}\frac{\sigma}{\bar\sigma}V_b,
 \qquad
 M_{\infty,+}=\sigma^{-1}K_{\infty,+},\quad
 M_{\infty,-}=\bar\sigma^{-1}K_{\infty,-},
\]
and thus $M_{\infty,+}=M_{\infty,-}V_b$.

Multiply the limiting factors on the right by the corresponding
Weber inverses. The products have matching real boundary values
and belong to harmonic-measure $H^1$, since the inverses are
bounded. The upper and lower Hardy gluing statement in
Appendix~\ref{app:analytic} says that their common extension
is constant in the spectral variable. We may therefore write
\[
 M_{\infty,\pm}(y,\zeta)=L(y)M^W_{b,\pm}(y,\zeta).
\]
Substitute this identity into the two spatial equations. It gives
\[
 L_y=-\ii \zeta[J,L]+Q(p)L-LQ(p_be^{\ii y^2}).
\]
The left side does not depend on $\zeta$. Evaluating at two distinct
spectral points yields $[J,L]=0$. The reduction then implies
$L=\diag(d,\bar d)$ for a scalar $d=d(y)$. Its derivative is
diagonal whereas the last two terms are off diagonal, so $L_y=0$.
The nonzero lower evaluation implies $d\ne0$. Hence
$\det L=|d|^2>0$, contradicting the vanishing determinant of
$M_{\infty,-}$. This proves boundedness of $\lambda_j$ and
therefore of $C_j$ on the whole interval.

Apply the same compactness argument to the unnormalized factors.
Now the limiting determinant is one, so $|d|=1$. The upper-right
entry of $Q(p)L=LQ(p_be^{\ii y^2})$ reads
$p\bar d=dp_be^{\ii y^2}$, or $p=d^2p_be^{\ii y^2}$.
The compactness of the fundamental matrices gives uniform
convergence of the lower evaluations in $y$.
\end{proof}

We next explain how the rescaled interval arises from a physical
ray. For $t>0$, define the positive spectral scale
$h=(2\sqrt t)^{-1}$. A ray is specified by $z\in\R$ and is
centered at $x=2tz$; its stationary spectral point is $-z/2$.
Use $y$ and $\zeta$ to measure, respectively, spatial and spectral
displacements from these centers:
\begin{equation}\label{eq:block-variables}
 x=\frac{z}{2h^2}+\frac yh,\qquad
 k=-\frac z2+h\zeta,\qquad
 P_{t,z}(y)=\sqrt2\,e^{\ii y^2}G(t,z+2hy).
\end{equation}
The new potential $P_{t,z}$ is the ray profile viewed on a
spatial interval of length comparable to $\sqrt t$. Set
\[
 H_{t,z}(y)=e^{-\ii(tz^2+zy/h)J/2},\qquad
 \widetilde M_\pm(y,\zeta)
   =H_{t,z}(y)M_\pm(t,x,k)H_{t,z}(y)^{-1}.
\]
The identity
\[
 2kx+4tk^2=\zeta^2+2y\zeta-tz^2-zy/h
\]
shows that the conjugated jump has the phase
$\zeta^2+2y\zeta$ and reflection $r(-z/2+h\zeta)$. Using the exact
definition of $G$, differentiation also gives
\[
 \partial_y\widetilde M
 =-\ii \zeta[J,\widetilde M]+Q(P_{t,z})\widetilde M.
\]
Its reciprocal transmission is $s(-z/2+h\zeta)$. Let
$C_{t,z}=\widetilde M_-(0,-\ii)$ and
$C_b^W=M^W_{b,-}(0,-\ii)$ denote the two lower evaluations
which will be compared.

Define an averaging operator on ray functions by
\[
 T_hF(z)=\int_{-1}^1(1-|y|)F(z+2hy)\dd y.
\]
The triangular weight has integral one; it is chosen because
the Fourier multiplier of $T_h$ lies between zero and one.
For a compact family $K\subset L^2(\R)$, attach a subscript
$v$ to the solution $u_v(t)=S(t)v$ and all its associated
quantities. Let $M=\sup_{v\in K}\|v\|_2$. The measurable phase
\begin{equation}\label{eq:original-phase}
 \gamma_v(t,z)=
 \unit\!\left[
 \left(C_{t,z}(C^W_{r_v(-z/2)})^{-1}\right)_{11}
 \right]^2
\end{equation}
extracts the diagonal rotation identified in
Lemma~\ref{lem:block}.

\begin{proposition}[Amplitude and phase of the ray profile]
\label{prop:critical}
For every strongly compact $K\subset L^2(\R)$, as $t\to\infty$,
\begin{enumerate}
\item The amplitudes and the averaging error satisfy
\begin{equation}\label{eq:uniform-amplitude}
 \sup_{v\in K}\big\||G_v(t)|^2-\rho_v\big\|_1\longrightarrow0,
 \qquad
 \sup_{v\in K}\|G_v(t)-T_hG_v(t)\|_2\longrightarrow0.
\end{equation}
\item The complex profiles satisfy
\begin{equation}\label{eq:original-connection}
 \sup_{v\in K}
 \left\|G_v(t)-\gamma_v(t,\cdot)
                   \frac{p_{r_v(-\cdot/2)}}{\sqrt2}\right\|_2
 \longrightarrow0.
\end{equation}
\end{enumerate}
\end{proposition}

\begin{proof}
We first obtain convergence in measure by the compactness lemma.
Conservation of mass will then turn this local information into
strong convergence on the whole line. It suffices to work with
an arbitrary sequence $v_j\to v$ in $L^2$ and $t_j\to\infty$:
failure of uniformity on $K$ would supply such a sequence.
Write $h_j=(2\sqrt{t_j})^{-1}$.

Lemma~\ref{lem:direct} implies local convergence in measure of
$r_{v_j}$ and $s_{v_j}$, and $L^1$ convergence of
$\ell_{v_j}$ to $\ell_v$. Since the first two symbols are
bounded by one, they converge locally in $L^2$. For every
bounded ray interval $E$ we have
\begin{equation}\label{eq:triangular-freeze}
 \int_E\int_\R
 \left(
 |r_{v_j}(-z/2+h_j\zeta)-r_v(-z/2)|^2
 +|s_{v_j}(-z/2+h_j\zeta)-s_v(-z/2)|^2
 \right)\dd\mu(\zeta)\dd z\longrightarrow0.
\end{equation}
To see this, restrict first to $|\zeta|\le R$. Local $L^2$
convergence controls the varying symbols, and translation
continuity in $L^2$ controls the displacement $h_j\zeta$ of the
limiting symbols, uniformly for $|\zeta|\le R$. On $|\zeta|>R$,
boundedness gives an error at most $8|E|\mu(\{|\zeta|>R\})$.
Let $j\to\infty$ and then $R\to\infty$.

The exact mass identity in the rescaled variables is
\begin{equation}\label{eq:block-energy}
 \int_\R\int_{-1}^1|P_{t_j,z}(y)|^2\dd y\dd z
 =4\|v_j\|_2^2.
\end{equation}
Thus centers at which the inner integral exceeds $B$ occupy
a set of measure at most $4M^2/B$. On a bounded interval of
centers, also discard the sets where
$|r_v(-z/2)|>c$ or $|s_v(-z/2)|<\eta$, with $c<1$ and
$\eta>0$. Their measures tend to zero as $c\uparrow1$ and
$\eta\downarrow0$. Finally discard centers where the
$\zeta$-integrated error in \eqref{eq:triangular-freeze} is
larger than a threshold tending slowly to zero. For the phase
assertion we additionally discard centers where
$|r_{v_j}(-z/2)-r_v(-z/2)|$ exceeds such a threshold.

On the remaining centers the compactness lemma is uniform.
For precision, if it were not uniform, choose a sequence of
failing centers. Their frozen values $b$ and $\sigma$ have
convergent subsequences in $|b|\le c$ and $|\sigma|\ge\eta$.
The symbol errors tend to zero and the spatial $L^2$ norms
are bounded by $\sqrt B$, so Lemma~\ref{lem:block} applies
and contradicts failure. The maps $b\mapsto p_b$ and
$b\mapsto C_b^W$ are uniformly continuous on $|b|\le c$.
Hence the same argument permits the moving value
$r_{v_j}(-z/2)$ in \eqref{eq:original-phase}.

The averaging operator is related to the rescaled potential by
\[
 T_hG(t,z)=\frac1{\sqrt2}\int_{-1}^1
          (1-|y|)e^{-\ii y^2}P_{t,z}(y)\dd y.
\]
Weak convergence in the compactness lemma therefore yields,
locally in measure,
\[
 |T_{h_j}G_{v_j}(t_j)|\longrightarrow\sqrt{\rho_v},
 \qquad
 T_{h_j}G_{v_j}(t_j)
 -\gamma_{v_j}(t_j,\cdot)
           \frac{p_{r_{v_j}(-\cdot/2)}}{\sqrt2}
 \longrightarrow0.
\]
The discarded sets disappear by first sending $j\to\infty$,
then $B\to\infty$, $c\uparrow1$, and $\eta\downarrow0$.

Set $F_j=T_{h_j}G_{v_j}(t_j)$. Since $T_h$ is an $L^2$
contraction, $\|F_j\|_2\le\|v_j\|_2$. Every subsequence
contains a further subsequence on which the amplitudes converge
almost everywhere. Fatou's lemma and
$\int\rho_v=\|v\|_2^2$ then give
\[
 \|v\|_2^2\le\liminf_j\|F_j\|_2^2
 \le\limsup_j\|F_j\|_2^2\le\|v\|_2^2.
\]
The equality of the limiting mass and the almost-everywhere
limit implies $\||F_j|^2-\rho_v\|_1\to0$. This is the
standard $L^1$ convergence criterion for nonnegative functions
with convergent integrals, applied to each subsequence.

With $\operatorname{sinc}(\xi)=\sin\xi/\xi$ and
$\operatorname{sinc}(0)=1$, the multiplier of $T_h$ is
$m_h(\xi)=\operatorname{sinc}^2(h\xi)\in[0,1]$.
The inequality $(1-m)^2\le1-m^2$, $0\le m\le1$, gives
\[
 \|G_{v_j}(t_j)-F_j\|_2^2
 \le\|v_j\|_2^2-\|F_j\|_2^2\longrightarrow0.
\]
The inequality
$\||f|^2-|g|^2\|_1\le(\|f\|_2+\|g\|_2)\|f-g\|_2$
and $\rho_{v_j}\to\rho_v$ in $L^1$ prove the two assertions
in \eqref{eq:uniform-amplitude}.

For the complex assertion let
$H_j=\gamma_{v_j}(t_j,\cdot)p_{r_{v_j}(-\cdot/2)}/\sqrt2$.
By \eqref{eq:weber-modulus},
$|H_j|^2=\rho_{v_j}$ and $\|H_j\|_2=\|v_j\|_2$.
For every measurable set $E$,
\[
 \int_E|F_j\bar H_j|
 \le M\left(\int_E\rho_{v_j}\right)^{1/2}.
\]
The $L^1$-compact family of densities is uniformly integrable
and tight. This estimate controls both small sets and the tails
of the pairing. The local convergence in measure already
proved therefore gives
$\Re\int F_j\bar H_j\to\|v\|_2^2$. Expanding
$\|F_j-H_j\|_2^2$ proves its convergence to zero, and adding
the averaging error proves \eqref{eq:original-connection}.
\end{proof}

\section{Construction of the asymptotic phase}
\label{sec:asymptotic-phase}

\subsection{The factors above the stationary spectral scale}
\label{sec:offcritical}

Proposition~\ref{prop:critical} determines the amplitude and
recovers a phase from the factors of the evolved potential. Our
next task is to replace those factors by quantities determined
at time zero. At a fixed positive distance from the real
spectral line, the Weyl quotients become small as mass leaves
each bounded spatial interval. The logarithms of the half-line
transmissions are then quadratic to leading order. Their limits
depend only on the conserved density $\rho$.

Fix a strongly compact $K\subset L^2$, and again let
$M=\sup_{v\in K}\|v\|_2$. To record the two errors already
controlled, write
\[
 \varepsilon_K(t)=\sup_{v\in K}
       \big\||G_v(t)|^2-\rho_v\big\|_1,
 \qquad
 \sigma_K(t)=\sup_{v\in K}\|G_v(t)-T_hG_v(t)\|_2.
\]
Both vanish as $t\to\infty$. Let
$e_K(t)=\sup_{v\in K,x}\int_x^{x+1}|u_v(t,y)|^2\dd y$
be the largest mass on an interval of physical length one.
The substitution $y=2tz$ gives
\begin{equation}\label{eq:local-evacuation}
 e_K(t)\le\varepsilon_K(t)+
 \sup_{\substack{v\in K\\ |E|\le(2t)^{-1}}}
             \int_E\rho_v(z)\dd z\longrightarrow0.
\end{equation}
The final limit follows from uniform absolute continuity of
the compact family $\{\rho_v:v\in K\}$ in $L^1$.

For a ray split $x=2tz_0$ and $k=\alpha+\ii\beta$,
$\beta>0$, define the limiting right and left transmissions by
\begin{equation}\label{eq:limiting-outers}
 \mathcal A_{v,z_0}(k)
 =\exp\left(\ii\int_{z_0}^{\infty}
                   \frac{\rho_v(z)}{z+2k}\dd z\right),
 \qquad
 \mathcal B_{v,z_0}(k)
 =\exp\left(\ii\int_{-\infty}^{z_0}
                   \frac{\rho_v(z)}{z+2k}\dd z\right).
\end{equation}
These formulas retain the mass on the corresponding side of
the physical split. The substitution $z=-2s$ in their product
and the definition $\rho_v(z)=\ell_v(-z/2)/(2\pi)$ give
$\mathcal A_{v,z_0}\mathcal B_{v,z_0}=a_v$, with precisely
the normalization of \eqref{eq:outer-a}.

The parameter $\beta$ is the height of evaluation above the
real spectral line. Put
\[
 \alpha_\beta(t)=\min\left\{1,
             \frac{2\sqrt{e_K(t)}}{1-e^{-2\beta}}\right\}.
\]
This quantity bounds the Weyl quotients: the numerator measures
the mass in unit spatial intervals, and the denominator sums
the exponential decay of the Volterra kernel across them.

\begin{lemma}[Limits of the half-line transmissions]
\label{lem:outer-splitting}
Uniformly for $v\in K$, $z_0,\alpha\in\R$, and $\beta>0$,
the following estimates hold at $x=2tz_0$ and
$k=\alpha+\ii\beta$.
\begin{enumerate}
\item The physical Weyl quotients satisfy
$|f_v(t,x;k)|,|g_v(t,x;k)|\le\alpha_\beta(t)$.
\item Each of
$|\log A_v(t,x;k)-\log\mathcal A_{v,z_0}(k)|$ and
$|\log B_v(t,x;k)-\log\mathcal B_{v,z_0}(k)|$ is at most
\begin{equation}\label{eq:outer-error}
 \frac{M^2\alpha_\beta(t)^2}{2\beta}
 +\frac{M\sigma_K(t)}{\beta}
 +\frac{M^2}{4\beta^2\sqrt t}
 +\frac{M^2}{16\beta^3t}
 +\frac{\varepsilon_K(t)}{2\beta}.
\end{equation}
In particular, for every fixed $\beta>0$,
\begin{equation}\label{eq:outer-splitting}
 \sup_{v\in K,z_0,\alpha\in\R}
 \left(
 |\log A_v(t,2tz_0;k)-\log\mathcal A_{v,z_0}(k)|
 +|\log B_v(t,2tz_0;k)-\log\mathcal B_{v,z_0}(k)|
 \right)\longrightarrow0.
\end{equation}
\end{enumerate}
\end{lemma}

\begin{proof}
We first quantify the slow variation of the ray profile which
was obtained through averaging. If $\zeta\in\R$ is a physical
translation, the exact ray representation gives
\[
 u(t,x+\zeta)=e^{\ii \zeta x/(2t)}e^{\ii \zeta^2/(4t)}
          \frac{e^{\ii x^2/(4t)}}{\sqrt{2t}}
          G\!\left(t,\frac{x}{2t}+\frac{\zeta}{2t}\right).
\]
The two errors in replacing $G$ by $T_hG$ cost
$2\sigma_K(t)$ in $L^2$. Since
$\sup_\xi|\xi|\operatorname{sinc}^2(h\xi)\le h^{-1}$,
translation of $T_hG$ by $\zeta/(2t)$ costs at most
$M|\zeta|/(2th)=M|\zeta|/\sqrt t$. Removing the constant factor
$e^{\ii \zeta^2/(4t)}$ costs at most $M\zeta^2/(4t)$. Thus
\begin{equation}\label{eq:chirped-translation}
 \|u_v(t,\cdot+\zeta)-e^{\ii \zeta(\cdot)/(2t)}u_v(t,\cdot)\|_2
 \le2\sigma_K(t)+\frac{M|\zeta|}{\sqrt t}+\frac{M\zeta^2}{4t}.
\end{equation}

For the right half-line, the Riccati equation and its
vacuum-normalized integral form are
\[
 f_x=u-2\ii kf-\bar uf^2,\qquad
 f(x)=-\int_x^\infty e^{2\ii k(y-x)}
                    \big(u(y)-\bar u(y)f(y)^2\big)\dd y,
\]
and the associated transmission obeys
\[
 \log A(x;k)=-\int_x^\infty\bar u(y)f(y;k)\dd y.
\]
These identities follow directly from the two components of
the spatial Dirac equation. They can first be checked for
smooth compactly supported potentials and then passed to
$L^2$ by fixed-height continuity of the half-line data and
the exponentially weighted bilinear estimates below.

The Schur bound $|f|\le1$ and the Volterra formula give
\[
 |f(x)|\le2\int_0^\infty e^{-2\beta \zeta}|u(x+\zeta)|\dd \zeta
 \le2\sqrt{e_K(t)}\sum_{m=0}^\infty e^{-2\beta m}
 =\frac{2\sqrt{e_K(t)}}{1-e^{-2\beta}}.
\]
The same argument for the reflected left quotient bounds
$g$. Taking the minimum with one proves the first item.

Let
$f_{\rm lin}(x)=-\int_0^\infty e^{2\ii k\zeta}u(x+\zeta)\dd \zeta$.
The difference $f-f_{\rm lin}$ has integrand $\bar uf^2$.
Young's inequality for the kernel
$e^{-2\beta \zeta}\mathbf1_{\zeta>0}$ therefore gives
\[
 \|f-f_{\rm lin}\|_2
 \le\frac{M\alpha_\beta(t)^2}{2\beta}.
\]
Replacing $f$ by $f_{\rm lin}$ in $\log A$ consequently
costs at most $M^2\alpha_\beta(t)^2/(2\beta)$. The resulting
quadratic expression is
\[
 \int_0^\infty e^{2\ii k\zeta}
       \int_{2tz_0}^\infty\bar u(x)u(x+\zeta)\dd x\dd \zeta.
\]
By \eqref{eq:chirped-translation}, replacing the product
in the inner integral by
$e^{\ii \zeta x/(2t)}|u(x)|^2$ has error at most
\[
 M\left(2\sigma_K(t)+\frac{M\zeta}{\sqrt t}
                         +\frac{M\zeta^2}{4t}\right).
\]
After the ray substitution $x=2tz$, this replacement leaves
\[
 \int_0^\infty e^{2\ii k\zeta}
       \int_{z_0}^\infty e^{\ii \zeta z}|G_v(t,z)|^2\dd z\dd \zeta.
\]
Replacing $|G_v|^2$ by $\rho_v$ costs
$\varepsilon_K(t)/(2\beta)$. The remaining translation errors
are obtained from
\[
 \int_0^\infty e^{-2\beta \zeta}\dd \zeta=\frac1{2\beta},
 \quad
 \int_0^\infty \zeta e^{-2\beta \zeta}\dd \zeta=\frac1{4\beta^2},
 \quad
 \int_0^\infty \zeta^2e^{-2\beta \zeta}\dd \zeta=\frac1{4\beta^3}.
\]
Together with the earlier nonlinear error they give exactly
\eqref{eq:outer-error}. Fubini's theorem applies because
$\rho_v\in L^1$ and $\beta>0$. Finally,
\[
 \int_0^\infty e^{\ii(z+2k)\zeta}\dd \zeta=\frac{\ii}{z+2k},
\]
so the limiting expression is $\log\mathcal A_{v,z_0}(k)$.

For the left half-line the quadratic expression is
\[
 \int_0^\infty e^{2\ii k\zeta}
       \int_{-\infty}^{2tz_0}u(x)\bar u(x-\zeta)\dd x\dd \zeta.
\]
Apply \eqref{eq:chirped-translation} with displacement
$-\zeta$ and take the complex conjugate. The product is again
replaced by $e^{\ii \zeta x/(2t)}|u(x)|^2$ with the same error.
The identical integration therefore gives
$\log\mathcal B_{v,z_0}(k)$ and
\eqref{eq:outer-error}. All bounds are independent of
$z_0$ and $\alpha$. For fixed $\beta$ every error tends to
zero, proving \eqref{eq:outer-splitting}.
\end{proof}

To compare matrices near a spectral point at which a
transmission may be large, we use a relative factor. Define
\[
 D_{v,z_0}(k)=\diag(\mathcal A_{v,z_0}(k)^{-1},
                              \mathcal A_{v,z_0}(k)).
\]
The next result compares the original matrix with this
diagonal matrix without placing its possibly large entries
inside an error bound.

\begin{corollary}[The relative factors]\label{cor:relative}
For every strongly compact $K\subset L^2$,
\begin{enumerate}
\item At each fixed height $\beta>0$,
\begin{equation}\label{eq:relative-offcritical}
 \sup_{v\in K,z_0,\alpha\in\R}
 \|M_{v,+}(t,2tz_0;\alpha+\ii\beta)
                  D_{v,z_0}(\alpha+\ii\beta)^{-1}-I\|
 \longrightarrow0.
\end{equation}
The lower half-plane statement follows by the Dirac reduction.
\item There is a nonincreasing height $\beta(t)\to0$ for which
$h(t)/\beta(t)\to0$ and the error in
\eqref{eq:relative-offcritical} is $o(\beta(t))$.
\end{enumerate}
\end{corollary}

\begin{proof}
Set $c=A/\mathcal A$ at the given split and spectral point.
The splice identity \eqref{eq:half-splice} gives
\[
 M_+D^{-1}=
 \begin{pmatrix}
 c^{-1}/(1-fg)&cf\\
 gc^{-1}/(1-fg)&c
 \end{pmatrix}.
\]
Lemma~\ref{lem:outer-splitting} controls $\log c$, $f$,
and $g$ uniformly. When these quantities have modulus at most
$1/4$, the displayed matrix differs from $I$ by
$\lesssim|\log c|+|f|+|g|$. This proves fixed-height
convergence and explains why no transmission size is lost.

For an explicit decreasing height, define the error envelope
\[
 \omega_K(t)=\sup_{s\ge t}
       \max\{e_K(s),\sigma_K(s),\varepsilon_K(s),s^{-1/2}\}.
\]
It is positive, nonincreasing, and tends to zero. Set
$\beta(t)=\omega_K(t)^{1/8}$ for large $t$. Since
$1-e^{-2\beta}\gtrsim\beta$ for $0<\beta\le1$,
\[
 \alpha_\beta(t)\lesssim\omega_K(t)^{3/8},
 \qquad
 |\log c|\lesssim_M
  \omega_K(t)^{5/8}+\omega_K(t)^{7/8}
       +\omega_K(t)^{3/4}+\omega_K(t)^{13/8}.
\]
Every exponent exceeds $1/8$, so the relative error is
$o(\beta(t))$. Also $h(t)\le\omega_K(t)/2$, whence
$h(t)/\beta(t)\le\omega_K(t)^{7/8}/2\to0$.
\end{proof}

\subsection{Comparison with inverse problems on spectral intervals}
\label{sec:bands}

We construct inverse factors using only the initial reflection
on an interval around the stationary point. The jump of this
interval problem will agree with the original jump in its
interior. Their quotient can then be continued across the
interior and estimated by the Cauchy formula on a rectangle.
The upper and lower horizontal sides use the previous
fixed-height limit. For the vertical sides, averaging over
endpoints on which the outer transmission is bounded avoids
requiring such a bound inside the rectangle.

\subsubsection{Initial data on spectral intervals}

Fix $u_0\in L^2$, $u_0\ne0$. Let $P_j$ be orthogonal
projection onto the span of the first $j$ Hermite functions
in $L^2(\R)$, using the usual ordering by degree. Define
\begin{equation}\label{eq:hermite}
 j_n=\min\{j\ge n:\|P_ju_0-u_0\|_2<2^{-n}\},
 \qquad v_n=P_{j_n}u_0.
\end{equation}
Completeness of the Hermite basis ensures that $j_n$ exists.
Every $v_n$ is Schwartz, and
$\|v_n-u_0\|_2<2^{-n}$, $\|v_n\|_2\le M:=\|u_0\|_2$.
Let $r_n$, $\ell_n$, and $a_n$ denote its initial reflection,
entropy, and inverse transmission. Then
$K_0=\{u_0,v_1,v_2,\ldots\}$ is strongly compact in $L^2$,
and $\ell_n\to\ell$ in $L^1$ by Lemma~\ref{lem:direct}.
In particular the entropies are uniformly integrable and tight.

For a nonnegative $F\in L^1(\R)$, define its Poisson maximal
function by
\[
 \mathcal PF(v)=\sup_{y>0}\frac1\pi
       \int_\R\frac{yF(s)}{(s-v)^2+y^2}\dd s.
\]
Let $C_P$ be a constant in its standard weak-type estimate
$|\{\mathcal PF>T\}|\le C_P\|F\|_1/T$; see
\cite{Garnett} and Appendix~\ref{app:analytic}.
Fix a spectral width $\delta>0$ and a center $k_c\in\R$.
Choose admissible left and right endpoints from the sets
\begin{equation}\label{eq:good-edges}
 \begin{split}
 T_\delta&=1+\frac{4C_P\pi M^2}{\delta},\qquad
 E_{\delta,n}=\{\mathcal P\ell_n\le T_\delta\},\\
 J_{-,n}(k_c)&=[k_c-2\delta,k_c-\delta]\cap E_{\delta,n},\\
 J_{+,n}(k_c)&=[k_c+\delta,k_c+2\delta]\cap E_{\delta,n}.
 \end{split}
\end{equation}
Indeed, the trace formula gives $\|\ell_n\|_1\le\pi M^2$,
so the weak-type estimate yields
\[
 |E_{\delta,n}^c|\le\delta/4,
 \qquad |J_{\pm,n}(k_c)|\ge3\delta/4.
\]
Thus both endpoint sets have positive measure. For every selected
endpoint $v$, the Poisson formula for the transmission gives
\begin{equation}\label{eq:edge-outer}
 \sup_{y>0}|a_n(v+\ii y)|\le e^{T_\delta/2}.
\end{equation}
The analytic reflection gives the identical bound below the
real line.

Let $a\in J_{-,n}(k_c)$ and $b\in J_{+,n}(k_c)$ be real
endpoints. Define the inverse of the left entropy outer factor
and the modified interval reflection by
\begin{equation}\label{eq:band-data}
 \Delta_{n,a}(w)=
 \exp\left(\frac1{2\pi\ii}
       \int_{s<a}\frac{-\ell_n(s)}{s-w}\dd s\right),
 \qquad
 r_{n,a,b}(k)=r_n(k)\Delta_{n,a}(k)^{-2}\mathbf1_{(a,b)}(k).
\end{equation}
The factor $\Delta_{n,a}$ removes the phase contributed by
the entropy strictly to the left of the interval. Its boundary
value has modulus one on $(a,b)$, so the entropy of
$r_{n,a,b}$ is exactly $\ell_n\mathbf1_{(a,b)}$. For every
fixed $n$ the reflection has a strict gap, inherited from the
Schwartz reflection $r_n$. Proposition~\ref{prop:strict-gap}
therefore provides its normalized physical inverse factor,
denoted by $N_{n,a,b}$.

Evaluate this factor and the original factor $M_n$ of $S(t)v_n$
at the same time $t$ and spatial split $x=-4tk_c$. In both
half-planes set
\begin{equation}\label{eq:band-factor}
 M_{n,a,b}=N_{n,a,b}\Delta_{n,a}^{J},
 \qquad Q_{n,a,b}=M_nM_{n,a,b}^{-1}.
\end{equation}
Here $\Delta^J=\diag(\Delta,\Delta^{-1})$ above the line,
and the lower factor is fixed by the Dirac reduction. These
matrices depend only on $v_n$ and its scattering data at time
zero, followed by the known scattering evolution.

For a spectral center $k_c$, write
$\dd\mu_{k_c}(\zeta)=\dd \zeta/[\pi(1+(\zeta-k_c)^2)]$ for harmonic
measure based at $k_c\pm\ii$.

\begin{lemma}[Compactness and weighted bounds for the interval factors]
\label{lem:band-compact}
Fix $\delta>0$ and let $n,k_c,a,b$ vary as above.
\begin{enumerate}
\item The interval reflections $r_{n,a,b}$ have compact closure
in $X$. All their inverse potentials and their $L^2$ limits
belong to a strongly compact family $K_\delta\subset L^2$.
Every member has $L^2$ norm at most $M$.
\item Uniformly in $t\ge0$ and in the endpoints, the upper
factors
\[
 M_n/a_n,\qquad M_{n,a,b}/a_n,
 \qquad M_{n,a,b}^{-1}/a_n
\]
have $H^2(\mu_{k_c})$ norms $\lesssim_M1$. The reflected
lower factors divided by $a_n^\sharp$ obey the same bounds.
\end{enumerate}
\end{lemma}

\begin{proof}
First consider a sequence with bounded centers. Pass to a
subsequence on which the centers and the two endpoints converge,
and the full reflections converge in the compact family
$\{r_n:n\ge1\}$ in $X$. Uniform integrability of the entropies
implies $L^1$ convergence of the entropies with the moving
cutoffs $\mathbf1_{(-\infty,a)}$ and $\mathbf1_{(a,b)}$.
The standard weak-$(1,1)$ estimate for the Hilbert transform
then gives local convergence in measure of the boundary phases
of $\Delta_{n,a}$. This is precisely the outer-phase continuity
recorded in Appendix~\ref{app:analytic}. Multiplying by
these phases proves convergence of the interval reflections in
$X$.

If the centers instead tend to infinity in modulus, the
intervals $(a,b)$ leave every compact set, and tightness gives
$\|\ell_n\mathbf1_{(a,b)}\|_1\to0$. Since
$|r_{n,a,b}|^2\le\ell_n\mathbf1_{(a,b)}$, the reflections
converge to zero locally in measure independently of their
phases. Thus they again converge in $X$. These two cases prove
compactness of the closure. Lemma~\ref{lem:proper} gives
compactness of all inverse images, and
Proposition~\ref{prop:strict-gap} identifies the actual
inverse potentials defining $N_{n,a,b}$. The trace formula identifies the squared norm of each inverse
potential with $\pi^{-1}\int_a^b\ell_n\le M^2$, and
continuity gives the same bound for the limits.

We next place the full and interval factors under the same
transmission weight. Decompose $a_n=a_La_Ia_R$, where the
three outer factors are formed from the entropy restricted to
$(-\infty,a)$, $(a,b)$, and $(b,\infty)$, respectively.
Above the line, $\Delta_{n,a}=a_L^{-1}$, so
\begin{equation}\label{eq:common-weight}
 \frac{M_{n,a,b,+}}{a_n}
 =\frac{N_{n,a,b,+}}{a_I}
        \diag(a_L^{-2}a_R^{-1},a_R^{-1}).
\end{equation}
Both displayed scalar multipliers are analytic contractions.
The reflected lower identity has diagonal multipliers
$(a_R^\sharp)^{-1}$ and
$(a_L^\sharp)^{-2}(a_R^\sharp)^{-1}$ in that order.

The harmonic majorant from Lemma~\ref{lem:weighted} can
be bounded at height one using only mass. At any physical
split the half-line outer formulas give
$|A(k_c+\ii)B(k_c+\ii)|\le e^{M^2/2}$. Since $|a_n|\ge1$
in the upper half-plane,
\[
 H(k_c+\ii)
 \le\frac2{|1-fg|}=2|AB/a_n|\le2e^{M^2/2}.
\]
Each weighted matrix consequently has squared Hardy norm at
most $8e^{M^2/2}$. The interval inverse has mass at most $M$,
so the same bound applies to $N_{n,a,b}/a_I$. The analytic
contractions in \eqref{eq:common-weight} transfer it to
$M_{n,a,b}/a_n$.

For a determinant-one $2\times2$ matrix, the inverse is its
adjugate. Because the adjugate is linear in the entries in
this dimension,
\[
 M_{n,a,b}^{-1}/a_n=\operatorname{adj}(M_{n,a,b}/a_n).
\]
The adjugate preserves the Frobenius norm. This proves the
inverse estimate under the same weight; reflection proves all
the lower estimates. No common strict gap enters these bounds.
\end{proof}

\subsubsection{Comparison at height \texorpdfstring{$t^{-1/2}$}{t to the power -1/2}}

We use normalized endpoint averages: for a measurable set $E$
of positive finite measure, $\fint_E F=|E|^{-1}\int_E F$.
The comparison at a height which tends to zero is the last
estimate needed before defining the phase from initial data.

\begin{proposition}[Comparison of the original and interval factors]
\label{prop:band-comparison}
For each fixed $\delta>0$ and $h=(2\sqrt t)^{-1}$,
\begin{enumerate}
\item The quotient defined in \eqref{eq:band-factor}
satisfies
\begin{equation}\label{eq:band-comparison}
 \sup_{n,k_c}
 \fint_{J_{-,n}(k_c)\times J_{+,n}(k_c)}
       \|Q_{n,a,b}(k_c-\ii h)-I\|\dd a\dd b
 \longrightarrow0\qquad(t\to\infty).
\end{equation}
\end{enumerate}
\end{proposition}

\begin{proof}
We first obtain an exact rectangle formula for each fixed
smooth approximant. On $(a,b)$ the boundary values of
$\Delta_{n,a}$ agree and are unimodular. Substitution of
$r_{n,a,b}=r_n\Delta_{n,a}^{-2}$ in the jump matrix shows
that $M_{n,a,b}$ and $M_n$ have the same jump there.
Their quotient $Q_{n,a,b}$ therefore has matching upper and
lower boundary values on $(a,b)$.

For fixed $n$, the Schwartz reflection has a global strict gap,
so $a_n$ is bounded. The common weighted $H^2$ estimates in
Lemma~\ref{lem:band-compact} then place the quotient in
local unweighted $H^1$ on both sides of the interval.
The Hardy gluing theorem gives analytic continuation through
$(a,b)$; see Appendix~\ref{app:analytic}.

Fix a positive height $\beta$ and suppose $h<\beta$.
Let $w=k_c-\ii h$. For $\epsilon>0$ small, the point $w$
is inside
$R_\epsilon=[a+\epsilon,b-\epsilon]\times[-\beta,\beta]$.
The Cauchy formula gives
\[
 Q_{n,a,b}(w)-I
 =\frac1{2\pi\ii}\int_{\partial R_\epsilon}
          \frac{Q_{n,a,b}(\zeta)-I}{\zeta-w}\dd\zeta.
\]
We justify the limit $\epsilon\downarrow0$ before estimating
the sides uniformly. Dividing an upper harmonic-measure
$H^2(\mu_{k_c})$ function by $\zeta-k_c+\ii$ gives a line
$H^2$ function; below the line the corresponding divisor is
$\zeta-k_c-\ii$. In particular, define the upper matrix
functions
\[
 F_{\rm orig}(\zeta)=
   \frac{M_{n,+}(\zeta)}{a_n(\zeta)(\zeta-k_c+\ii)},
 \qquad
 F_{a,b}(\zeta)=
   \frac{M_{n,a,b,+}(\zeta)}{a_n(\zeta)(\zeta-k_c+\ii)}.
\]
Their line $H^2$ norms are $\lesssim_M1$.
Use the reflected definitions for lower points.

For line $H^2$ the vertical evaluation estimate is
\[
 \int_0^\infty|F(v+\ii y)|^2\dd y
       \le\tfrac12\|F\|_{L^2(\R)}^2,
 \qquad v\in\R.
\]
This is the standard Laplace-transform estimate, obtained by
the Carleman kernel $(\xi+\eta)^{-1}$; its exact application
and normalization are recorded in
Appendix~\ref{app:analytic}. Strong continuity of real
translations in $H^2$, followed by this vertical estimate,
shows that the traces on $a+\epsilon$ and $b-\epsilon$
converge in $L^2(0,\beta)$ to the endpoint traces.
Their products converge in $L^1(0,\beta)$, and the fixed-$n$
bound on the outer multiplier permits the same conclusion
for the quotient. The horizontal traces converge because
they are at a fixed nonreal height. We therefore obtain the
Cauchy formula on the full rectangle
$R=[a,b]\times[-\beta,\beta]$, with no additional term on
the real interval.

We now estimate its horizontal sides. Corollary~\ref{cor:relative}
applies uniformly to both compact families $K_0$ and $K_\delta$.
At the split $z_0=-2k_c$, their limiting diagonal factors
agree after the multiplication by $\Delta_{n,a}^J$.
Indeed, if $\mathcal A_I$ and $\mathcal B_I$ denote the
limiting transmissions of the interval inverse, its entropy
is $\ell_n\mathbf1_{(a,b)}$. The physical right half-line
corresponds under $z=-2k$ to $k<k_c$, and $a<k_c<b$.
Thus
\[
 \mathcal A_{\mathrm{full}}=a_L\mathcal A_I,
 \qquad \mathcal B_{\mathrm{full}}=\mathcal B_Ia_R,
 \qquad D_I\Delta_{n,a}^{J}=D_{\mathrm{full}}.
\]
Writing
$M_n=(I+E_n)D_{\mathrm{full}}$ and
$M_{n,a,b}=(I+E_I)D_{\mathrm{full}}$ gives the exact cancellation
\[
 Q_{n,a,b}=(I+E_n)(I+E_I)^{-1}.
\]
Here $E_n$ and $E_I$ are the relative errors at the chosen
nonreal point. Hence the quantity
\[
 \mathcal E_\delta(t,\beta)=
 \sup_{\substack{n,k_c,a,b\\ \alpha\in[a,b]}}
 \max_{\pm}\|Q_{n,a,b}(\alpha\pm\ii\beta)-I\|
\]
tends to zero for each fixed $\beta>0$. Uniformity includes
every center and every allowed endpoint, by the two compact
families and the suprema over ray splits and real spectral
centers in Corollary~\ref{cor:relative}. If $h\le\beta/2$,
the Cauchy denominator on either horizontal side has modulus
at least $\beta/2$. Since $b-a\le4\delta$, the sum of
their contributions is
$\lesssim_\delta\mathcal E_\delta(t,\beta)/\beta$.

For a vertical side, the determinant-one identity gives above
the line
\[
 Q_{n,a,b}(\zeta)
 =a_n(\zeta)^2(\zeta-k_c+\ii)^2
        F_{\rm orig}(\zeta)\operatorname{adj}(F_{a,b}(\zeta)).
\]
Below the line the identity uses $(a_n^\sharp)^2$ and the
lower divisor. On the selected endpoint lines,
\eqref{eq:edge-outer} bounds the outer multiplier.
When $0<\beta\le1$, the divisor is bounded on the whole
rectangle by a constant depending only on $\delta$.
Also the denominator obeys $|\zeta-w|\ge\delta$ on both
vertical sides.

The line $H^2$ contraction on horizontal slices gives the
area estimate for the original factor
\[
 \int_{J_{\pm,n}}\int_0^\beta
        \|F_{\rm orig}(v+\ii y)\|_F^2\dd y\dd v
 \lesssim_M\beta.
\]
At its own endpoint, every interval factor has the vertical
estimate
\[
 \int_0^\beta\|F_{a,b}(v+\ii y)\|_F^2\dd y
 \lesssim_M1,
 \qquad v=a\ \hbox{or}\ v=b.
\]
The lower factors satisfy the same two bounds. Average the
products over $a\in J_{-,n}$ and $b\in J_{+,n}$, and apply
Cauchy--Schwarz. The original factor contributes the square
root of its area bound; the interval factor contributes a
bounded vertical norm. Since each endpoint set has measure
at least $3\delta/4$, the normalized average of the vertical
contributions is $\lesssim_{\delta,M}\sqrt\beta$. The
constant matrix $I$ in the integrand contributes
$\lesssim_\delta\beta$, which is absorbed by $\sqrt\beta$.
We have proved, for $0<\beta\le1$ and $h\le\beta/2$,
\begin{equation}\label{eq:rectangle-bound}
 \sup_{n,k_c}\fint_{J_{-,n}\times J_{+,n}}
       \|Q_{n,a,b}(k_c-\ii h)-I\|\dd a\dd b
 \lesssim_{\delta,M}
       \frac{\mathcal E_\delta(t,\beta)}{\beta}+\sqrt\beta.
\end{equation}

Choose heights decreasing to zero slowly enough that
$\mathcal E_\delta(t,\beta(t))=o(\beta(t))$ and
$h(t)/\beta(t)\to0$. Such a choice follows from convergence
at every fixed height: for a sequence $\beta_m\downarrow0$,
choose successively increasing times after which
$\mathcal E_\delta(t,\beta_m)\le\beta_m^2$ and
$h(t)\le\beta_m^2$, and use $\beta_m$ up to the next such
time. Substitution in \eqref{eq:rectangle-bound} proves
\eqref{eq:band-comparison}.
\end{proof}

\subsection{Construction of the phase from the initial datum}
\label{sec:phase-construction}

We now complete the proof of Theorem~\ref{thm:generalized}.  The
stationary-scale analysis already approximates the solution by a
coefficient of the constant-reflection model, multiplied by a unit
complex number.  That number is read from an original inverse factor
of the evolved solution.  Our first task is to replace it by the
factors constructed from the initial scattering data in
Subsection~\ref{sec:bands}.  The averaged comparison of those matrices is sufficient,
because the matrix entry used to extract the phase stays away from
zero outside a set carrying arbitrarily little transmission density.

We retain the smooth initial approximations $v_n$ from
\eqref{eq:hermite}, their scattering data $r_n,\ell_n$, and the
endpoint sets $J_{\pm,n}(k_c)$ from \eqref{eq:good-edges}.
The parameter $\delta>0$ specifies the width of the spectral interval
being retained.  We first keep this width fixed.  Only after proving
convergence for each fixed width will we select a width depending on
time.

\subsubsection{Extraction of a unit phase}

For a center $k_c\in\R$ and time $t>0$, put
\[
             z=-2k_c,\qquad h=\frac1{2\sqrt t},\qquad
             D_c=e^{\ii2tk_c^2J}.
\]
The coordinate $z$ is the corresponding physical ray.  The height
$h$ is the stationary spectral scale, and the diagonal matrix $D_c$
removes the constant oscillation at the center of that scale.
For endpoints $a\in J_{-,n}(k_c)$ and $b\in J_{+,n}(k_c)$, define
\begin{align}
 C_{n,a,b}(t,k_c)
      &=D_c^{-1}M_{n,a,b,-}(k_c-\ii h)D_c,\notag\\
 L_{n,a,b}(t,k_c)
      &=C_{n,a,b}(t,k_c)(C^W_{r_n(k_c)})^{-1},
                                   \label{eq:band-phase}\\
 \gamma_{n,a,b}(t,k_c)
      &=\unit\big((L_{n,a,b}(t,k_c))_{11}\big)^2,\notag\\
 \Gamma_{\delta,n}(t,k_c)
      &=\unit\left(
          \fint_{J_{-,n}(k_c)\times J_{+,n}(k_c)}
             \gamma_{n,a,b}(t,k_c)\dd a\dd b\right).\notag
\end{align}
Here $M_{n,a,b,-}$ is the lower factor in
\eqref{eq:band-factor}, and
$C_b^W=M^W_{b,-}(0,-\ii)$ is the model matrix defined after
\eqref{eq:block-variables}.  Thus $L_{n,a,b}$ compares two
matrices in the same gauge.  The square in the definition of
$\gamma_{n,a,b}$ is required by the relation
$p=d^2p_b e^{\ii y^2}$ in Lemma~\ref{lem:block}.
The final normalization turns the mean over endpoints into a unit
complex number, including when the mean is zero through the convention
$\unit(0)=1$.  Every quantity in \eqref{eq:band-phase} is obtained
from the initial datum $v_n$, the scalar outer factor, and the inverse
problem for the retained spectral interval.

\begin{proposition}\label{prop:fixed-width}
Fix $\delta>0$, and define
\[
 A_{\delta,n}(t,z)=
             \Gamma_{\delta,n}(t,-z/2)
                      \frac{p_{r_n(-z/2)}}{\sqrt2}.
\]
Then:
\begin{enumerate}
\item The approximation is uniform over the initial approximations:
\begin{equation}\label{eq:fixed-width}
 \varepsilon_\delta(t):=
       \sup_n\|G_{v_n}(t)-A_{\delta,n}(t)\|_2
                                                 \longrightarrow0.
\end{equation}
\end{enumerate}
\end{proposition}

\begin{proof}
Let $C_n^0(t,k_c)$ be the gauged original lower factor of $S(t)v_n$
at $k_c-\ii h$, and put
$L_n^0=C_n^0(C^W_{r_n(k_c)})^{-1}$ and
$\gamma_n^0=\unit((L_n^0)_{11})^2$.  Thus
$\gamma_n^0(t,k_c)$ equals the phase
\eqref{eq:original-phase} at the ray $z=-2k_c$.
The compactness of $K_0=\{u_0,v_1,v_2,\ldots\}$ and
Proposition~\ref{prop:critical} imply
\[
 \sup_n\left\|G_{v_n}(t)-
       \gamma_n^0(t,-\cdot/2)
                  \frac{p_{r_n(-\cdot/2)}}{\sqrt2}\right\|_2
                                                       \longrightarrow0.
\]
It remains to compare $\Gamma_{\delta,n}$ and $\gamma_n^0$.
The appropriate finite measure for this comparison is
\[
                  \dd\nu_n(k)=\frac{\ell_n(k)}\pi\dd k,
                  \qquad \nu_n(\R)=\|v_n\|_2^2\le M^2.
\]
This weight is precisely the squared modulus of the model coefficient.

We explain the lower bound needed before normalizing a matrix entry.
For every $\eta>0$, the compactness of $\{\ell_n\}$ in $L^1$ allows
us to discard spectral tails and sets of small Lebesgue measure with
total $\nu_n$ measure at most $\eta$, uniformly in $n$.
The sets where $|r_n|>c$ or $|s_n|<\sigma$, with $s_n=1/a_n$, also
carry arbitrarily little $\nu_n$ measure when $c\uparrow1$ and
$\sigma\downarrow0$.  Indeed, on the boundary
\[
                  |s_n|^2=1-|r_n|^2=e^{-\ell_n},
\]
so both exclusions are high level sets of the uniformly integrable
functions $\ell_n$.  The integrated $L^2$ identity
\eqref{eq:block-energy} bounds the Lebesgue measure of centers
with a large rescaled $L^2([-1,1])$ norm.  Finally, the estimate
\eqref{eq:triangular-freeze}, applied uniformly to the compact
family $K_0$, makes the set with a large local symbol error have
vanishing measure as $t\to\infty$.

On the remaining centers, the sequential conclusion of
Lemma~\ref{lem:block} gives fixed constants $B_*,C_*>0$ such
that, for all sufficiently large $t$,
\[
 \|C_n^0\|\le B_*,\qquad
 \|(C^W_{r_n(k_c)})^{-1}\|\le C_*,\qquad
                         |(L_n^0)_{11}|\ge\frac12.
\]
To justify the uniformity, a sequence violating one of these bounds
would have bounded rescaled $L^2([-1,1])$ norms, reflections in a compact subdisk,
nonzero limiting reciprocal transmission, and the spectral variation tending
to zero.  The lemma would then give
$L_n^0\to\diag(d,\bar d)$ with $|d|=1$, contradicting the violation.
The model inverse is bounded by continuity on the retained compact
subdisk.  The discarded centers have $\nu_n$ measure as small as
desired by the preceding choices.

At a remaining center, set
$R_{n,a,b}=D_c^{-1}Q_{n,a,b}(k_c-\ii h)D_c$.  This matrix relates
the two factors by $C_n^0=R_{n,a,b}C_{n,a,b}$.  As $D_c$ is unitary,
$\|R_{n,a,b}-I\|=\|Q_{n,a,b}-I\|$.  When this norm is less than
$1/2$, the inverse estimate $\|R_{n,a,b}^{-1}\|_{\mathrm{op}}\le2$ gives
\[
 \|L_{n,a,b}-L_n^0\|
       \le2B_*C_*\|Q_{n,a,b}-I\|.
\]
For complex numbers $x,y$ with $x\ne0$,
\[
                 |\unit(y)-\unit(x)|
                              \le\frac{2|y-x|}{|x|}.
\]
This bound is valid at $y=0$ with our convention.  Applying it to the
$(1,1)$ entries, and then using $|w_1^2-w_2^2|\le2|w_1-w_2|$ for unit
complex numbers, yields
\[
       |\gamma_{n,a,b}-\gamma_n^0|
                  \le16B_*C_*\|Q_{n,a,b}-I\|.
\]
If $\|Q_{n,a,b}-I\|\ge1/2$, the left-hand side is at most
$2\le4\|Q_{n,a,b}-I\|$.  Thus one fixed constant controls the phase
difference in both cases.

If $Z$ is the mean of unit complex numbers and $\gamma$ is unit,
then $|\unit(Z)-\gamma|\le2|Z-\gamma|$.  Also $|Z-\gamma|\le2$.
With the endpoint mean in \eqref{eq:band-phase}, these two
elementary inequalities imply
\begin{align*}
 |\Gamma_{\delta,n}-\gamma_n^0|^2
 &\le8\fint_{J_{-,n}\times J_{+,n}}
                       |\gamma_{n,a,b}-\gamma_n^0|\dd a\dd b\\
 &\lesssim_{B_*,C_*}
       \fint_{J_{-,n}\times J_{+,n}}\|Q_{n,a,b}-I\|\dd a\dd b.
\end{align*}
Integrating on the retained centers against $\nu_n$ and using
Proposition~\ref{prop:band-comparison} makes this contribution
tend to zero uniformly in $n$.  The discarded centers contribute
at most four times their $\nu_n$ measure.  We first let $t\to\infty$
and then make that measure tend to zero.  Hence
\[
       \sup_n\|\Gamma_{\delta,n}(t)-\gamma_n^0(t)\|_{L^2(\nu_n)}
                                                       \longrightarrow0.
\]
Finally $|p_{r_n(k)}|^2=\ell_n(k)/\pi$ and $z=-2k$ show exactly
that
\[
 \left\|[\Gamma_{\delta,n}(t,-\cdot/2)
                  -\gamma_n^0(t,-\cdot/2)]
                         \frac{p_{r_n(-\cdot/2)}}{\sqrt2}\right\|_2^2
       =\|\Gamma_{\delta,n}(t)-\gamma_n^0(t)\|_{L^2(\nu_n)}^2.
\]
Together with the original phase approximation, this proves
\eqref{eq:fixed-width}.
\end{proof}

\subsubsection{Choice of the approximation and the spectral width}

The large-time assertion just proved is uniform in $n$ for each fixed
width.  To transfer it to $u_0$, we need a finite approximation index
whose error remains small at the particular time under consideration.
The following quantitative version of the standard local stability
estimate provides such an index using the initial norm alone.

\begin{lemma}\label{lem:stability}
For each $R>0$ there are constants $\tau_R>0$ and $K_R\ge1$, determined
only by $R$ and the free and inhomogeneous Strichartz constants, such
that:
\begin{enumerate}
\item If $\|u_0\|_2,\|v_0\|_2\le R$, then for every $t\ge0$,
\begin{equation}\label{eq:stability}
 \|S(t)u_0-S(t)v_0\|_2
       \le\Lambda_R(t)\|u_0-v_0\|_2,
       \qquad \Lambda_R(t)=K_R^{\,1+t/\tau_R}.
\end{equation}
\end{enumerate}
\end{lemma}

\begin{proof}
For an interval $I=[t_0,t_1]$, equip
$X_I=L^\infty(I;L^2)\cap L^8(I;L^4)$ with the maximum of its two
norms.  The free and retarded Strichartz estimates
\cite[Theorem~1.2, (5) and (7)]{KeelTao1998}, for the admissible pairs
$(\infty,2)$ and $(8,4)$ in one dimension, give constants $C_S\ge1$
and $C_D>0$ such that
\[
 \|e^{\ii(t-t_0)\partial_x^2}f\|_{X_I}\le C_S\|f\|_2,
 \qquad
 \left\|\int_{t_0}^t e^{\ii(t-s)\partial_x^2}H(s)\dd s\right\|_{X_I}
                       \le C_D\|H\|_{L^{8/7}(I;L^{4/3})}.
\]
The length factor in the cubic term is important here.  H\"older's
inequality in space puts the product in $L^{4/3}$, and the embedding
$L^{8/3}(I)\subset L^{8/7}(I)$ gives the factor $|I|^{1/2}$:
\[
 \|2|u|^2u\|_{L^{8/7}L^{4/3}}
                         \le2|I|^{1/2}\|u\|_{L^8L^4}^3.
\]
The pointwise cubic difference estimate gives in the same way
\[
 \|2|u|^2u-2|v|^2v\|_{L^{8/7}L^{4/3}}
 \le4|I|^{1/2}
       (\|u\|_{L^8L^4}^2+\|v\|_{L^8L^4}^2)\|u-v\|_{L^8L^4}.
\]
Let $B=2C_SR$, choose
$\tau_R=(16C_DB^2)^{-2}$, and put $K_R=2C_S$.
Mass conservation and the usual local contraction on the ball
$\|u\|_{X_I}\le B$ give this bound for both solutions on every
interval of length at most $\tau_R$.  Subtracting their Duhamel
formulas, the coefficient of $\|u-v\|_{X_I}$ is at most
\[
                     8C_D\tau_R^{1/2}B^2=\frac12.
\]
Absorption therefore bounds the difference on that interval by
$K_R$ times its incoming $L^2$ difference.  At most
$1+t/\tau_R$ such intervals are needed to reach time $t$.
Iteration proves \eqref{eq:stability}.
\end{proof}

We next select the width by comparing only the approximations
constructed from initial scattering data.  This avoids making a
choice based on an unknown error relative to the evolved solution.

\begin{proof}[Proof of Theorem~\ref{thm:generalized}]
For $u_0=0$, mass conservation gives the conclusion with phase one.
Assume henceforth that $M=\|u_0\|_2>0$ and put $R=M+1$.
For $t\ge16$, define the finite integer
\begin{equation}\label{eq:fast-index}
 N(t)=1+\max\left\{\lceil t\rceil,
           \left\lceil\log_2(t\Lambda_R(t))\right\rceil\right\}.
\end{equation}
The approximation error in \eqref{eq:hermite} is smaller than
$2^{-N(t)}$, so Lemma~\ref{lem:stability} and the unitary ray
normalization give
\[
                  \|G_{v_{N(t)}}(t)-G_{u_0}(t)\|_2\le t^{-1}.
\]

The model coefficient is also stable under these initial
approximations.  The map $b\mapsto p_b$ is continuous on $|b|<1$,
including at $b=0$.  Lemma~\ref{lem:direct} gives
$r_n\to r$ locally in measure and $\ell_n\to\ell$ in $L^1$.
Since $|p_{r_n}|^2=\ell_n/\pi$, these squared moduli are uniformly
integrable and have uniformly small tails.  Continuity yields local
convergence in measure of the coefficients; applying Vitali's theorem
on bounded intervals and then controlling the tails gives
\begin{equation}\label{eq:weber-L2}
                  p_{r_n}\longrightarrow p_r
                                      \quad\hbox{in }L^2(\R_k).
\end{equation}

For each fixed $\delta>0$, form the candidate using the exact
amplitude of $u_0$:
\begin{equation}\label{eq:candidate}
 B_\delta(t,z)=\Gamma_{\delta,N(t)}(t,-z/2)
                                  \frac{p_{r(-z/2)}}{\sqrt2}.
\end{equation}
Its squared modulus is $\rho(z)$.  The triangle inequality gives
\begin{equation}\label{eq:candidate-error}
 \|G_{u_0}(t)-B_\delta(t)\|_2
 \le t^{-1}+\varepsilon_\delta(t)
                    +\|p_{r_{N(t)}}-p_r\|_{L^2_k}\longrightarrow0.
\end{equation}
The last term has no extra numerical factor: the Jacobian
$\dd z=2\dd k$ cancels the square of $1/\sqrt2$ in the amplitude.
This limit holds for every fixed width.

To let the width decrease, set $\delta_m=2^{-m}$ and
$B_m=B_{\delta_m}$ for integers $m\ge1$.  At time $t\ge16$, consider
only $1\le m\le c(t)$, where
$c(t)=\lfloor\tfrac14\log_2t\rfloor$.  Call such an index admissible
if
\[
                         \|B_m(t)-B_1(t)\|_2\le\frac1m,
\]
and choose $m(t)$ to be the largest admissible index.  This is a
finite test and index one always passes it.  For each fixed $m$,
\eqref{eq:candidate-error} implies that it passes for all
sufficiently large $t$.  Thus $m(t)\to\infty$.  Moreover,
\[
              \delta_{m(t)}\longrightarrow0,
              \qquad t\delta_{m(t)}^2\ge\sqrt t.
\]
The second inequality follows from the upper bound $m(t)\le c(t)$;
it ensures that the selected interval remains larger than the
stationary spectral scale.

Define the desired unit phase for $t\ge16$ and $\rho(z)>0$ by
\begin{equation}\label{eq:final-phase}
 \Phi_{u_0}(t,z)=\Gamma_{\delta_{m(t)},N(t)}(t,-z/2)
                                      \unit(p_{r(-z/2)}).
\end{equation}
Set it equal to one when $t<16$ or $\rho(z)=0$.  On all of $\R$,
$\Phi_{u_0}(t)\sqrt\rho=B_{m(t)}(t)$, and the admissibility test
gives
\[
 \|G_{u_0}(t)-\Phi_{u_0}(t)\sqrt\rho\|_2
 \le\|G_{u_0}(t)-B_1(t)\|_2+\frac1{m(t)}
                                                       \longrightarrow0.
\]

We also verify measurability, since no pointwise continuity of the
rough reflection was assumed.  The selected endpoint sets are
measurable: in the definition of the Poisson maximal function, the
supremum can be taken over positive rational heights, because the
Poisson integral is continuous in its height.  For fixed $n$, the
reflection $r_n$ is smooth and has a strict gap.  Moving the center
and endpoints changes the interval reflection in $L^2$ by dominated
convergence and the scalar outer representation.  The corresponding
jumps converge strongly as multiplication operators on $L^2$.
The inverse operators have a common bound for this fixed $n$ from
the same reflection gap.  The resolvent identity consequently gives
strong continuity of their solutions and continuity of evaluation at
nonreal spectral points.  The time oscillation is strongly continuous
as well.  Hence the entries used in \eqref{eq:band-phase}, their
measurable endpoint restrictions, and their normalized integrals are
measurable.  The explicit integer $N(t)$ is measurable.  Tonelli's
theorem makes every norm in the finite admissibility test measurable,
so $m(t)$ is measurable.  Formula \eqref{eq:final-phase} therefore
defines a jointly measurable representative from $u_0$ alone.

The squared-amplitude assertion \eqref{eq:amplitude} follows from
\eqref{eq:uniform-amplitude} for the singleton $\{u_0\}$.
The amplitude assertion follows in turn from the pointwise inequality
$(\sqrt a-\sqrt b)^2\le|a-b|$ for $a,b\ge0$:
\[
                    \||G_{u_0}(t)|-\sqrt\rho\|_2^2
                       \le\||G_{u_0}(t)|^2-\rho\|_1\longrightarrow0.
\]
\end{proof}

\begin{remark}
Different admissible regularizations may give different unit phases at
a finite time.  Their asymptotic comparison is nevertheless fixed by
the theorem.  If $\Phi^{(1)}$ and $\Phi^{(2)}$ are two such phases,
then
\[
 \|(\Phi^{(1)}(t)-\Phi^{(2)}(t))\sqrt\rho\|_2
 \le\|G(t)-\Phi^{(1)}(t)\sqrt\rho\|_2
       +\|G(t)-\Phi^{(2)}(t)\sqrt\rho\|_2\longrightarrow0.
\]
Thus the weighted asymptotic phase is independent of this choice.
\end{remark}

\subsection{Comparison with the classical logarithmic phase}
\label{sec:localized}

For localized initial data, both our construction and the classical
formula approximate the same physical ray profile with the same
amplitude.  The comparison is therefore made in the space weighted by
that amplitude.  This is also the correct interpretation at spectral
points where the amplitude vanishes: a phase there has no effect on
the approximation.

\begin{proof}[Proof of Theorem~\ref{thm:localized}]
Let $u_0\in H^{1,1}(\R)$, and let $W=W_{u_0}$ be the coefficient
in \eqref{eq:classical-ray-coefficient}.  Since $|W|^2=\rho$, the
triangle inequality, Theorem~\ref{thm:generalized}, and
Lemma~\ref{lem:classical-ray} give
\begin{align*}
 &\|[\Phi_{u_0}(t)-e^{-\ii\rho\log t}\unit(W)]\sqrt\rho\|_2\\
 &\qquad\le\|G_{u_0}(t)-\Phi_{u_0}(t)\sqrt\rho\|_2
                  +\|G_{u_0}(t)-e^{-\ii\rho\log t}W\|_2
                                                   \longrightarrow0.
\end{align*}
The exact Fourier conclusion for Schwartz initial data is
Lemma~\ref{lem:smoothstrong}.
\end{proof}

\begin{remark}\label{rem:phase-normalization}
The two reflection conventions can be checked directly at the level
of the phase.  The gamma functional equation, the formula for $p_b$
in \eqref{eq:weber-coeff}, and \eqref{eq:dz-conversion} give
\[
 \unit(W(z))=\unit(p_{r(-z/2)})e^{\ii\cT L(z)}
                                  \frac{\overline{a(-z/2)}}{a(-z/2)}
                         \qquad(\rho(z)>0).
\]
Hence the selected $\Gamma$ in \eqref{eq:final-phase} converges,
in the weighted sense of Theorem~\ref{thm:localized}, to
$e^{-\ii\rho\log t+\ii\cT L}\overline a/a$ at $k=-z/2$.
The transmission ratio in this expression is the phase contribution
introduced by the change of scattering normalization.
\end{remark}

\section{Failure of fixed-profile modified scattering}
\label{sec:failure-proof}

\subsection{Separated potentials and the asymptotic phase}
\label{sec:far-bumps}

The construction of Theorem~\ref{thm:failure} rests on a difference
between the size of an initial perturbation and the size of the phase
which it produces at large times.  Translating a potential through a
distance $X$ multiplies its reflection coefficient by an oscillation
of frequency $X$.  The transform $\cT$ occurring in the classical
asymptotic coefficient acts on such an oscillation with a factor of
size $\log X$.  We shall therefore add a translated potential of size
$\kappa/\log X$, where $\kappa>0$ is fixed.  Its initial norm tends to
zero, whereas its contribution to the asymptotic phase remains
nonzero.  The estimates below make this observation precise and
allow $X$ to be chosen after all the preceding finite-time
requirements have been fixed.

\subsubsection{Scattering of separated potentials}
Fix a real-valued, nonzero function $p\in C_c^\infty(0,1)$.
For $X>1$ let $p_X(x)=p(x-X)$; thus $X$ specifies the location of
the new potential, and $\epsilon p_X$ has amplitude $\epsilon$.
We continue to use the scattering matrices and coefficients of
\eqref{eq:classical-scattering-matrix}.  All identities below are on the
real spectral axis.

\begin{lemma}[Composition and translation]
\label{lem:transfer}
The scattering matrices have the following properties.
\begin{enumerate}
\item If $u_l,u_r\in L^2$ are compactly supported and satisfy
$\sup\operatorname{supp}u_l<\inf\operatorname{supp}u_r$, then
\begin{equation}\label{eq:scattering-composition}
 \widetilde T_{u_l+u_r}=\widetilde T_{u_l}\widetilde T_{u_r}.
\end{equation}
\item Define $E(X,z)=e^{\ii XzJ/2}$, the free propagation matrix
across a distance $X$.  Translation gives
\begin{equation}\label{eq:scattering-translation}
 \widetilde T_{p_X}=E(-X,z)\widetilde T_pE(X,z),\qquad
 \widetilde a_{p_X}=\widetilde a_p,\qquad
 \widetilde b_{p_X}=e^{\ii Xz}\widetilde b_p.
\end{equation}
In particular, if $X$ lies beyond the support of $u$, then
\begin{align}\label{eq:scattering-insertion}
 \widetilde a_{u+\epsilon p_X}
 &=\widetilde a_u\widetilde a_{\epsilon p}
   +\overline{\widetilde b_u}\,e^{\ii Xz}\widetilde b_{\epsilon p},\notag\\
 \widetilde b_{u+\epsilon p_X}
 &=\widetilde b_u\widetilde a_{\epsilon p}
   +\overline{\widetilde a_u}\,e^{\ii Xz}\widetilde b_{\epsilon p}.
\end{align}
\end{enumerate}
\end{lemma}

\begin{proof}
In the interaction representation, the matrix which transports a
solution from the left of the potential to its right is
$\widetilde T_u^{-1}$.  Crossing $u_l$ and then $u_r$ therefore gives
$\widetilde T_{u_r}^{-1}\widetilde T_{u_l}^{-1}$; inversion proves
\eqref{eq:scattering-composition}.  If $\psi$ is either normalized Jost solution
for $p$, then $\psi(x-X,z)E(X,z)$ is the corresponding solution
for $p_X$, since its free normalization is $e^{\ii xzJ/2}$.
The connection between the two Jost solutions is consequently
conjugated by $E(X,z)$.  Its diagonal entries are unchanged and
its lower-left entry acquires $e^{\ii Xz}$, which proves
\eqref{eq:scattering-translation}.  The two entries in \eqref{eq:scattering-insertion} follow
by multiplication of the matrices in \eqref{eq:scattering-composition}.
\end{proof}

We need the expansion for the small, untranslated potential
$\epsilon p$ with control of its spectral decay.  To state this
control, let $j,k$ denote nonnegative integers and, for an
integer $N\ge0$, set
\[
 \langle z\rangle=(1+z^2)^{1/2},\qquad
 s_N(f)=\max_{j+k\le N}\sup_{z\in\R}
             \langle z\rangle^j|\partial_z^kf(z)|,
 \qquad
 \|f\|_{C_b^N}=\max_{k\le N}\sup_{z\in\R}|\partial_z^kf(z)|.
\]
The seminorm $s_N$ measures bounded derivatives together with
polynomial decay.  In particular, $f\in\Sch$ precisely when all
these seminorms are finite.  The first-order reflected amplitude is
\begin{equation}\label{eq:linear-reflection}
 \beta(z)=-\int_\R\overline{p(y)}e^{\ii yz}\dd y.
\end{equation}
This is the Fourier transform of $-\bar p$ with the sign imposed
by our scattering convention.

\begin{lemma}[Expansion and continuity of the scattering data]
\label{lem:small-scattering}
For the fixed function $p$, the following estimates hold.
\begin{enumerate}
\item There is $\epsilon_0>0$ such that, for every integer $N\ge0$
and $|\epsilon|\le\epsilon_0$,
\begin{equation}\label{eq:small-transmission}
 \|\widetilde a_{\epsilon p}-1\|_{C_b^N}\lesssim_N\epsilon^2,
 \qquad
 s_N(\widetilde b_{\epsilon p}-\epsilon\beta)\lesssim_N|\epsilon|^3,
 \qquad
 s_N(L_{\epsilon p})\lesssim_N\epsilon^2.
\end{equation}
The reciprocal $\widetilde a_{\epsilon p}^{-1}$ is bounded in every
$C_b^N$, uniformly for these $\epsilon$.
\item If $u,\widetilde u\in L^1\cap L^2$ and
$\|u\|_1,\|\widetilde u\|_1\le M$, then
\begin{equation}\label{eq:insertion-log-density}
 \sup_{z\in\R}\|\widetilde T_u(z)-\widetilde T_{\widetilde u}(z)\|_{\rm op}
 \lesssim_M\|u-\widetilde u\|_1,
\end{equation}
where $\|\cdot\|_{\rm op}$ is the operator norm on $2\times2$
matrices.  Consequently,
\begin{equation}\label{eq:logarithmic-multiplier}
 \|R_u-R_{\widetilde u}\|_\infty
 +\|\rho_u-\rho_{\widetilde u}\|_\infty
 \le C_M\|u-\widetilde u\|_1,
 \qquad 1-|R_u|^2\gtrsim_M1.
\end{equation}
Here $C_M$ depends only on $M$.
\end{enumerate}
\end{lemma}

\begin{proof}
The real-axis Volterra equations and coefficient formulas are
\cite[(3.4)--(3.5)]{DeiftZhou2003}.  Because $Q(p)$ is off diagonal,
an even number of its factors contributes to a diagonal entry,
and an odd number contributes to an off-diagonal entry.  The
zeroth diagonal term is $1$, while the first lower-left term is
$\epsilon\beta$.  This gives the powers of $\epsilon$ in
\eqref{eq:small-transmission}.

Here is the additional decay argument needed for the remainder.
An odd coefficient of order $m$ in $\widetilde b_{\epsilon p}$ is
an integral over $y_1<\cdots<y_m$ whose amplitude is a product
of $p$ and $\bar p$, and whose exponential is
\[
 e^{\ii z(y_1-y_2+\cdots+y_m)}.
\]
All variables lie in $(0,1)$.  Differentiation in $z$ thus costs
at most a polynomial in $m$.  To gain powers of $z^{-1}$,
integrate by parts in the common translation direction
\[
 D=\partial_{y_1}+\cdots+\partial_{y_m}.
\]
This derivative moves all integration variables together.
It is tangent to each face $y_j=y_{j+1}$, so these faces
produce no boundary integrals.  Since $m$ is odd,
$D(y_1-y_2+\cdots+y_m)=1$; each integration by parts therefore
replaces the exponential by a factor $(\ii z)^{-1}$ and
derivatives of the amplitude.  The outside faces give zero
because $p$ vanishes near the endpoints of $(0,1)$.
For every fixed number of derivatives, the resulting
coefficients are bounded by a polynomial in $m$ times
$C_N^m/m!$.  The factor $1/m!$ is the volume bound for the
ordered integration region.  Summation of the series gives
$s_N(\widetilde b_{\epsilon p}-\epsilon\beta)\lesssim_N|\epsilon|^3$.
The diagonal series needs no integration by parts to give
$\|\widetilde a_{\epsilon p}-1\|_{C_b^N}\lesssim_N\epsilon^2$.
Since $|\widetilde a_{\epsilon p}|\ge1$, differentiation of the
reciprocal gives its stated uniform bounds.

Using $|\widetilde a|^2- |\widetilde b|^2=1$, we obtain
\begin{equation}\label{eq:small-reflection}
 L_{\epsilon p}
 =-\log(1+|\widetilde b_{\epsilon p}|^2)
 =-2\log|\widetilde a_{\epsilon p}|.
\end{equation}
The preceding decay estimates, followed by the chain rule,
prove the estimate for $L_{\epsilon p}$.  Notice that this
argument uses the real logarithm of the modulus; no decay
assertion for the complex function $\widetilde a_{\epsilon p}-1$
is required.

For continuity, remove the free diagonal evolution from the
spatial equation.  The remaining coefficient matrix has
operator norm $|u(x)|$ at every real $z$.  Its fundamental
matrix is bounded by $e^{\|u\|_1}$.  Subtracting the integral
equations for $u$ and $\widetilde u$, and applying Gronwall's
inequality, gives
\[
 \sup_z\|\widetilde T_u-\widetilde T_{\widetilde u}\|_{\rm op}
 \lesssim e^{C(\|u\|_1+\|\widetilde u\|_1)}
                    \|u-\widetilde u\|_1.
\]
The quotients defining $R$ are Lipschitz on this matrix ball
because $|\widetilde a|\ge1$.  The density equals
$\pi^{-1}\log|\widetilde a|$, and the real function $\log s$
has derivative at most one for $s\ge1$; this proves
\eqref{eq:logarithmic-multiplier}.  Finally,
$|\widetilde a_u|\le e^{C\|u\|_1}$ and
$1-|R_u|^2=|\widetilde a_u|^{-2}$ give the uniform lower bound.

These formulas for $L^1\cap L^2$ potentials agree with the
$L^2$ spectral data used earlier.  Indeed, approximation by
smooth compactly supported potentials in both norms gives
uniform convergence of the Volterra coefficients by
\eqref{eq:insertion-log-density}, and convergence of the densities in $L^1$
by Lemma~\ref{lem:direct}.  Their almost-everywhere limits
must therefore coincide.
\end{proof}

For a fixed compactly supported smooth $u$, define
\begin{equation}\label{eq:small-density}
 C_u(z)=\beta(z)\frac{\overline{\widetilde b_u(z)}}{\widetilde a_u(z)}.
\end{equation}
This coefficient is the product of the first reflected
amplitude of $p$ and the reflected amplitude already present
in $u$.  It measures the first-order interference between
the two potentials.  It is Schwartz: the off-diagonal
coefficient of $u$ has rapid decay, whereas
$\widetilde a_u^{-1}$ has bounded derivatives.  Factoring
\eqref{eq:scattering-insertion}, with
\begin{equation}\label{eq:interference-coefficient-error}
 w_\epsilon=
 \frac{\widetilde b_{\epsilon p}\overline{\widetilde b_u}}
      {\widetilde a_{\epsilon p}\widetilde a_u},
 \qquad s_N(w_\epsilon-\epsilon C_u)\lesssim_{N,u}|\epsilon|^3,
\end{equation}
gives
\begin{equation}\label{eq:transmission-interference-factor}
 \widetilde a_{u+\epsilon p_X}
 =\widetilde a_u\widetilde a_{\epsilon p}
                 (1+e^{\ii Xz}w_\epsilon).
\end{equation}
For $|\epsilon|$ sufficiently small, $|w_\epsilon|<1/2$.
The convergent Taylor series for $\log(1+w)$ therefore yields
\begin{equation}\label{eq:insertion-factor}
 L_{u+\epsilon p_X}-L_u
 =-2\epsilon\Re(e^{\ii Xz}C_u)+R_{u,\epsilon,X},
\end{equation}
where the real error $R_{u,\epsilon,X}$ is exactly
\begin{align}\label{eq:interference-remainder-series}
 R_{u,\epsilon,X}
 ={}&L_{\epsilon p}
 -2\Re\bigl(e^{\ii Xz}(w_\epsilon-\epsilon C_u)\bigr)\notag\\
 &-2\Re\sum_{j=2}^\infty
                  \frac{(-1)^{j+1}}j e^{\ii jXz}w_\epsilon^j.
\end{align}
Thus the first-order change in $L$ has frequency $X$.
The remaining oscillations have frequencies $2X,3X,\ldots$
and amplitudes at least quadratic in $\epsilon$.

\subsubsection{The logarithmic factor in the phase}
We now calculate how $\cT$ acts on the oscillations in
\eqref{eq:insertion-factor}.  The derivative in its definition removes
the low-frequency ambiguity of the Fourier transform of
$\log u$ on $u>0$.  The resulting multiplier is logarithmic
at high frequency; this is where the distance $X$ enters
the asymptotic phase.

\begin{lemma}[The one-sided logarithmic transform]
\label{lem:multiplier}
Let $\gamma$ be Euler's constant.  The following properties
hold for $\cT$.
\begin{enumerate}
\item If $h\in\Sch$, and $\theta$ is the Fourier variable dual
to the spectral variable $z$, then
\begin{equation}\label{eq:one-sided-transform-multiplier}
 \widehat{\cT h}(\theta)=m(\theta)\widehat h(\theta),
 \qquad
 m(\theta)=-\frac{\log|\theta|+\gamma}{\pi}
                  -\frac{\ii}{2}\sgn\theta
 \quad(\theta\ne0).
\end{equation}
The value assigned to $m(0)$ is immaterial.
\item For a fixed $H\in\Sch$ and bounded interval $I$, as $X\to\infty$,
\begin{equation}\label{eq:modulated-transform-asymptotics}
 \left\|\cT(e^{\ii X\cdot}H)
  +\left(\frac{\log X+\gamma}{\pi}+\frac\ii2\right)
      e^{\ii X\cdot}H\right\|_{L^\infty(I)\cap L^2(I)}
 \longrightarrow0.
\end{equation}
Moreover, for some fixed integer $N$ and every $c\in\Sch$,
\begin{equation}\label{eq:modulated-transform-bound}
 \|\cT(e^{\pm\ii Y\cdot}c)\|_{L^\infty(I)\cap L^2(I)}
 \lesssim_I(1+\log Y)s_N(c),\qquad Y\ge2.
\end{equation}
Here an intersection norm denotes the sum of its two norms.
\end{enumerate}
\end{lemma}

\begin{proof}
Write the defining integral in \eqref{eq:one-sided-log-transform} as
$\pi^{-1}\int_0^\infty\log u\,h'(z-u)\dd u$.
For $\delta>0$, insert $e^{-\delta u}$ so that the Fourier
calculation is absolutely convergent.  The Laplace transform
\[
 \int_0^\infty e^{-(\delta+\ii\theta)u}u^{s-1}\dd u
       =\Gamma(s)(\delta+\ii\theta)^{-s}
\]
can be differentiated at $s=1$.  Since $\Gamma'(1)=-\gamma$,
the regularized multiplier is
\[
 \frac{\ii\theta}{\pi(\delta+\ii\theta)}
               \bigl[-\gamma-\log(\delta+\ii\theta)\bigr].
\]
The logarithm here has argument in $(-\pi/2,\pi/2)$.
Its limit for $\theta>0$ is
$-(\log\theta+\gamma)/\pi-\ii/2$, and the negative-frequency
limit is the complex conjugate.  Near zero the limiting
logarithm and its square are integrable; away from zero
$\widehat h$ decays faster than any power.  Splitting the
Fourier integral at $|\theta|=1$ thus permits passage to the
limit in both $L^1$ and $L^2$.  This proves
\eqref{eq:one-sided-transform-multiplier} by Fourier inversion and Plancherel's
theorem.

Modulation translates the Fourier variable:
$\widehat{e^{\ii X\cdot}H}(\theta)=\widehat H(\theta-X)$.
After subtraction of $m(X)e^{\ii Xz}H$, the translated
Fourier transform is
\[
       [m(X+\theta)-m(X)]\widehat H(\theta).
\]
On $|\theta|\le X/2$ the sign is fixed and
\[
 |m(X+\theta)-m(X)|
   =\pi^{-1}|\log(1+\theta/X)|\lesssim|\theta|/X.
\]
On the complementary set, the decay of $\widehat H$
controls the growth of the logarithm and its integrable
singularity at $\theta=-X$.  For clarity, on
$|X+\theta|\le1$ we have $|\theta|\ge X-1$, so the
$L^1$ and squared $L^2$ contributions are bounded by
$C_KX^{-K}$ times, respectively,
$\int_{-1}^1(1+|\log|s||)\dd s$ and
$\int_{-1}^1(1+|\log|s||^2)\dd s$.
On the rest of $|\theta|>X/2$, ordinary polynomial
bounds for the logarithm and rapid Fourier decay apply.
Both Fourier norms consequently tend to zero, giving
\eqref{eq:modulated-transform-asymptotics} globally.

The same decomposition, applied with $X=Y$ or $X=-Y$
and without subtracting the main term, bounds the Fourier
$L^1$ and $L^2$ norms by $(1+\log Y)s_N(c)$ for a
sufficiently large fixed $N$.  One may use, for example,
$|\widehat c(\theta)|\lesssim s_N(c)\langle\theta\rangle^{-4}$;
this follows by four integrations by parts and integration
of the weighted derivatives of $c$.  Fourier inversion and
Plancherel now prove \eqref{eq:modulated-transform-bound}.
\end{proof}

\begin{lemma}[The change in the phase]
\label{lem:remainder}
Fix $u\in C_c^\infty$, a bounded interval $I$, and
$\kappa>0$.  For $X$ sufficiently large set
$\epsilon=\kappa/\log X$ and place $p_X$ to the right of $u$.
Then
\begin{enumerate}
\item The error in \eqref{eq:insertion-factor} satisfies
\begin{equation}\label{eq:phase-remainder-bound}
 \|\cT R_{u,\epsilon,X}\|_{L^\infty(I)\cap L^2(I)}
 \lesssim_{u,p,I}\epsilon^2(1+\log X)\longrightarrow0.
\end{equation}
\item The leading phase difference is
\begin{equation}\label{eq:leading-insertion-phase}
 \cT(L_{u+\epsilon p_X}-L_u)
  =\frac{2\kappa}{\pi}\Re(e^{\ii Xz}C_u(z))
             +o_{L^\infty(I)\cap L^2(I)}(1).
\end{equation}
\end{enumerate}
\end{lemma}

\begin{proof}
We apply Lemma~\ref{lem:multiplier} to the terms already
separated in \eqref{eq:interference-remainder-series}.  The nonoscillatory term
$L_{\epsilon p}$ contributes $O(\epsilon^2)$, by
\eqref{eq:small-transmission} and the integrability of the multiplier
near zero.  The first-harmonic error contributes
$O(\epsilon^3(1+\log X))$, by \eqref{eq:interference-coefficient-error} and
\eqref{eq:modulated-transform-bound}.

For the remaining harmonics, the product rule gives, for
every fixed $N$,
\[
 s_N(w_\epsilon^j)\le C_N^j j^N|\epsilon|^j.
\]
Indeed, there are at most $j^N$ choices for distributing
at most $N$ derivatives among the $j$ factors; the
polynomial weight can be attached to any one of the
Schwartz factors.  After division by $j$ in the logarithmic
series, we therefore have
\begin{equation}\label{eq:harmonic-series-bound}
 \sum_{j=2}^\infty\frac{1+\log j}{j}s_N(w_\epsilon^j)
             \lesssim_{N,u,p}\epsilon^2.
\end{equation}
This series is convergent for sufficiently small
$|\epsilon|$, since exponential decay in $j$ controls
both $j^{N-1}$ and $\log j$.  Applying
\eqref{eq:modulated-transform-bound} with $Y=jX$ is legitimate term by
term, and the resulting norm series is summable because
$\log(jX)=\log X+\log j$.  Combining the three types
of terms gives
\begin{equation}\label{eq:phase-remainder-decay}
 \|\cT R_{u,\epsilon,X}\|_{L^\infty(I)\cap L^2(I)}
 \lesssim_{u,p,I}\epsilon^2(1+\log X)
 \lesssim_{u,p,I,\kappa}(\log X)^{-1}.
\end{equation}
This proves \eqref{eq:phase-remainder-bound}.

The kernel defining $\cT$ is real, so it commutes with
real parts.  Formula \eqref{eq:modulated-transform-asymptotics} consequently gives
\[
 -2\epsilon\cT\Re(e^{\ii Xz}C_u)
 =\frac{2\epsilon\log X}{\pi}\Re(e^{\ii Xz}C_u)
       +O_{L^\infty(I)\cap L^2(I)}(\epsilon)+o(\epsilon).
\]
Since $\epsilon\log X=\kappa$, this and
\eqref{eq:phase-remainder-bound} prove \eqref{eq:leading-insertion-phase}.
\end{proof}

\subsubsection{A nonzero change in the asymptotic coefficient}
A nonzero phase perturbation can only be detected where the
original asymptotic amplitude is nonzero.  Since $\beta$ is
not identically zero, choose a compact interval $I$ of
positive length on which $\beta$ has no zeros.  Take
$\eta>0$ sufficiently small and define the starting datum
\begin{equation}\label{eq:starting-potential}
                         u_0=\eta p.
\end{equation}
The parameter $\eta$ fixes the initial amplitude on $I$.
By \eqref{eq:small-transmission},
\begin{equation}\label{eq:starting-reflection}
 \widetilde b_{u_0}=\eta\beta+O_{L^\infty(I)}(\eta^3),
 \qquad C_{u_0}=\eta|\beta|^2+O_{L^\infty(I)}(\eta^3).
\end{equation}
The quantities $|R_{u_0}|$, $|C_{u_0}|$, and $\rho_{u_0}$
therefore have positive lower bounds on this interval.

We shall keep the total size of all later additions below
one fixed number $E_*>0$.  Decreasing $E_*$ if necessary,
Lemma~\ref{lem:small-scattering} ensures that every
finite sum with disjoint supports ordered from left to right,
\begin{equation}\label{eq:admissible-finite-potentials}
 u=u_0+\sum_{j=1}^N\epsilon_jp_{X_j},
 \qquad \epsilon_j>0,
 \qquad \sum_{j=1}^N\epsilon_j\le E_*,
\end{equation}
satisfies
\begin{equation}\label{eq:uniform-insertion-bounds}
 \inf_I|R_u|\ge R_*>0,\quad
 \inf_I|C_u|\ge c_*>0,\quad
 \inf_I\rho_u\ge\rho_*>0,\quad
 \sup_I|R_u|\le\sqrt{1-\delta_*}.
\end{equation}
These four constants are independent of the number and
locations of the additions.  To check this, all such sums
belong to the $L^1$ ball of radius
$(\eta+E_*)\|p\|_1$.  The continuity estimates imply
\[
 \|R_u-R_{u_0}\|_\infty+
 \|C_u-C_{u_0}\|_{L^\infty(I)}
 \lesssim\|u-u_0\|_1\le E_*\|p\|_1.
\]
Taking $E_*$ small preserves the positive lower bounds,
and the gap $\delta_*$ follows from
\eqref{eq:logarithmic-multiplier}.  We also have the useful upper bound
$|C_u|\le|\beta|$, since $|\widetilde b_u/\widetilde a_u|<1$.

\begin{proposition}[Addition of a separated potential]
\label{prop:insertion}
There are constants $\kappa,d_*>0$, determined by $p$, $I$,
and the starting datum, with the following property.
Fix a finite sum $u$ as in \eqref{eq:admissible-finite-potentials}, positive
numbers $\delta_2,\delta_\rho$, and a lower bound $X_0$ for
the next location.  There is $X>X_0$ such that
$\epsilon=\kappa/\log X$ and
\begin{enumerate}
\item The support of $p_X$ lies to the right of that of $u$,
and the changes in the datum and its density obey
\begin{equation}\label{eq:small-insertion-bounds}
 \|\epsilon p_X\|_2<\delta_2,
 \qquad
 \|\rho_{u+\epsilon p_X}-\rho_u\|_\infty<\delta_\rho.
\end{equation}
Any prescribed finite collection of positive upper bounds
on $\epsilon$ may also be imposed.
\item The smooth modified-scattering coefficients satisfy
\begin{equation}\label{eq:asymptotic-coefficient-separation}
                    \|V_{u+\epsilon p_X}-V_u\|_2\ge d_*.
\end{equation}
\end{enumerate}
\end{proposition}

\begin{proof}
The formula for $V_u=e^{\ii\pi/4}W_u$ can be written without
choosing arguments of the reflection coefficient or the Gamma
function.  Namely, define
\[
 g(s)=\left(\frac{-\log(1-s)}{2\pi s}\right)^{1/2}
 \quad(0<s<1),\qquad g(0)=(2\pi)^{-1/2},
\]
using the positive square root, and let
\[
 H(r)=e^{-\ii\rho(r)\log2}\mathcal G(\rho(r))\,r\,g(|r|^2),
 \qquad \rho(r)=-\frac1{2\pi}\log(1-|r|^2),\quad |r|<1.
\]
The function $H$ contains the part of the asymptotic
coefficient depending only on the reflection at the same
spectral point.  Equation \eqref{eq:classical-ray-coefficient}
then reads
\begin{equation}\label{eq:classical-coefficient-factorization}
                  V_u=\ii e^{\ii\cT L_u}H(R_u).
\end{equation}
On every compact annulus inside $0<|r|<1$,
$H$ is smooth and has no zeros.  The uniform $L^1$ continuity
estimate gives
\[
 \|R_{u+\epsilon p_X}-R_u\|_{L^\infty(I)}\lesssim\epsilon.
\]
If the constant in this estimate is $K$, choose $X$ large
enough that $K\epsilon\le\min\{R_*/2,\delta_*/4\}$.
Since $||r_1|^2-|r_2|^2|\le2|r_1-r_2|$ on the unit disk,
both reflections now belong on $I$ to the fixed compact annulus
\[
 \{r\in\C:R_*/2\le|r|\le\sqrt{1-\delta_*/2}\}.
\]
The quotient of the two $H$ factors therefore equals
$1+h_X$, with
$\|h_X\|_{L^\infty(I)}\lesssim\epsilon$.  The constant in
this estimate is independent of the preceding finite sum.
The phase estimate \eqref{eq:leading-insertion-phase} now gives
\begin{equation}\label{eq:coefficient-ratio}
 \frac{V_{u+\epsilon p_X}}{V_u}
     =e^{\ii(d_X+o_X)}(1+h_X),\qquad
 d_X(z)=\frac{2\kappa}{\pi}\Re(e^{\ii Xz}C_u(z)),
\end{equation}
where $o_X$ is real and
$\|o_X\|_{L^\infty(I)\cap L^2(I)}\to0$ for each fixed
$u$.  In particular, this comparison involves actual
complex coefficients and is unaffected by a choice of
phase branches.

Fix
\begin{equation}\label{eq:phase-amplitude-choice}
          0<\kappa\le\frac{\pi}{8\|\beta\|_{L^\infty(I)}}.
\end{equation}
Then $|d_X|\le1/4$.  Moreover,
\[
 \bigl[\Re(e^{\ii Xz}C_u(z))\bigr]^2
 =\frac12|C_u(z)|^2+
                     \frac12\Re(e^{2\ii Xz}C_u(z)^2).
\]
The last term has integral tending to zero over $I$ by
the Riemann--Lebesgue lemma.  Hence
\begin{equation}\label{eq:oscillatory-phase-lower-bound}
 \|\Re(e^{\ii X\cdot}C_u)\|_{L^2(I)}^2
 \longrightarrow\frac12\|C_u\|_{L^2(I)}^2
                 \ge\frac12 c_*^2|I|.
\end{equation}
In particular, for sufficiently large $X$,
$\|d_X\|_{L^2(I)}\ge\kappa c_*|I|^{1/2}/\pi$.
Increase $X$ until $\|o_X\|_\infty\le1/4$ and
\[
 \|o_X\|_{L^2(I)}+|I|^{1/2}\|h_X\|_{L^\infty(I)}
             \le\frac{\kappa c_*|I|^{1/2}}{4\pi}.
\]
For $|s|\le1$, the elementary bound
$|e^{\ii s}-1|\ge|s|/2$ applies.  Subtracting the
$h_X$ and $o_X$ errors in \eqref{eq:coefficient-ratio} therefore gives
\[
 \left\|\frac{V_{u+\epsilon p_X}}{V_u}-1\right\|_{L^2(I)}
 \ge\frac12\|d_X\|_{L^2(I)}
        -\|o_X\|_{L^2(I)}-|I|^{1/2}\|h_X\|_{L^\infty(I)}
 \ge\frac{\kappa c_*|I|^{1/2}}{4\pi}.
\]
Since $|V_u|=\sqrt{\rho_u}\ge\sqrt{\rho_*}$ on $I$,
we may take
\begin{equation}\label{eq:coefficient-separation-constant}
                  d_*=\frac{\kappa c_*\sqrt{\rho_*|I|}}{8\pi}.
\end{equation}
This proves \eqref{eq:asymptotic-coefficient-separation}.

All remaining requirements follow by increasing $X$ once
more.  Indeed, $\epsilon=\kappa/\log X$ tends to zero,
$\|\epsilon p_X\|_2=\epsilon\|p\|_2$, and
\eqref{eq:logarithmic-multiplier} bounds the density change by
$C_M\epsilon\|p\|_1$.  The estimates for the phase and
\eqref{eq:oscillatory-phase-lower-bound} only need to hold for the already fixed
$u$; no rate uniform over all finite sums is used.
\end{proof}

\subsection{Construction of one fixed initial datum}
\label{sec:counterexample}

The preceding proposition gives small changes in the initial
datum with a definite change in the asymptotic coefficient.
To produce a single solution, we must also ensure that the
infinite tail is too small to affect the finitely many times
already chosen.  This determines the order of the construction:
we first choose an observation time for the current smooth
solution, then fix the allowed size of every later addition.
The distance of the next potential remains free and makes
these requirements compatible with \eqref{eq:asymptotic-coefficient-separation}.

\subsubsection{Choice of the potentials and observation times}
Let $u_0=\eta p$ be the datum from \eqref{eq:starting-potential}, and
set $t_{-1}=0$.  Suppose the first $n$ additions have been
chosen, so that
\begin{equation}\label{eq:finite-initial-sums}
              u_n=u_0+\sum_{m=1}^n\epsilon_mp_{X_m}.
\end{equation}
Write $V_n=V_{u_n}$ and $\rho_n=\rho_{u_n}$.
Since $u_n$ is compactly supported and smooth,
Lemma~\ref{lem:smoothstrong} permits a time
$t_n>\max\{e^{n+1},t_{n-1}+1\}$ with
\begin{equation}\label{eq:observation-asymptotic-error}
 \|e^{\ii(\log t_n)\rho_n}F_{u_n}(t_n)-V_n\|_2<2^{-n}.
\end{equation}
At this fixed time, continuity of the cubic NLS flow in
$L^2$, as given by \cite[Theorem~1.1 and (1.6)]{Tsutsumi1987},
provides a number $d_n>0$ such that
\begin{equation}\label{eq:observation-continuity}
 \|h-u_n\|_2<d_n
 \quad\Longrightarrow\quad
 \|F_h(t_n)-F_{u_n}(t_n)\|_2<2^{-n}.
\end{equation}
The Fourier multiplier in the definition of $F$ is unitary,
so the same continuity statement holds for $F$ and for the
solution itself.

Define the uniform initial norm bounds
\[
       M_1=(\eta+E_*)\|p\|_1,
       \qquad M_2=(\eta+E_*)\|p\|_2.
\]
Choose $B_n>0$ such that
\begin{equation}\label{eq:future-tail-size}
 \|p\|_2B_n<d_n,
 \qquad
 (\log t_n)C_{M_1}\|p\|_1B_nM_2<2^{-n}.
\end{equation}
Here $B_n$ is an upper bound for the sum of amplitudes of
all additions after $u_n$.  Its first condition protects
the flow comparison \eqref{eq:observation-continuity}; its second protects
the logarithmic phase correction at time $t_n$.

We now choose $X_{n+1}$ using
Proposition~\ref{prop:insertion}, and put
$\epsilon_{n+1}=\kappa/\log X_{n+1}$.
In general, at the $m$th addition require $X_m>m$, require
its support to lie strictly to the right of the preceding
supports, and impose
\begin{equation}\label{eq:inductive-amplitude-constraints}
 \epsilon_m\le E_*2^{-m},\qquad
 \epsilon_m\le2^{-(m-j)}B_j\quad(0\le j<m).
\end{equation}
Only finitely many upper bounds occur at this stage, and
all their right sides have already been fixed.  The
proposition therefore gives such an addition together with
\begin{equation}\label{eq:consecutive-coefficient-separation}
                         \|V_m-V_{m-1}\|_2\ge d_*.
\end{equation}
The first bound in \eqref{eq:inductive-amplitude-constraints} keeps every partial sum
in the class \eqref{eq:admissible-finite-potentials}.  For each fixed $n$, the
second gives the tail estimate
\begin{equation}\label{eq:summable-future-tail}
             \sum_{m>n}\epsilon_m
              \le B_n\sum_{k=1}^\infty2^{-k}=B_n.
\end{equation}
This completes the inductive choice for all $n$.

Define the final initial datum by
\begin{equation}\label{eq:limiting-initial-datum}
                 u_*=u_0+\sum_{m=1}^\infty\epsilon_mp_{X_m}.
\end{equation}
The amplitudes are summable, so this series converges in
$L^1\cap L^2$.  Its supports escape to $+\infty$, and
there are only finitely many summands on each compact
set.  Thus $u_*$ has a real-valued smooth representative.
Disjointness gives, for every integer $k\ge0$,
\begin{equation}\label{eq:limiting-sobolev-norms}
 \|\partial_x^ku_*\|_2^2
  =\left(\eta^2+\sum_{m=1}^\infty\epsilon_m^2\right)
                              \|\partial_x^kp\|_2^2<\infty.
\end{equation}
In particular, $u_*\in\bigcap_{k\ge0}H^k$.
On the other hand, $x\ge X_m$ on the $m$th support, so
\[
 \|xu_*\|_2^2\ge\|p\|_2^2\sum_{m=1}^\infty\epsilon_m^2X_m^2
                              =\infty.
\]
Indeed, $X_m\to\infty$ and
$\epsilon_mX_m=\kappa X_m/\log X_m\to\infty$;
the terms in this last series do not even tend to zero.

\subsubsection{Failure of convergence}
The density continuity estimate and \eqref{eq:summable-future-tail} give
\begin{equation}\label{eq:limiting-density-error}
       \|\rho_{u_*}-\rho_n\|_\infty
                   \le C_{M_1}\|p\|_1B_n.
\end{equation}
For $t>1$, denote the corrected Fourier profile of the
single solution starting from $u_*$ by
\begin{equation}\label{eq:counterexample-corrected-profile}
                 Z_*(t)=e^{\ii(\log t)\rho_{u_*}}F_{u_*}(t).
\end{equation}
We compare it at the prescribed times with $V_n$.
First, disjointness and \eqref{eq:summable-future-tail} imply
\[
 \|u_*-u_n\|_2
  =\|p\|_2\left(\sum_{m>n}\epsilon_m^2\right)^{1/2}
  \le\|p\|_2 B_n<d_n.
\]
Thus the solution comparison at time $t_n$ is less than
$2^{-n}$ by \eqref{eq:observation-continuity}.  For the change in its
phase, use $|e^{\ii s}-1|\le|s|$, \eqref{eq:limiting-density-error},
and mass conservation:
\[
 \|(e^{\ii(\log t_n)(\rho_{u_*}-\rho_n)}-1)F_{u_n}(t_n)\|_2
 \le(\log t_n)C_{M_1}\|p\|_1B_n\|u_n\|_2<2^{-n},
\]
since $\|u_n\|_2\le M_2$.  Finally,
\eqref{eq:observation-asymptotic-error} controls the smooth asymptotic error.
The triangle inequality consequently yields
\begin{align}\label{eq:counterexample-observation-error}
 \|Z_*(t_n)-V_n\|_2
 &\le\|F_{u_*}(t_n)-F_{u_n}(t_n)\|_2\notag\\
 &\quad+\|(e^{\ii(\log t_n)(\rho_{u_*}-\rho_n)}-1)F_{u_n}(t_n)\|_2\notag\\
 &\quad+\|e^{\ii(\log t_n)\rho_n}F_{u_n}(t_n)-V_n\|_2
 <3\cdot2^{-n}.
\end{align}
The same single solution thus approaches $V_n$ at time
$t_n$, while consecutive $V_n$ remain at distance at least
$d_*$.  More explicitly,
\[
 \|Z_*(t_n)-Z_*(t_{n-1})\|_2
 \ge d_*-3\cdot2^{-n}-3\cdot2^{-(n-1)}.
\]
The right side tends to $d_*>0$, so $Z_*(t)$ cannot have a
strong $L^2$ limit.  Proposition~\ref{prop:necessity}
shows that any coefficient satisfying \eqref{eq:fixed}
would have squared modulus $\rho_{u_*}$ and would force
precisely this convergence.  We have proved the failure
assertion of Theorem~\ref{thm:failure} for the datum
\eqref{eq:limiting-initial-datum}.

\begin{remark}
The order in which the parameters are fixed is necessary.
At a late observation time the continuity radius $d_n$
may be very small, and the phase error contains the
possibly large number $\log t_n$.  The tail bound $B_n$
therefore has to be chosen after both $t_n$ and $d_n$.
The relation $\epsilon\log X=\kappa$ then permits all
subsequent initial perturbations to fit these bounds while
their asymptotic coefficients still satisfy
\eqref{eq:consecutive-coefficient-separation}.
\end{remark}

\subsubsection{Adjustment of the initial norm}
For any solution $u$ and any $\lambda>0$, the cubic NLS
scaling gives the solution
\begin{equation}\label{eq:nls-rescaling}
 u^{(\lambda)}(t,x)=\lambda u(\lambda^2t,\lambda x),
 \qquad
 \|u^{(\lambda)}(0)\|_2=\lambda^{1/2}\|u(0)\|_2.
\end{equation}
We verify that this change of scale preserves the failure
of a fixed modified-scattering coefficient.  The unitary
Fourier normalization gives exactly
\[
             F_{u^{(\lambda)}(0)}(t,\xi)
                  =F_{u(0)}(\lambda^2t,\xi/\lambda).
\]
If $W\in L^2$ satisfied \eqref{eq:fixed} for the scaled
solution, define, in the original Fourier variable $\zeta$,
\[
             V(\zeta)=e^{2\ii(\log\lambda)|W(\lambda\zeta)|^2}
                                  W(\lambda\zeta).
\]
Its modulus is $|W(\lambda\zeta)|$.  With $s=\lambda^2t$,
the logarithmic identity $\log t=\log s-2\log\lambda$
and the change of variable $\xi=\lambda\zeta$ give
\begin{align*}
 &\|F_{u(0)}(s)-e^{-\ii(\log s)|V|^2}V\|_{L^2_\zeta}\\
 &\qquad=\lambda^{-1/2}
 \|F_{u^{(\lambda)}(0)}(s/\lambda^2)
       -e^{-\ii\log(s/\lambda^2)|W|^2}W\|_{L^2_\xi}
 \longrightarrow0.
\end{align*}
Thus the scaled coefficient would give a coefficient for
the original solution.  Rescaling by $1/\lambda$ proves
the converse as well.

The stronger failure of a limit for the density-corrected
profile also persists.  Rescaling the spatial spectral
system yields
\[
 a_{u^{(\lambda)}(0)}(k)=a_{u(0)}(k/\lambda),\qquad
 \rho_{u^{(\lambda)}(0)}(\xi)=\rho_{u(0)}(\xi/\lambda).
\]
Hence the corrected profiles, defined as in
\eqref{eq:counterexample-corrected-profile}, obey
\[
 Z_{u^{(\lambda)}(0)}(t,\xi)
 =e^{-2\ii(\log\lambda)\rho_{u(0)}(\xi/\lambda)}
                    Z_{u(0)}(\lambda^2t,\xi/\lambda).
\]
The multiplier on the right has modulus one and is fixed
in time; dilation is an invertible map on $L^2$.
Therefore one side has a strong limit if and only if the
other does.

For a prescribed norm $M>0$, apply this scaling to
\eqref{eq:limiting-initial-datum} with
\begin{equation}\label{eq:prescribed-norm-scale}
                         \lambda=\left(\frac M{\|u_*\|_2}\right)^2.
\end{equation}
The scaled datum is real and has norm $M$.  Its remaining
properties follow from the direct changes of variables
\[
 \|u^{(\lambda)}(0)\|_1=\|u(0)\|_1,\qquad
 \|\partial_x^ku^{(\lambda)}(0)\|_2
             =\lambda^{k+1/2}\|\partial_x^ku(0)\|_2,
 \qquad
 \|xu^{(\lambda)}(0)\|_2^2=\lambda^{-1}\|xu(0)\|_2^2.
\]
The last equality holds also when both sides are infinite.
This finishes the proof of Theorem~\ref{thm:failure}.

\appendix
\section{Comparison of the scattering conventions}
\label{app:conventions}

The spectral results and the large-time asymptotic formula use
different normalizations of the spatial system.  We record the
changes of basis explicitly.  They determine the spectral
variable, the sign of the reflection coefficient, and the
constant phase in the linear limit.

\subsection{The real Dirac system}
The real Dirac equation in \cite{BessonovGubkin} is
\[
 J_0U_y+Q_G(p)U=\zeta U,\qquad
 J_0=\begin{pmatrix}0&-1\\1&0\end{pmatrix},\qquad
 Q_G(p)=\begin{pmatrix}\Im p&\Re p\\\Re p&-\Im p\end{pmatrix}.
\]
Here $y\ge0$ is the half-line spatial variable, $p(y)$ is
its complex potential, and $\zeta$ is its spectral
parameter.  Introduce the constant unitary matrix
\[
 S=\frac1{\sqrt2}\begin{pmatrix}1&-1\\-\ii&-\ii\end{pmatrix}
 \quad\hbox{and set}\quad Y=S^*U.
\]
The columns of $S$ diagonalize $J_0$.  In fact,
\[
 S^*J_0S=\ii J,
 \qquad
 S^*J_0Q_G(p)S=
       \begin{pmatrix}0&\bar p\\p&0\end{pmatrix}=Q(\bar p).
\]
Multiplying the real equation by $-J_0$ and applying $S^*$
therefore gives
\[
                         Y_y=(-\ii\zeta J+Q(\bar p))Y.
\]
Comparing with \eqref{eq:zs}, the physical coordinate
and spectral parameter remain unchanged: $y=x$ and
$\zeta=k$ on the right half-line, with $p=\bar u$.
For the left half-line starting at a split $x$, use
$\sigma_1Y(x-y)$, where $\sigma_1$ exchanges the two
components.  Differentiation changes the potential to
$-\overline{u(x-y)}$ in the physical system, so the
corresponding real Dirac potential is $p(y)=-u(x-y)$.

The half-line quotient also has a sign to fix.  If $m$
is the real Weyl function, its square-integrable solution
at the split has direction $(-m,1)^{\mathrm T}$.  The
basis change gives
\[
 S^*\begin{pmatrix}-m\\1\end{pmatrix}
     =\frac1{\sqrt2}\begin{pmatrix}\ii-m\\\ii+m\end{pmatrix}.
\]
Its first component divided by its second is
$(\ii-m)/(\ii+m)$, the negative of the Dirac Schur
function $(m-\ii)/(m+\ii)$.  Applying the same calculation
after reflection fixes the same minus sign on the left.

In the full-line notation of \cite[\S4.4]{BessonovGubkin},
let $u_G=-u$ denote their potential.  Their two half-line
potentials at the origin are
\[
 u_G^+(y)=-\overline{u_G(y)}=\bar u(y),
 \qquad u_G^-(y)=u_G(-y)=-u(-y).
\]
These are exactly the real Dirac potentials just obtained.
If $f_G,g_G$ are their right and left Schur functions,
our physical quotients are $f=-f_G$ and $g=-g_G$.
The outer transmissions $A,B$ are unchanged, and
\[
 a=AB(1-fg)=AB(1-f_Gg_G).
\]
To check the reflection sign as well, the physical
lower-left transition entry is
\[
 b=B\bar A(g-\bar f)=B\bar A(\bar f_G-g_G).
\]
The expression on the right is their full-line $b$.
Since our jump convention is $r=-b/a$, while theirs
uses $b/a$, the transmissions coincide and the
reflections are negatives of one another.  No dilation
of the spectral variable occurs in this comparison.

\subsection{The Deift--Zhou normalization}
For smooth potentials, let $T_u(k)$ be the transition
matrix determined by $Y_-=Y_+T_u$ in the physical
spatial system.  Its real-axis form is
\[
 T_u(k)=\begin{pmatrix}a_u(k)&\overline{b_u(k)}\\
                         b_u(k)&\overline{a_u(k)}\end{pmatrix},
 \qquad r_u(k)=-\frac{b_u(k)}{a_u(k)}.
\]
The system \eqref{eq:classical-spatial-system} uses $z=-2k$ and reverses
the connection relation, so
\[
                     \widetilde T_u(z)=T_u(-z/2)^{-1}.
\]
Inversion of this determinant-one matrix gives
\begin{align*}
 \widetilde a_u(z)&=\overline{a_u(-z/2)},\qquad
 \widetilde b_u(z)=-b_u(-z/2),\\
 R_u(z)&=-\overline{r_u(-z/2)}
                         \frac{\overline{a_u(-z/2)}}{a_u(-z/2)}.
\end{align*}
These identities prove \eqref{eq:dz-conversion}; in
particular, the two reflection coefficients have equal
modulus at the corresponding spectral point.

For completeness, the columns in the half-line formula
\eqref{eq:half-splice} can be identified directly.
At a split, after removal of the free exponential, the
right square-integrable column is $A(f,1)^{\mathrm T}$,
and the left one is $B(1,g)^{\mathrm T}$.  Their
Wronskian is $AB(1-fg)=a$.  Dividing the left column
by this Wronskian and leaving the right column unchanged
gives
\[
 M_+=\begin{pmatrix}B/a&Af\\Bg/a&A\end{pmatrix}.
\]
The conjugation symmetry of the spatial equation gives
$M_- =\sigma_1\overline{M_+(\bar k)}\sigma_1$.
These column identities extend from smooth potentials
by the half-line spectral continuity used in
Lemma~\ref{lem:direct}.

For the smooth inverse construction in
Lemma~\ref{lem:proper}, suppose $r$ is smooth,
compactly supported, and $\|r\|_\infty<1$.  Let $a_r$
be its outer transmission from \eqref{eq:outer-a}.
The reflection prescribed to the inverse theorem of
\cite[\S3, pp.~1058--1059]{DeiftZhou2003} is then
\[
 R_{\rm DZ}(z)=-\overline{r(-z/2)}
                    \frac{\overline{a_r(-z/2)}}{a_r(-z/2)}.
\]
The outer factor is smooth on the real axis for this
$r$.  The displayed reflection is therefore smooth,
compactly supported, and has the same strict modulus
gap.  It belongs to the required space
\[
 H^{1,1}_1=\{R\in H^{1,1}:\|R\|_\infty<1\},
 \qquad H^{1,1}=\{R\in L^2:R',zR\in L^2\}.
\]
The cited inverse theorem supplies the physical
potential after the change $z=-2k$.  Its transmission
is fixed by the outer formula and normalization at
imaginary infinity, so the reflection in our jump is
exactly the prescribed $r$.  The inverse theorem and
higher regularity assertions in
\cite[Theorems~2.4 and 2.11]{Zhou1998} provide the
corresponding alternative formulation.

The first-order expansion also checks the Fourier
normalization.  For $u=\epsilon f$, $f\in\Sch$, the
physical lower-left coefficient is
\[
 b_u(k)=\epsilon\int_\R\bar f(x)e^{-2\ii kx}\dd x
                                    +O(\epsilon^3).
\]
Consequently,
\[
 R_u(z)=\epsilon\sqrt{2\pi}\widehat f(z)+O(\epsilon^3).
\]
Since $\mathcal G(0)=-\ii$ and the square-root factor
in \eqref{eq:classical-ray-coefficient} tends to
$(2\pi)^{-1/2}$, the coefficient of $\epsilon$ in
$V_u=e^{\ii\pi/4}W_u$ is $\widehat f$.
Similarly, the coefficient of $\epsilon$ in
$p_{r_u(-z/2)}/\sqrt2$, as defined by
\eqref{eq:weber-coeff}, is
$e^{-\ii\pi/4}\widehat f(z)$.  This is the Fresnel
phase relating the exact Fourier profile to the ray
profile.  Stating these as first-order checks avoids
an unnecessary assertion about higher-order expansions
of the normalized Gamma phase at a zero of $r$.

\subsection{The time evolution}
For a smooth solution, write $U=-\ii kJ+Q(u)$ for the
spatial Lax matrix and introduce the time matrix
\[
 V=-2\ii k^2J+2kQ+\ii JQ_x-\ii JQ^2.
\]
The compatibility condition is
$U_t-V_x+[U,V]=0$, with $[U,V]=UV-VU$.
Its upper-right entry equals
\[
                     u_t-\ii u_{xx}+2\ii|u|^2u=0,
\]
which is exactly \eqref{eq:nls}.  Outside the potential,
the spatial and temporal solutions are
$e^{-\ii kxJ}$ and $e^{-2\ii k^2tJ}$.  Comparing the
two components thus produces the factors
$e^{\pm2\ii kx}$ and $e^{\pm4\ii k^2t}$ in the jump
\eqref{eq:jump}.

This agrees with \cite[Theorem~2.3 and (2-9)]{BessonovDenisov}.
Their system, with spectral parameter $\zeta$, is
\[
 Y_x^{\rm BD}=
   \begin{pmatrix}-\ii\zeta/2&\bar u\\u&\ii\zeta/2\end{pmatrix}
                                               Y^{\rm BD}.
\]
Exchanging components and setting $\zeta=-2k$ gives
our physical system.  If $T_u^{\rm BD}$ denotes their
matrix with the same relation $Y_-=Y_+T$, then
\[
 T_u(k)=\sigma_1T_u^{\rm BD}(-2k)\sigma_1,
 \qquad r_u(k)=-\overline{r_u^{\rm BD}(-2k)},
 \qquad r_u^{\rm BD}=b_u^{\rm BD}/a_u^{\rm BD}.
\]
Their evolution $r_u^{\rm BD}(\zeta,t)
=e^{-\ii\zeta^2t}r_u^{\rm BD}(\zeta,0)$ therefore
becomes
\[
                         r_u(k,t)=e^{4\ii k^2t}r_u(k,0).
\]
For $L^2$ data these identities are passed at each fixed
time by smooth approximation, continuity of the NLS
flow, and Lemma~\ref{lem:direct}.  This passage concerns
a fixed-time spectral identity and does not interchange
an approximation limit with a large-time limit.

\section{Analytic facts for the inverse factors}
\label{app:analytic}

We collect the boundary and Hardy-space facts needed in
the inverse argument.  The standard maximum and
weak-type estimates are quoted from the literature;
the weighted normalization and the estimate on vertical
segments are written out because they determine the
endpoint passage in the comparison of inverse factors.

\subsection{Outer functions and Hardy spaces}
For $h\ge0$ in $L^1(\R)$ let $a_h$ be the outer function
obtained from \eqref{eq:outer-a} by replacing
$\ell$ with $h$.  Taking the real part of its logarithm
gives
\[
 \log|a_h(x+\ii y)|
   =\frac1{2\pi}\int_\R\frac{yh(s)}{(s-x)^2+y^2}\dd s,
 \qquad y>0.
\]
The right side is half the Poisson integral of $h$.
Hence the boundary modulus is $e^{h/2}$, the reciprocal
is bounded by one in the upper half-plane, and
$a_h(\ii y)\to1$ as $y\to\infty$.
The last assertion follows directly from
$|\log a_h(\ii y)|\le\|h\|_1/(2\pi y)$.
The standard outer-factor uniqueness fixes the
unimodular constant by this normalization
\cite[Chapter~II, \S4]{Garnett}.

For clarity, the disk Smirnov class consists of quotients
$u/v$, where $u$ is bounded analytic and $v$ is bounded
outer.  Hardy functions belong to this class, as do
quotients of outer functions.  We use the following
standard consequences of its maximum principle and
factorization theorem \cite[Chapter~II]{Garnett}.

\begin{lemma}[Boundary control in the Smirnov class]
\label{app:lem:smirnov}
Let $F$ be a Smirnov function in the unit disk, and let
$F^*$ denote its almost-everywhere boundary function.
\begin{enumerate}
\item If $|F^*|\le1$ almost everywhere, then
$|F|\le1$ throughout the disk.  In particular, if also
$|F|\ge1$ inside, then $F$ is constant of modulus one.
\item If $F^*\in L^2$ on the circle, then $F\in H^2$ and
\[
                         |F|^2\le P[|F^*|^2],
\]
where $P$ is the disk Poisson extension.  Its $H^2$ norm
is the $L^2$ norm of $F^*$.
\end{enumerate}
\end{lemma}

To apply this lemma to the half-plane, use
$z=(w-\ii)/(w+\ii)$.  Normalized circle measure becomes
\[
                    \frac{\dd w}{\pi(1+w^2)}.
\]
Thus the half-plane Hardy class obtained from the disk
is exactly the class for harmonic measure based at
$\ii$.  Translation gives the corresponding weight
based at any $k_c+\ii$.  This identifies the Hardy
spaces appearing in the compactness proof with the
standard disk spaces in the cited theorem.  The
Smirnov hypothesis is essential: an almost-everywhere
boundary bound alone does not give an interior bound
for an arbitrary analytic function.

The following continuation property is also used there.
For upper and lower harmonic-measure $H^1$ functions
with equal boundary values, the common function is
constant.  Indeed, under the same Cayley map the upper
function has only nonnegative circle Fourier modes;
the lower one is analytic in the exterior disk,
including infinity, and has only nonpositive modes.
Their matching boundary functions have only the
constant mode.  This applies to each matrix entry.
Products of two $H^2$ functions are in $H^1$ by
H\"older's inequality, so it applies to the determinants
and matrix products used in the inverse argument.

There is a local version without continuity of the
boundary functions.  If the upper and lower functions
have local $H^1$ bounds and equal $L^1$ traces on an
open interval, their piecewise union is locally
integrable.  Integration against a compactly supported
smooth test function, first at height $y=\epsilon$
and then letting $\epsilon\downarrow0$, shows that
its distributional $\bar\partial$ is the difference
of the two traces on the interval.  This difference
vanishes.  The distributional Cauchy--Riemann theorem
then extends the union analytically across the
interval.

\subsection{Boundary phases and weak-type bounds}
Use the Hilbert-transform convention
\[
 Hh(x)=\frac1\pi\operatorname{pv}\int_\R
                                \frac{h(s)}{x-s}\dd s.
\]
Here $\operatorname{pv}$ denotes the symmetric principal
value at $s=x$.  The Hardy--Littlewood maximal theorem
and the weak $(1,1)$ theorem for the Hilbert transform
are \cite[Chapters~I--II]{SteinSingular}.  In the
normalizations used here they give, for $h\in L^1$
and $\lambda>0$,
\begin{equation}\label{app:eq:weak-type}
 |\{\mathcal P h>\lambda\}|+
                  |\{|Hh|>\lambda\}|
                      \lesssim\lambda^{-1}\|h\|_1,
\end{equation}
where $\mathcal P h$ is the supremum of the Poisson
integrals of $|h|$ over all positive heights.
The domination of this Poisson maximal function by
the Hardy--Littlewood maximal function follows by
splitting its kernel into the intervals
$|s-x|\le y$ and $2^jy<|s-x|\le2^{j+1}y$;
the latter contributions have summable weights
$O(2^{-j})$.  This also fixes precisely the version
of the maximal function needed for the endpoint
exclusions.

For the outer function above, the boundary formula is
\begin{equation}\label{app:eq:outer-phase}
               a_h^*(x)=\exp\{h(x)/2+\ii Hh(x)/2\}.
\end{equation}
The plus sign can be checked directly: the principal
value of $(2\pi\ii)^{-1}\int h(s)/(s-x)\dd s$ is
$\ii Hh(x)/2$.  Smooth approximation gives the formula
for all $h\in L^1$, using \eqref{app:eq:weak-type} for
the imaginary part.  In particular, if $h_j\to h$
in $L^1$, then $Hh_j\to Hh$ in measure, because
\[
 |\{|H(h_j-h)|>\lambda\}|\lesssim
                         \lambda^{-1}\|h_j-h\|_1.
\]
The outer boundary phases and outer factors therefore
converge in measure as well.  This is the continuity
used for the truncated spectral intervals.

\subsection{Evaluation on vertical segments}
The line Hardy space $H^2(\Hp)$ consists of analytic
functions with bounded $L^2$ norms on horizontal
lines.  Its boundary Fourier transform is supported
in $[0,\infty)$; see \cite[Chapter~II, \S3]{Garnett}.
The following estimate passes this boundary control
to any vertical line.

\begin{lemma}[Vertical trace estimate]
\label{app:lem:vertical}
Let $F\in H^2(\Hp)$ and use its boundary $L^2$ norm.
\begin{enumerate}
\item For every real $v$,
\begin{equation}\label{app:eq:vertical}
                 \int_0^\infty|F(v+\ii y)|^2\dd y
                                \le\frac12\|F\|_2^2.
\end{equation}
Consequently, as $v_j\to v$,
$F(v_j+\ii\cdot)\to F(v+\ii\cdot)$ in $L^2(0,\infty)$.
\item If $F$ instead belongs to the harmonic-measure
$H^2$ class based at $k_c+\ii$, with boundary measure
\[
              \dd\mu_{k_c}(w)=\frac{\dd w}{\pi(1+(w-k_c)^2)},
\]
then $F(w)/(w-k_c+\ii)$ is in line $H^2(\Hp)$, with
squared norm $\pi\|F\|_{H^2(\mu_{k_c})}^2$.
For such functions the vertical traces vary strongly
in $L^2$ on every finite vertical segment.
\end{enumerate}
The reflected statements hold in $\Hm$, with
$w-k_c-\ii$ in the second assertion.
\end{lemma}

\begin{proof}
The Fourier representation of a line Hardy function is
\[
 F(v+\ii y)=\frac1{\sqrt{2\pi}}
              \int_0^\infty\widehat F(\xi)e^{\ii v\xi-y\xi}\dd\xi.
\]
For smooth Fourier data supported in a compact subset
of $(0,\infty)$, integration in $y$ gives
\[
 \int_0^\infty|F(v+\ii y)|^2\dd y
 =\frac1{2\pi}\int_0^\infty\int_0^\infty
 \frac{\widehat F(\xi)\overline{\widehat F(\eta)}
                         e^{\ii v(\xi-\eta)}}{\xi+\eta}
                                      \dd\eta\,\dd\xi.
\]
Schur's test, with the positive weight $\xi^{-1/2}$,
bounds the integral operator with kernel
$(\xi+\eta)^{-1}$ by $\pi$, since
\[
       \int_0^\infty\frac{\eta^{-1/2}}{\xi+\eta}\dd\eta
       =\xi^{-1/2}\int_0^\infty\frac{s^{-1/2}}{1+s}\dd s
       =\pi\xi^{-1/2}.
\]
Multiplication by the phases $e^{\ii v\xi}$ preserves
$L^2$, so the prefactor $1/(2\pi)$ gives the constant
$1/2$ in \eqref{app:eq:vertical}.  Density proves the
estimate for every line Hardy function.

Real translation is continuous in boundary $L^2$,
because its Fourier multiplier $e^{\ii h\xi}$ tends
to one and is uniformly bounded.  Applying
\eqref{app:eq:vertical} to $F(\cdot+h)-F$ proves
continuity of its vertical trace.

For the weighted assertion, the translated Cayley
map gives the line Hardy function
$F(w)/(w-k_c+\ii)$.  On the boundary its norm is
\[
 \int_\R\frac{|F(w)|^2}{1+(w-k_c)^2}\dd w
                            =\pi\|F\|_{H^2(\mu_{k_c})}^2.
\]
On a fixed finite vertical segment the removed
linear factor is uniformly bounded and varies
continuously as the real endpoint moves.  The first
assertion therefore gives the stated weighted
continuity.  Reflection in the real axis proves
the lower-half-plane statements.
\end{proof}

In the rectangle argument, each vertical side is
first placed strictly inside its spectral interval
and then moved to the endpoint.  The lemma gives
strong $L^2$ convergence of the individual traces.
Their products converge in $L^1$ by Cauchy--Schwarz.
For each fixed smooth approximating potential the
full transmission and its reciprocal are bounded
on the relevant half-plane.  The same endpoint
passage therefore applies to the quotient matrices
after division by the common transmission, as in
the rectangle proof.  This
justifies the endpoint limit of the vertical
integrals in the rectangle identity; no
additional real-axis integral is introduced.

\section*{Statements and Declarations}
\noindent\textbf{Use of generative artificial intelligence.}
The author used ChatGPT during the research process to discuss
mathematical ideas and to assist with the preparation of the manuscript.
All mathematical statements and proofs in the article were carefully checked
and independently verified by the author.  The author takes full responsibility
for the accuracy and integrity of the work.


\begin{thebibliography}{99}
\bibitem{Bessonov2020}
R. V. Bessonov, \textit{Szeg\H{o} condition and scattering for
one-dimensional Dirac operators}, Constr. Approx. \textbf{51} (2020),
273--302. \href{https://doi.org/10.1007/s00365-018-9453-3}{doi:10.1007/s00365-018-9453-3}.

\bibitem{BessonovDenisov}
R. V. Bessonov and S. A. Denisov, \textit{Sobolev norms of
$L^2$-solutions to the nonlinear Schr\"odinger equation}, Pacific
J. Math. \textbf{331} (2024), 217--258.
\href{https://doi.org/10.2140/pjm.2024.331.217}{doi:10.2140/pjm.2024.331.217}.

\bibitem{BessonovGubkin}
R. V. Bessonov and P. V. Gubkin, \textit{Direct and inverse spectral
continuity for Dirac operators}, Geom. Funct. Anal. \textbf{36}
(2026), 351--411.
\href{https://doi.org/10.1007/s00039-026-00735-3}{doi:10.1007/s00039-026-00735-3}.

\bibitem{Conway}
J. B. Conway, \textit{Functions of One Complex Variable I},
second edition, Graduate Texts in Mathematics, vol. 11, Springer,
1978. \href{https://doi.org/10.1007/978-1-4612-6313-5}{doi:10.1007/978-1-4612-6313-5}.

\bibitem{DeiftZhou1993}
P. Deift and X. Zhou, \textit{A steepest descent method for
oscillatory Riemann--Hilbert problems. Asymptotics for the MKdV
equation}, Ann. of Math. (2) \textbf{137} (1993), 295--368.
\href{https://doi.org/10.2307/2946540}{doi:10.2307/2946540}.

\bibitem{DeiftZhou2003}
P. Deift and X. Zhou, \textit{Long-time asymptotics for solutions of
the NLS equation with initial data in a weighted Sobolev space},
Comm. Pure Appl. Math. \textbf{56} (2003), 1029--1077.
\href{https://doi.org/10.1002/cpa.3034}{doi:10.1002/cpa.3034}.

\bibitem{Denisov2006}
S. A. Denisov, \textit{Continuous analogs of polynomials orthogonal
on the unit circle and Krein systems}, Int. Math. Res. Surv.
\textbf{2006} (2006), Art. ID 54517, 148 pp.
\href{https://doi.org/10.1155/IMRS/2006/54517}{doi:10.1155/IMRS/2006/54517}.

\bibitem{DenisovNCC}
S. A. Denisov, \textit{The strong version of nonlinear Carleson
conjecture fails}, preprint (2026),
\href{https://arxiv.org/abs/2605.01658}{arXiv:2605.01658}.

\bibitem{DenisovPointwise}
S. A. Denisov, \textit{Pointwise behavior of $SU(1,1)$ nonlinear
Fourier transform}, preprint (2026),
\href{https://arxiv.org/abs/2605.25108v2}{arXiv:2605.25108v2}.

\bibitem{DiengMcLaughlin}
M. Dieng and K. D. T.-R. McLaughlin, \textit{Long-time asymptotics
for the NLS equation via $\bar\partial$ methods}, preprint (2008),
\href{https://arxiv.org/abs/0805.2807}{arXiv:0805.2807}.

\bibitem{FaddeevTakhtajan}
L. D. Faddeev and L. A. Takhtajan, \textit{Hamiltonian Methods in
the Theory of Solitons}, Classics in Mathematics, Springer, 2007.
\href{https://doi.org/10.1007/978-3-540-69969-9}{doi:10.1007/978-3-540-69969-9}.

\bibitem{Garnett}
J. B. Garnett, \textit{Bounded Analytic Functions}, Pure and Applied
Mathematics, vol. 96, Academic Press, 1981.

\bibitem{GeorgievOzawa}
V. Georgiev and T. Ozawa, \textit{On completeness of modified wave
operators for defocusing NLS}, preprint (2025),
\href{https://arxiv.org/abs/2509.15921v2}{arXiv:2509.15921v2}.

\bibitem{HayashiNaumkin1998}
N. Hayashi and P. I. Naumkin, \textit{Asymptotics for large time of
solutions to the nonlinear Schr\"odinger and Hartree equations},
Amer. J. Math. \textbf{120} (1998), 369--389.
\href{https://doi.org/10.1353/ajm.1998.0011}{doi:10.1353/ajm.1998.0011}.

\bibitem{HayashiNaumkin2006}
N. Hayashi and P. I. Naumkin, \textit{Domain and range of the
modified wave operator for Schr\"odinger equations with a critical
nonlinearity}, Comm. Math. Phys. \textbf{267} (2006), 477--492.
\href{https://doi.org/10.1007/s00220-006-0057-6}{doi:10.1007/s00220-006-0057-6}.

\bibitem{IfrimTataru2015}
M. Ifrim and D. Tataru, \textit{Global bounds for the cubic
nonlinear Schr\"odinger equation (NLS) in one space dimension},
Nonlinearity \textbf{28} (2015), 2661--2675.
\href{https://doi.org/10.1088/0951-7715/28/8/2661}{doi:10.1088/0951-7715/28/8/2661}.

\bibitem{IfrimTataru2023}
M. Ifrim and D. Tataru, \textit{Global solutions for 1D cubic
defocusing dispersive equations: Part I}, Forum Math. Pi
\textbf{11} (2023), e31, 46 pp.
\href{https://doi.org/10.1017/fmp.2023.30}{doi:10.1017/fmp.2023.30}.

\bibitem{KatoPusateri}
J. Kato and F. Pusateri, \textit{A new proof of long-range scattering
for critical nonlinear Schr\"odinger equations}, Differential
Integral Equations \textbf{24} (2011), 923--940.
\href{https://arxiv.org/abs/1004.0721}{arXiv:1004.0721}.

\bibitem{KawamotoMizutani}
M. Kawamoto and H. Mizutani, \textit{Modified wave operators for the
defocusing cubic nonlinear Schr\"odinger equation in one space
dimension with large scattering data}, preprint (2026),
\href{https://arxiv.org/abs/2506.01871v3}{arXiv:2506.01871v3}.

\bibitem{KeelTao1998}
M. Keel and T. Tao, \textit{Endpoint Strichartz estimates},
Amer. J. Math. \textbf{120} (1998), 955--980.
\href{https://doi.org/10.1353/ajm.1998.0039}{doi:10.1353/ajm.1998.0039}.

\bibitem{KillipVisanZhang}
R. Killip, M. Vi\c{s}an and X. Zhang, \textit{Low regularity
conservation laws for integrable PDE}, Geom. Funct. Anal.
\textbf{28} (2018), 1062--1090.
\href{https://doi.org/10.1007/s00039-018-0444-0}{doi:10.1007/s00039-018-0444-0}.

\bibitem{LindbladSoffer}
H. Lindblad and A. Soffer, \textit{Scattering and small data
completeness for the critical nonlinear Schr\"odinger equation},
Nonlinearity \textbf{19} (2006), 345--353.
\href{https://doi.org/10.1088/0951-7715/19/2/006}{doi:10.1088/0951-7715/19/2/006}.

\bibitem{DLMF}
F. W. J. Olver et al. (eds.), \textit{NIST Digital Library of
Mathematical Functions}, Chapter 12: Parabolic cylinder functions,
\url{https://dlmf.nist.gov/12}.

\bibitem{Ozawa1991}
T. Ozawa, \textit{Long range scattering for nonlinear
Schr\"odinger equations in one space dimension}, Comm. Math.
Phys. \textbf{139} (1991), 479--493.
\href{https://doi.org/10.1007/BF02101876}{doi:10.1007/BF02101876}.

\bibitem{Poltoratski2024}
A. Poltoratski, \textit{Pointwise convergence of the non-linear
Fourier transform}, Ann. of Math. (2) \textbf{199} (2024), 741--793.
\href{https://doi.org/10.4007/annals.2024.199.2.4}{doi:10.4007/annals.2024.199.2.4}.

\bibitem{SteinSingular}
E. M. Stein, \textit{Singular Integrals and Differentiability
Properties of Functions}, Princeton Mathematical Series, vol. 30,
Princeton University Press, 1970.

\bibitem{SylvesterWinebrenner}
J. Sylvester and D. P. Winebrenner, \textit{Linear and nonlinear
inverse scattering}, SIAM J. Appl. Math. \textbf{59} (1999), 669--699.
\href{https://doi.org/10.1137/S0036139997319773}{doi:10.1137/S0036139997319773}.

\bibitem{Tsutsumi1987}
Y. Tsutsumi, \textit{$L^2$-solutions for nonlinear Schr\"odinger
equations and nonlinear groups}, Funkcial. Ekvac. \textbf{30}
(1987), 115--125.

\bibitem{ZakharovManakov}
V. E. Zakharov and S. V. Manakov, \textit{Asymptotic behavior of
non-linear wave systems integrated by the inverse scattering
method}, Sov. Phys. JETP \textbf{44} (1976), 106--112.

\bibitem{Zhou1998}
X. Zhou, \textit{$L^2$-Sobolev space bijectivity of the scattering
and inverse-scattering transforms}, Comm. Pure Appl. Math.
\textbf{51} (1998), 697--731.
\href{https://doi.org/10.1002/(SICI)1097-0312(199807)51:7<697::AID-CPA1>3.0.CO;2-1}{doi:10.1002/(SICI)1097-0312(199807)51:7\textless697::AID-CPA1\textgreater3.0.CO;2-1}.
\end{thebibliography}
\end{document}